\documentclass[11pt,english,oneside,a4paper]{amsart}

\usepackage[english]{babel}
\usepackage[utf8]{inputenc}
\usepackage[T1]{fontenc}
\usepackage[bottom=3cm,a4paper]{geometry}
\usepackage[hidelinks]{hyperref}
\usepackage{amsmath,amsfonts,amssymb,mathtools,amsthm}
\usepackage{listings,verbatim,fancyvrb,xspace,enumerate}
\usepackage{xcolor,graphicx,ytableau,tikz}
\usepackage[all, cmtip]{xy}
\usepackage{fancyhdr}
\usepackage[toc,page]{appendix}
\usepackage{pdfpages}

\newtheorem{theorem}{Theorem}[subsection]
\newtheorem*{theorem*}{Theorem}

\newtheorem{corollary}[theorem]{Corollary}
\newtheorem*{corollary*}{Corollary}
\newtheorem{lemma}[theorem]{Lemma}
\newtheorem{proposition}[theorem]{Proposition}
\theoremstyle{definition}
\newtheorem{example}[theorem]{Example}
\newtheorem*{example*}{Example}
\newtheorem*{examples*}{Examples}
\newtheorem{remark}[theorem]{Remark}
\newtheorem{definition}[theorem]{Definition}

\newtheorem*{ack}{Acknowledgements}
\newtheorem*{con}{Conventions}
\newsavebox{\eqbox}
\newenvironment{longequation*} {\begin{lrbox}{\eqbox}$} {$\end{lrbox}\begin{equation*}\resizebox{\linewidth}{!}{\ensuremath{\displaystyle\usebox{\eqbox}}}\end{equation*}}

\makeatletter
\newcommand*\bigcdot{\mathpalette\bigcdot@{.5}}
\newcommand*\bigcdot@[2]{\mathbin{\vcenter{\hbox{\scalebox{#2}{\(\m@th#1\bullet\)}}}}}

\makeatother

\def\bydef{\coloneqq}

\DeclarePairedDelimiter{\abs}{\lvert}{\rvert}
\DeclarePairedDelimiter{\Set}{\lbrace}{\rbrace}
\newcommand{\suchthat}{\hskip.5ex\color{black!66}\middle/\color{black}\hskip1ex}

\DeclareMathOperator{\diff}{d}

\DeclareMathOperator{\Flag}{Flag}
\DeclareMathOperator{\Grass}{Grass}
\DeclareMathOperator{\rank}{rank}

\DeclareMathOperator{\codim}{codim}

\DeclareMathOperator{\Image}{Im}

\DeclareMathOperator{\pr}{pr}
\DeclareMathOperator{\Id}{Id}

\DeclareMathOperator{\Ker}{Ker}
\DeclareMathOperator{\Coker}{Coker}
\DeclareMathOperator{\GL}{GL}

\let\S\relax 
\DeclareMathOperator{\S}{S}
\DeclareMathOperator{\Ext}{\mathchoice
{\scalebox{1.2}{$\mathsf{\Lambda}$}}
{\scalebox{1.1}{$\mathsf{\Lambda}$}}
{\scalebox{0.8}{$\mathsf{\Lambda}$}}
{\scalebox{.65}{$\mathsf{\Lambda}$}}
}

\newcommand{\C}{\mathbb{C}}
\renewcommand{\L}{\mathcal{L}} 
\renewcommand{\O}{\mathcal{O}} 
\renewcommand{\P}{\mathbf{P}} 

\newcommand{\N}{\mathbb{N}}
\newcommand{\Z}{\mathbb{Z}}

\newcommand{\kk}{\mathbf{k}}
\let\mathbi\boldsymbol

\newcommand{\E}{\mathcal{E}}

\author{Antoine Etesse}
\email{antoine.etesse@math.univ-toulouse.fr}
\address{Institut Mathématique de Toulouse (IMT), Université Paul Sabatier}
\title{Cohomology of twisted symmetric powers of cotangent bundles of smooth complete intersections}

\begin{document}
\sloppy

\begin{abstract}
In this paper, we provide two different resolutions of structural sheaves of projectivized tangent bundles of smooth complete intersections. These resolutions allow in particular to obtain convenient (and completely explicit) descriptions of cohomology of twisted symmetric powers of cotangent bundles of complete intersections, which are easily implemented on computers. We then provide several applications. First, we recover the known vanishing theorems on the subject, and show that they are optimal via some non-vanishing theorems. Then, we study the symmetric algebra \(\bigoplus_{m \in \N} H^{0}(X, S^{m}\Omega_{X}(m))\) of a smooth complete intersection of codimension \(c < \frac{N}{2}\), improving the known results in the literature. We also study partial ampleness of cotangent bundles of general hypersurfaces. Finally, we illustrate how the explicit descriptions of cohomology can be implemented on computer. In particular, this allows to exhibit new and simple examples of family of surfaces along which the canonically twisted pluri-genera do not remain constant.
\end{abstract}
\maketitle

\tableofcontents

\section*{Introduction.}
The main goal of this paper is to provide a convenient and explicit description of cohomology groups of (negatively) twisted symmetric powers of cotangent bundles of complete intersections, and provide several applications. Some of them are, to our knowledge, new and some others consist of alternative proofs of known results. 

The study of this subject is not new, but there seems to be, as far as we know, surprisingly few references. In \cite{Bruck}, the author proves a vanishing theorem for global sections of tensor powers of cotangent bundles of hypersurfaces. In \cite{Sakai}, the author proves that a smooth complete intersection in \(\P^{N}\) of  codimension \(c<\frac{N}{2}\) has no global symmetric differentials. (Note that in both previous references, twists by the Serre line bundle are not considered).
In \cite{BR}, the authors consider more generally Schur powers of cotangent bundles of complete intersections, and prove strong vanishing results (see \cite{BR}[Theorem 4 (i), (ii), (iii)]). 
In our first application (see Section \ref{sect: vanish and non-vanish}), we provide new proofs of these vanishing theorems, and show that they are optimal via non-vanishing theorems. These non-vanishing results are,  to our knowledge, new, and allowed us to generalize results of Bogomolov--De Oliveira in \cite{Bog}: see Theorem \ref{cor: generalization intro} below (or Theorem \ref{cor: generalization}).

Let us continue our overview. In \cite{Sch}, the author proves that vanishing theorems \cite{BR}[Theorem 4 (i), (ii) for symmetric powers] are actually true for any smooth subvarieties (within the codimension range allowed by the hypothesis of the theorem). Schneider's result was then generalized by Brotbek in \cite{Brotdiff}, allowing to handle the case of a general Schur power. It remains, to our knowledge, unknown whether \cite{BR}[Theorem 4 (iii) for symmetric powers] remains valid for any smooth subvariety in \(\P^{N}\) (of codimension \(c < \frac{N}{2}\)). Note that, by Hartshorne's conjecture \cite{Laz1}[Conjecture 3.2.8], this should be the case at least when \(c < \frac{N}{3}\).
In \cite{Bog}, the authors studied more specifically symmetric powers of \(\Omega_{X}(1)\), where \(X \subset \P^{N}\) is a smooth subvariety. Amongst other things, they recovered \cite{BR}[Theorem (iii) for symmetric powers] in the case of hypersurfaces (see Theorem B in \textsl{loc.cit}), and they showed that a very particular case of \cite{BR}[Theorem (iii) for symmetric powers] holds for smooth subvarieties of codimension \(2\): see Section \ref{sect: vanish and non vanish 1} for more details. Later, in \cite{DeOl}, the study of the algebra \(\bigoplus_{m \in \N} H^{0}(X,S^{m}\Omega_{X}(m))\) was pushed further. In \textsl{loc.cit}, the authors proved in particular that, for a smooth complete intersection \(X\) of codimension \(c < \frac{(N+2)}{3}\), there is an isomorphism of algebra
\[
\bigoplus_{m \in \N} H^{0}(X,S^{m}\Omega_{X}(m))
\simeq
\bigoplus_{m \in \N} S^{m}(H^{0}(X, \mathcal{I}_{X}(2)),
\]
where \(\mathcal{I}_{X}\) is the ideal sheaf of the smooth complete intersection \(X\).
We improved their result by showing that the statement holds more generaly in codimension \(c<\frac{N}{2}\) (via a completely different method): see Theorem \ref{thm: sym alg intro} below or Theorem \ref{thm: sym alg}.

In \cite{Deb}, the author studied cohomology of twisted symmetric powers of cotangent bundles of complete intersections in abelian varieties. Most notably, he showed that a sufficiently general and sufficiently ample complete intersection in an abelian variety, with codimension at least as large as its dimension, has ample cotangent bundle. By analogy, he suggested that the same statement should hold for projective spaces: for a bit more than a decade, this was known as \textsl{Debarre's conjecture on ampleness}. This conjecture attracted a lot of attention, and the strategy to tackle it was initiated by Brotbek in \cite{Brot15}. In \textsl{loc.cit}, the author provides an explicit description of cohomology of (negatively) twisted symmetric powers of cotangent bundles of smooth complete intersections (see \cite{Brot15}[Theorem A], and Section 2.5 in \textsl{loc.cit} for details). The description provided is rather technical. It is probably worth specifying that the descriptions we provide in Theorem \ref{thm: coho ci 1} or Theorem \ref{thm: coho ci 2} are different, and, at least it seems to us, more concise and less technical. Using this description, Brotbek was able to prove particular cases of Debarre's conjecture (see \cite{Brot15}[Theorem D]). A bit later, Brotbek's strategy was further developed by Brotbek himself and Darondeau in \cite{BD}, and they proved in \textsl{loc.cit} Debarre's conjecture in full generality. Around the same period, with a similar but more technical approach, Xie provided a different proof of Debarre's conjecture in \cite{Xie}. 

After this introductory overview on the literature on symmetric powers of cotangent bundles, let us now detail the content of the present paper.
\newline

In Section \ref{sect: coho proj}, we recall the study of twisted symmetric powers of cotangent bundles (and tangent bundles) of projective spaces, which allows to compute their cohomology. Whereas this is very classical, we provide here a very explicit description: it may very well be possible that our approach has already been described elsewhere, but if so, we were not aware of it. 

The starting point is the following generalization of the so-called \textsl{Euler exact sequence} (see Section \ref{subs: gen Euler seq}):
\begin{theorem}[Generalized Euler exact sequence]
\label{thm: gen Euler intro}
Let \(N \in \N_{\geq 1}\). For any \(m\in \N_{\geq 1}\), and any \(n \in \Z\), one has the following short exact sequence:
\begin{equation*}
\xymatrix{
0 \ar[r] 
& 
S^{m}\Omega_{\P^{N}}(m+n) 
\ar[r] 
& 
\C[Y]_{m}\otimes\O_{\P^{N}} (n)
\ar[r]^-{\delta} 
&
 \C[Y]_{m-1}\otimes\O_{\P^{N}}(n+1)  
 \ar[r] 
&
0,
}
\end{equation*}
where \(\delta\bydef\sum\limits_{i=0}^{N} X_{i} \frac{\partial}{\partial Y_{i}}\).
\end{theorem}
This result allows us to interpret local twisted symmetric differentials of the projective space \(\P^{N}\) (i.e. local sections of twisted symmetric powers of \(\Omega_{\P^{N}}\)) as solutions in \(\C(X)[Y]\) of the partial differential equation 
\[
\delta
\bydef
\sum\limits_{i=0}^{N} X_{i} \frac{\partial}{\partial Y_{i}},
\] 
 where \(X\bydef (X_{0}, \dotsc, X_{N})\) and \(Y\bydef (Y_{0}, \dotsc, Y_{N})\). This is the point of view that we are going to adopt throughout the paper. 

The geometric interpretation is quite straightforward once one interprets the projectivized\footnote{Throughout the paper,   the projectivization \(\P(E)\) of a vector bundle \(E\) on a variety means the projectivization of lines.} tangent bundle \(\P(T\P^{N})\) as the flag variety \(\Flag_{(1,2)} \C^{N+1}\) of lines included in planes in \(\C^{N+1}\): see Section \ref{subs: geom inter}.

The generalized Euler exact sequence allows to compute the cohomology of \(\Ker \delta\) (or equivalently, twisted symmetric powers of \(\Omega_{\P^{N}}\)). The key fact in the computations (at least for the \(H^{0}\) and the \(H^{1}\)) is that the partial differential equation \(\delta\) is equivariant with respect to the natural action of \(\GL_{N+1}\C\) on \(\C[Y,X]\) (which allows to use representation theory). Using rather the geometric interpretation \(\P(T\P^{N}) \simeq \Flag_{(1,2)}\C^{N+1}\), the cohomology of \(\Ker \delta\) follows immediately from Bott's formulas (see Appendix \ref{appendix: Bott}). These two approaches are detailed in Section \ref{subs: coho proj space}.

By dualizing the generalized Euler exact sequence, i.e. by applying the functor \(\mathcal{H}om(\cdot, \O_{\P^{N}})\), one also gets a description of twisted symmetric powers of the tangent bundle:
\begin{equation*}
\xymatrix{
0 \ar[r] 
& 
\C[Y]_{m-1}\otimes\O_{\P^{N}}(n-1)
\ar[r]^-{\delta^{*}} 
&
 \C[Y]_{m}\otimes\O_{\P^{N}}(n)
 \ar[r] 
&
S^{m}T\P^{N}(n-m) 
\ar[r] 
& 
0.
}
\end{equation*}
Here, \(m\) is a positive natural number, \(n\) an integer, and one uses the natural isomorphism \(\mathcal{H}om(\O_{\P^{N}}(i), \O_{\P^{N}})\simeq \O_{\P^{N}}(-i)\). It turns out that, doing so, the dual map \(\delta^{*}\) is quite unaesthetic. In order to obtain a convenient description, one has to renormalize the vector space \(\C[Y]\) by the following map:
\[
u \colon\left(
\begin{array}{ccc}
  \C[Y] & \longrightarrow   &  \C[Y]
  \\
  Y^{\mathbi{\alpha}} & \longmapsto  &  \frac{Y^{\alpha}}{\mathbi{\alpha}!}
  \end{array}
\right).
\]
After this renormalization, the modified map (denoted \(\delta_{*}\)) becomes simply the multiplication by the quadratic polynomial
\[
q \bydef \sum\limits_{i=0}^{N} X_{i}Y_{i}.
\]
We refer to Section \ref{subs: dual Euler} for more details. We chose to insist on this renormalization map in the introduction because it will be used repeatedly throughout the paper. This might seem somewhat anecdotal, but it will turn out to be crucial to obtain neat descriptions (of cohomology of twisted symmetric powers of cotangent bundles of complete intersections).

In Section \ref{subs: notations}, we set some natural notations, important for the whole paper.

We conclude Section \ref{sect: coho proj} with Section \ref{subs: a particular class}, in which we study partial differential equations of the form
\[
\sum\limits_{i=0}^{N} P_{i} \frac{\partial}{\partial Y_{i}},
\]
where \(P_{0}, \dotsc, P_{N}\) are homogeneous polynomials of same degree \(d \geq 1\), sharing only the origin as common zero. More precisely, we want to understand the set of bi-homogeneous polynomials satisfying such a partial differential equation. The main result of this section is that, in order to understand this set, it is enough to:
\begin{enumerate}
\item{} understand the set of polynomials satisfying the differential equation \(\delta\);
\item{} understand the push-forward sheaf
\(
f_{*}\O_{\P^{N}},
\)
where 
\[
f\bydef [P_{0}: \dotsb : P_{N}]
\]
is the finite map induced by the polynomials \(P_{0}, \dotsc, P_{N}\).
\end{enumerate}
The first item is dealt with (as an outcome of the computation of the cohomology of \(\Ker \delta\)), and the real difficulty lies in the second item. It is known that \(f_{*}\O_{\P^{N}}\) splits as a direct sum of line bundles (it follows from a result of Horrocks, see e.g. \cite{Beauville}[Section 1] for details), but the question of identifying the factors appearing in the decomposition seems very difficult in general. Let us state the main result of Section \ref{subs: a particular class}:
\begin{theorem}
\label{thm: partial intro}
The partial differential equation 
\[
\sum\limits_{i=0}^{N} P_{i} \frac{\partial}{\partial Y_{i}}
\]
admits a bi-homogeneous polynomial \(A \in \C[Y,X]_{m,n}\) as a (non-trivial) solution if and only if \(dm \leq n\).
\end{theorem}
\
\newline

In Section \ref{sect: resolutions}, we provide two different locally free resolutions of the structural sheaf \(\O_{\P(TX)}\) of the projectivized tangent bundle \(\P(TX) \rightarrow X\) of a smooth complete intersection \(X \subset \P^{N}\), and we describe their push-forward on the base \(X\) (more precisely, we describe the push-forward of various twists of these resolutions).

For both resolutions, we start with the case of smooth hypersurfaces (see Sections \ref{section: Koszul complex 1: hyp} and \ref{section: Koszul complex 2: hyp}), and then move onto the general case of smooth complete intersections (see Sections \ref{sect: Koszul complex 1: ci} and \ref{sect: Koszul complex 2: ci}). The reason for this presentation is the following. The first (resp. the second) resolution is given by a Koszul complex associated to sections of rank \(1\)(resp. rank \(2\)) vector bundles on \(\P(T\P^{N}_{\vert X})\)(resp. on \(\P(T\P^{N})\)). Therefore, we take the time in Sections \ref{section: Koszul complex 1: hyp} and \ref{section: Koszul complex 2: hyp} to do properly the construction for a single section, which is easy in the first case, and less straightforward in the second case. 

The push-forward on the base of the previous twisted resolutions provide locally free resolutions of twisted symmetric powers of the cotangent bundle \(\Omega_{X}\) of the complete intersection \(X\). It is thus important to obtain an explicit description. For the first resolution provided, it turns out to be easy (see also Theorem \ref{thm: resolution base 1}):
\begin{theorem}
\label{thm: complex 1 intro}
Let \(X \bydef \Set{P_{1}= \dotsb = P_{c}=0} \subset \P^{N}\) be a smooth complete intersection of codimension \(c\) and multi-degree \(\mathbi{d}=(d_{1}, \dotsc, d_{c})\). Denote by \(\alpha(P_{i})\) the multiplication map by the bi-homogeneous polynomial \(\frac{1}{d_{i}}(\diff P_{i})_{X}(Y) \in \C[Y,X]\).
The Koszul complex
\[
\mathcal{K}\big(\alpha(P_{1}), \dotsc, \alpha(P_{c})\big)
\] 
induced by the multiplication maps \(\big(\alpha(P_{i})\big)_{1 \leq i \leq c}\) on \((\Ker \delta)_{\vert X}\)
provides a locally free bi-graded\footnote{The bi-graduations on \(\bigoplus_{m \geq c, n \in \Z} S^{m}\Omega_{X}(m+n)\) and on \(\Ker \delta\) are the natural ones.} resolution of
\[
\bigoplus_{m \geq c, n \in \Z} S^{m}\Omega_{X}(m+n).
\]
\end{theorem}
However, for the second resolution, the interpretation of the push-forward twisted resolution is much more complicated, and we choose to provide an explicit description only in codimension \(1\) and \(2\) (see also Theorem \ref{thm: resol hyp} and Theorem \ref{thm: resol ci}):

\begin{theorem}
\label{thm: complex 2 intro cod 1}
Let \(H\bydef \Set{P=0} \subset \P^{N}\) be a smooth hypersurface of degree \(d \geq 1\).
Denote by \(\beta(P)\) the following map
\[
\beta(P)(\cdot)
\bydef 
\cdot \times P - \alpha(P) \circ \delta (\cdot),
\]
where one recalls that \(\alpha(P)\) is the multiplication map by \(\frac{1}{d} (\diff P)_{X}(Y)\).

The complex 
\begin{equation*}
\xymatrix{
0 
\ar[r]
&
\Ker \delta[-1,-2d+1]
\ar[r]^-{\alpha(P)}
&
\Ker \delta^{2}[0,-d]
\ar[r]^-{\beta(P)}
&
\Ker \delta
}
\end{equation*}
provides a locally free bi-graded resolution of
 \[
 \bigoplus_{m \geq 1, n \in \Z} S^{m}\Omega_{H}(m+n).
 \]
\end{theorem}

\begin{theorem}
\label{thm: complex 2 intro cod 2}
Let \(X\bydef \Set{P_{1}=P_{2}=0}\subset \P^{N}\) be a smooth complete intersection of codimension \(2\) and multi-degree \(\mathbi{d}=(d_{1}, d_{2})\).
The complex
\[
\xymatrix{
0 
\ar[r]
&
\Ker \delta[-2,-2\abs{\mathbi{d}}+2]
\ar@{-}[r]^-{(f_{1i})_{1 \leq i \leq 2}}
&
\\
\ar[r]
&
\Ker \delta^{2}[-1,-2\abs{\mathbi{d}}+d_{2}+1]
\oplus 
\Ker \delta^{2}[-1,-2\abs{\mathbi{d}}+d_{1}+1]
\ar@{-}[r]^-{(f_{2i})_{1 \leq i \leq 4}}
&
\\
\ar[r]
&
(\underset{i}{\oplus}\Ker \delta[-1,-2d_{i}+1])
 \oplus 
 \Ker \delta^{3}[0,-\abs{\mathbi{d}}]
 \oplus
  \Ker \delta[-1,-\abs{\mathbi{d}}+1]
  \ar@{-}[r]
  &
\\
\ar[r]^-{(f_{3i})_{1 \leq i \leq 2}}
&
\Ker \delta^{2}[0, -d_{1}]
\oplus 
\Ker \delta^{2}[0, -d_{2}]
\ar[r]^-{f_{41}}
&
\Ker \delta,
}
\]
where
\begin{equation*}
\left\{
\begin{array}{ll}
f_{1i}=\alpha(P_{i}), 1 \leq i \leq 2;
\\
f_{21}(A,B)=\beta(P_{2})(B) \ \text{and} \ f_{22}(A,B)=\beta(P_{1})(A);
\\
f_{23}(A,B)=\alpha(P_{2})(A)-\alpha(P_{1})(B);
\\
f_{24}(A,B)=\beta(P_{2})(A)+\beta(P_{1})(B);
\\
f_{31}(A,B,C,D)=\alpha(P_{1})(A) + \frac{1}{2}\big(\beta(P_{2})(C)+P_{2}C - \alpha(P_{2})(D)\big);
\\
f_{32}(A,B,C,D)=\alpha(P_{2})(B) - \frac{1}{2}\big(\beta(P_{1})(C) +P_{1}C + \alpha(P_{1})(D)\big);
\\
f_{4}(A,B)=\beta(P_{1})(A) + \beta(P_{2})(B).
\end{array}
\right.
\end{equation*}
provides a locally free bi-graded resolution of 
\[
\bigoplus_{m \geq 2, n \in \Z}
S^{m}\Omega_{X}(m+n).
\]
\end{theorem}
At this point, one may wonder why bother using such a complicated complex, whereas the complex given in Theorem \ref{thm: complex 1 intro} (in the codimension \(2\) case) seems much easier to apprehend. We must admit that we were not able \textsl{yet} to take advantage of it (it is not used in the applications, only the resolution upstairs from which this resolution comes from is used). However, we do believe that this resolution is better than the one in Theorem \ref{thm: complex 1 intro}. We discuss this further below, but in a few words the reason is roughly the following. In the complex of Theorem \ref{thm: complex 1 intro}, both the objects and the arrows depend on the complete intersection, whereas in the complex of Theorem \ref{thm: complex 2 intro cod 1} or Theorem \ref{thm: complex 2 intro cod 2}, the dependency on \(X\) only lies in the arrows. We believe that this should somehow make the study more tractable.
\newline

In Section \ref{sect: coho sym power ci}, we use the resolutions provided in Theorem \ref{thm: complex 1 intro}, \ref{thm: complex 2 intro cod 1} and \ref{thm: complex 2 intro cod 2} to exhibit complexes computing the cohomology of (negatively) twisted symmetric powers of cotangent bundles of smooth complete intersections: see respectively Section \ref{sect: coho ci 1}, \ref{sect: coho hyp 2} and \ref{sect: coho ci 2}.
 
 The basic observation is that suitable bi-graded parts of the above complexes have their cohomology supported in maximal degree. Therefore, by a simple exercice in cohomological algebra, one can compute the cohomology by applying the functor \(H^{top}\) to the (suitable graded parts of the) complexes. However, for computational purposes, it is not practical to work with \(H^{top}\). Therefore, we use Serre duality to provide a reformulation of the complex. In order to obtain a neat description, it is crucial to use the renormalization map defined earlier in the introduction: see Section \ref{sect: coho hyp 1} for details (the ideas are the same in the subsequent Sections \ref{sect: coho ci 1}, \ref{sect: coho hyp 2} and \ref{sect: coho ci 2}). 
 
Using the resolution given in Theorem \ref{thm: complex 1 intro}, we obtain the following (see Theorem \ref{thm: coho ci 1}):
\begin{theorem}
\label{thm: coho ci 1 intro}
Let \(X \bydef \Set{P_{1}= \dotsb = P_{c}=0} \subset \P^{N}\) be a smooth complete intersection of codimension \(c\) and multi-degree \(\mathbi{d}=(d_{1}, \dotsc, d_{c})\). Denote by \(\alpha^{*}(P_{i})\) the following partial differential equation:
\[
 \alpha^{*}(P_{i})(\cdot) \bydef \frac{1}{d_{i}}(\sum\limits_{i=0}^{N} \frac{\partial P}{\partial X_{i}} \frac{\partial}{\partial Y_{i}})(\cdot),
\]
which induces a map of bi-graded algebra \(\frac{\C[Y,X]}{(P_{1}, \dotsc, P_{c}, q)}[-1,-d_{i}+1] \longrightarrow \frac{\C[Y,X]}{(P_{1}, \dotsc, P_{c}, q)}\).

Then, the cohomology of 
\[
\bigoplus_{m \geq c,  n > 1} S^{m}\Omega_{X}(m-n)
\]
can be computed via the Koszul complex on \(\frac{\C[Y,X]}{(P_{1}, \dotsc, P_{c}, q)}\):
\[
K(\alpha^{*}(P_{1}), \dotsc, \alpha^{*}(P_{c}))[-c, -(N+1)+2\abs{\mathbi{d}}-c].
\]
Namely, the \(i\)th cohomology group of one graded component of 
\(
\bigoplus_{m \geq 1, n > 1 }
S^{m}\Omega_{X}(m-n)
\) 
is isomorphic to the \((N-c-i)\)th cohomology group of the corresponding graded part of the Koszul complex.

\end{theorem}

Using the resolutions given in Theorem \ref{thm: complex 2 intro cod 1} and \ref{thm: complex 2 intro cod 2}, we obtain respectively (see Theorem \ref{thm: coho hyp 2} and \ref{thm: coho ci 2}):
\begin{theorem}
\label{thm: coho hyp 2 intro}
Let \(H\bydef \Set{P=0} \subset \P^{N}\) be a smooth hypersurface of degree \(d \geq 1\).
Denote by \(\beta^{*}(P)\) the following map
\[
\beta^{*}(P)(\cdot)
\bydef 
\cdot \times P - q\times \alpha^{*}(P)(\cdot),
\]
where one recalls that \(\alpha^{*}(P)\) is the partial differential equation
 \(
 \frac{1}{d}(\sum\limits_{i=0}^{N} \frac{\partial P}{\partial X_{i}} \frac{\partial}{\partial Y_{i}})(\cdot)
 \).

Then, the cohomology of 
\(
\bigoplus_{m \geq 1, n > 1 }
S^{m}\Omega_{H}(m-n)
\)
can be computed via the graded complex
\[
\xymatrix{
\frac{S}{(q)}[0, -(N+1)]
\ar[rr]^-{\beta^{*}(P)}
&
&
\frac{S}{(q^{2})}[0,-(N+1)+d]
\ar[rr]^-{\alpha^{*}(P)}
&
&
\frac{S}{q}[-1, -(N+1)+2d-1],
}
\]
where \(S\bydef \C[Y,X]\).

Namely, the \(i\)th cohomology group of one graded component of 
\(
\bigoplus_{m \geq 1, n > 1 }
S^{m}\Omega_{H}(m-n)
\) 
is isomorphic to the \((N-i)\)th cohomology group of the corresponding graded part of the graded complex.

\end{theorem}

\begin{theorem}
\label{thm: coho ci 2 intro}
Let \(X\bydef \Set{P_{1}=P_{2}=0}\subset \P^{N}\) be a smooth complete intersection of codimension \(2\) and multi-degree \(\mathbi{d}=(d_{1}, d_{2})\).
The cohomology of 
\[
\bigoplus_{m \geq 2, n > 1 }
S^{m}\Omega_{X}(m-n)
\]
can be computed via the following bi-graded complex
\begin{equation*}
\xymatrix{
\Big(\frac{S}{(q)}
\ar[r]^-{(g_{1i})_{1 \leq i \leq 2}}
&
\frac{S}{(q^{2})}[0, d_{1}]\oplus \frac{S}{(q^{2})}[0, d_{2}]
\ar@{-}[r]^-{(g_{2i})_{1 \leq i \leq 4}}
&
\\
\ar[r]
&
(\oplus_{i}
\frac{S}{(q)}[-1, 2d_{i}-1])
\oplus 
 \frac{S}{(q^{3})}[0, -\abs{\mathbi{d}}]
\oplus \frac{S}{(q)}[-1, \abs{\mathbi{d}}-1]
\ar@{-}[r]
&
\\
\ar[r]^-{(g_{3i})_{1 \leq i \leq 2}}
&
\frac{S}{(q^{2})}[-1, +2\abs{\mathbi{d}}-d_{2}-1]
\oplus \frac{S}{(q^{2})}[-1, 2\abs{\mathbi{d}}+d_{1}-1]
\ar@{-}[r]
&
\\
\ar[r]^-{g_{4}}
&
\frac{S}{(q)}[-2, +2\abs{\mathbi{d}}-2]\Big)[0,-(N+1)],
}
\end{equation*}
where \(S\bydef \C[Y,X]\), and where one has:
\begin{equation*}
\left\{
\begin{array}{ll}
g_{11}=\beta^{*}(P_{1}) \ \text{and} \ g_{12}=\beta^{*}(P_{2});
\\
g_{21}(A,B)=\alpha^{*}(P_{1})(A) \ \text{and} \ g_{22}(A,B)=\alpha^{*}(P_{2})(B);
\\
g_{23}(A,B)=\frac{1}{2}\big(\beta^{*}(P_{2})(A)+P_{2}A-\beta^{*}(P_{1})(B)-P_{1}B\big);
\\
g_{24}=-\frac{1}{2}\big(\alpha^{*}(P_{2})(A)+\alpha^{*}(P_{1})(B)\big);
\\
g_{31}(A,B,C,D)=\beta^{*}(P_{1})(B)+\alpha^{*}(P_{2})(C)+\beta^{*}(P_{2})(D);
\\
g_{32}(A,B,C,D)=\beta^{*}(P_{2})(A)-\alpha^{*}(P_{1})(C)-\beta^{*}(P_{1})(D);
\\
g_{4}(A,B)=\alpha^{*}(P_{1})(A)+\alpha^{*}(P_{2})(B).
\end{array}
\right.
\end{equation*}
Namely, the \(i\)th cohomology group of one graded component of 
\(
\bigoplus_{m \geq 2, n > 1 }
S^{m}\Omega_{X}(m-n)
\) 
is isomorphic to the \((N-i)\)th cohomology group of the corresponding graded part of the above complex.
\end{theorem}
The main issue with the complex given in Theorem \ref{thm: coho ci 1 intro} is that we have to deal with quotients. Consider the simplest case of an hypersurface: in order to study the cohomology of the complex, one has to understand the kernel of the following map
\[
\alpha^{*}(P): \frac{S}{(P,q)}[-1,-d] \to \frac{S}{(P,q)}.
\]
Said otherwise, one has to study the solutions of the partial differential equation \(\sum\limits_{i=0}^{N} \frac{\partial P}{\partial X_{i}} \frac{\partial}{\partial Y_{i}}\) \textsl{modulo} the ideal \((P,q)\). We managed to do so when the polynomial \(P\) has a particular shape: see Section \ref{sect: intermediate ampleness}. The rough idea is that, in this case, one can reduce the problem to a situation where there is no quotient, so that one can use Theorem \ref{thm: partial intro}.

Whereas there are still quotients in the complexes in Theorems \ref{thm: coho hyp 2 intro} and \ref{thm: coho ci 2 intro}, the upshot is that they are of relatively mild form, namely of the following form
\[
\frac{S}{(q^{i})}
\]
where \(i \in \N_{\geq 1}\). Note that the quadratic form \(q=\sum\limits_{i=0}^{N}X_{i}Y_{i}\) is left unchanged under the following action of \(\GL_{N+1}(\C)\)
\[
A \in \GL_{N+1}(\C),
\ \
A \cdot \big(Y,X\big)
\bydef
\big((A^{-1})^{T}Y, AX\big).
\]
Therefore, representation theory of the general linear group \(\GL_{N+1}(\C)\) applies. In theory, this should allow to get rid of the quotient via a section of the quotient map \(S \to \frac{S}{(q^{i})} \). The real difficulty is to understand how the arrows in the complex behaves with respect to these sections: this is the object of a current work.
\newline

Finally, in Section \ref{sect: applications}, we provide several applications.

In Section \ref{sect: vanish and non-vanish}, we use the Koszul complex on \(\Flag_{(1,2)}\C^{N+1}\simeq \P(T\P^{N})\), whose construction is detailed in Section \ref{sect: Koszul complex 2: ci}, in order to recover all known vanishing theorems on twisted symmetric powers of cotangent bundles of complete intersection (namely, \cite{BR}[Theorem (i), (ii), (iii)]). Then, we prove non-vanishing theorems, which show in particular that the above vanishing theorems are optimal: these non-vanishing statements are, to our knowledge, new. In particular, we prove the following (see also Theorem \ref{thm: application 1.1 non-vanish 1}):
\begin{theorem}
\label{thm: application 1.1 non-vanish 1 intro}
Let \(n \in \N\), \(\mathbi{d}=(d_{1}, \dotsc, d_{c})\), \(d \bydef \min\Set{d_{i} \ | \ 1 \leq i \leq c}\), and suppose furthermore that \(d \geq 2\). 
A smooth complete intersection \(X \subset \P^{N}\) of codimension \(c < \frac{N+1}{2}\) and multi-degree \(\mathbi{d}\) satisfies the following non-vanishing statements:
\begin{itemize}
\item{} for \(d \geq 3\), \(H^{0}(X,S^{d-1}\Omega_{X}(2d-3)) \neq (0)\);
\item{} for \(d=2\), \(H^{0}(X,S^{2}\Omega_{X}(2)) \neq (0)\).
\end{itemize}
\end{theorem}

The combination of some of these vanishing and non-vanishing theorems allows us to prove the following theorem (see also Theorem \ref{cor: generalization}):
\begin{theorem}
\label{cor: generalization intro}
Let \(d \in \N_{\geq 2}\) be a natural number, and let \(X \subset \P^{N}\) be a non-degenerate\footnote{This means that the complete intersection is not included in an hyperplane.} smooth complete intersection of codimension \(c < \frac{N}{2}\). The complete intersection \(X\) satisfies
\[
\bigoplus_{m \geq \max(d-2,1)} 
H^{0}\big(
X, S^{m}\Omega_{X}(m+\max(d-3,0))
\big)
=
(0)
\]
if and only if \(X\) is not included in an hypersurface of degree \( 2 \leq i \leq d\).
\end{theorem}
This result can be seen as a generalization of results obtained in \cite{Bog}[Theorems B and D]. In view of Hartshorne's conjecture on smooth projective varieties that are complete intersections (see \cite{Laz1}[Conjecture 3.2.8], and Section \ref{sect: vanish and non vanish 1}), a natural problem would be to give a proof of Theorem \ref{cor: generalization intro} that works for any smooth subvariety of codimension \(c<\frac{N}{3}\).

In Section \ref{sect: algebra}, we study the algebra 
\[
\bigoplus_{m \in \N} H^{0}(X, S^{m}\Omega_{X}(m))
\]
of a smooth complete intersection \(X \subset \P^{N}\) of codimension \(c < \frac{N}{2}\), using the explicit resolution of \(\O_{\P(TX)}\) detailed in Section \ref{sect: Koszul complex 2: ci}. We proved the following (see also Theorem \ref{thm: sym alg}):
\begin{theorem}
\label{thm: sym alg intro}
Let
\[
X\bydef \Set{q_{1}=0} \cap \dotsb \cap \Set{q_{k}=0} \cap \Set{P_{k=1}=0} \cap \dotsb \cap \Set{P_{c}=0}
\subset \P^{N}
\]
be a smooth complete intersection of codimension \(c < \frac{N}{2}\), where \(q_{1}, \dotsc, q_{k}\) are homogeneous quadratic polynomials, and where \(P_{k+1}, \dotsc, P_{c}\) are homogeneous of degree \(> 2\).

There is a natural graded isomorphism of \(\C\)-algebras
\[
\bigoplus_{m \in \N} H^{0}(X, S^{m}\Omega_{X}(m))
\simeq
\C[q_{1}, \dotsc, q_{k}] 
\subset \C[X].
\]
\end{theorem}
This theorem generalizes a result in \cite{DeOl}: in \textsl{loc.cit}, under the more restrictive hypothesis on the codimension \(c < \frac{N+2}{3}\), the authors provided a completely different proof of the same statement. Let us note that in the case where \(X=\Set{q_{1}=0}\cap \Set{q_{2}=0} \subset \P^{4}\)(which is not covered by the above theorem), we showed that the conclusion of Theorem \ref{thm: sym alg intro} still holds: see Remark \ref{rem: sym alg intro}. This provides an alternative proof of the fact that the tangent bundle of \(X\) is not big, which was proved independently in \cite{Hor} and \cite{Mall}.

In Section \ref{sect: intermediate ampleness}, we use the complex given in Theorem \ref{thm: coho ci 1 intro} (or the one given in Theorem \ref{thm: coho hyp 1}) in order to prove a vanishing theorem for cohomology of negatively twisted symmetric powers of cotangent bundles of special kinds of hypersurfaces (see also Theorem \ref{thm: vanishing hyp}):
\begin{theorem}
\label{thm: vanishing hyp intro}
Let \(H\) be a smooth hypersurface of degree \(d\), whose defining equation \(P\) is of the following form
\[
P \bydef X_{N}^{d}-F(X_{0}, \dotsc, X_{N-1}),
\]
where \(F\) is a smooth\footnote{By definition, this means that the associated hypersurface (in \(\P^{N-1}\) here) is smooth.} polynomial, homogeneous of degree \(d\). 
Let \(m \in \N\), \(n \in \Z\), and suppose that \(d \geq 3\).
As soon as \(m(d-3) > n\), one has the vanishing:
\[
H^{N-1}(H, S^{m}\Omega_{H}(d-n-(N+1)))
=
(0).
\]
In particular, as soon as \(d \geq 4\), this shows that \(\Omega_{H}\) is uniformly \(q\)-ample (with uniform bound \(\lambda=\frac{1}{d-3}\) if \(d \geq N+1\)), and hence so is a general hypersurface of degree \(d \geq 4\).
\end{theorem}
For the notion of (uniform) intermediate ampleness, we refer to Section  \ref{sect: intermediate ampleness}. A stronger version of this vanishing theorem was actually proved recently in \cite{Hor}, in a different context and via a different method. In \textsl{loc.cit}, the authors prove that for \textsl{any} smooth hypersurface \(H \subset \P^{N}, N \geq 3\), of degree \(d \geq 3\), and for any \(m \in \N, n \in \Z\) satisfying the \textsl{broad} inequality
\[
m(d-3) \geq n,
\]
one has the vanishing
\[
H^{0}(H,S^{m}TX(n))
=
(0).
\]
(One recovers our statement via Serre duality).

In Section \ref{sect: effective}, we use the explicit description of cohomology detailed in Section \ref{sect: coho sym power ci}. Once implemented on computers (see Appendix \ref{appendix: sage}), it allows to compute completely the cohomology for small twists and small symmetric powers. In particular, we showed from a purely computational point of view the existence of a (unique) global section of \(S^{6}\Omega_{H}(8)\) on a quartic Fermat hypersurface in \(\P^{3}\). From a theoretical point of view, this fact is well-known as such a section defines the exceptional divisor of bi-tangent lines in \(\P(TH)\): see Section \ref{subs: bi-tangent} for more details.

We also provide a simple example of a family of surfaces in \(\P^{4}\) along which the canonically twisted symmetric pluri-genera do not remain constant (see also Theorem \ref{thm: non-invariance})
\begin{theorem}
The \(1\)-parameter family
\[
X_{t}
\bydef
\Set{X_{0}^{4}+X_{1}^{4}+\dotsb+X_{4}^{4}-tX_{0}^{2}X_{4}^{2}=0} \cap \Set{-2X_{0}^{4}-X_{1}^{4}+X_{3}^{4}+2X_{4}^{4}=0} \subset \P^{4}
\]
satisfies \(h^{0}(S^{6}\Omega_{X_{0}}(K_{X_{0}}))=1\), whereas \(h^{0}(S^{6}\Omega_{X_{0}}(K_{X_{t}}))=0\) for a general \(t \in \C\).
\end{theorem}
It was asked by Paun if such a phenomenon occurs, and it was first answered negatively in \cite{Bog}, via a counter-example more sophisticated than the one provided here.

Finally, in Section \ref{subs: criteria ampleness}, we discuss possible strategies to provide explicit examples of complete intersection surfaces in \(\P^{4}\) whose cotangent bundle is ample. There is nothing particularly new in this section, but we felt that it was worth mentioning, particularly because this is one of the questions that originally motivated this paper.

\begin{con}
The terminology \textsl{Bott's formulas} is written at many places in the text. It encodes a general statement, that is often too general for the purposes at hand. Therefore, from time to time, it happens that what we are referring to as \textsl{Bott's formulas} is not originally due to Bott.
However, we found it convenient to englobe in one expression a general statement that can be declined in various ways. For that reason, we wrote an appendix on these formulas (see Appendix \ref{appendix: Bott}). This appendix is also the occasion to recall some notations regarding Schur functors: we therefore invite the unfamiliar reader to briefly read it.

Throughout the paper, the projective space \(\P^{N}\) is fixed, and we suppose that \(N \geq 2\). Recall as well that, throughout the paper, the projectivization \(\P(E)\) of a vector bundle \(E\) on a variety means the projectivization of lines.
\end{con}

\ack{I would like to particularly thank Andreas Höring, as the starting point of this work was the following question he asked me: on a surface \(S \subset \P^{3}\) (say, a Fermat surface for simplicity), can we construct (via "Wronskian methods") global sections of \(S^{m}\Omega_{S}(n)\) with ratio \( \frac{n}{m} < \frac{1}{2}\) ? As I found it hopeless to try constructing them blind, I sought for a way to detect where one could hope to find such global sections. This is what eventually lead to this paper, as along the way many other questions popped off. The original question is still to be answered, but at least we have a better idea where to find these sections!

I also would like to thank Lionel Darondeau and Damian Brotbek for initiating me to the general subject of cohomology of symmetric powers of cotangent bundles of complete intersections. During my thesis, they offered me to work with them in order to construct explicitly a surface in \(\P^{4}\) with ample cotangent bundle, following the description and strategy detailed in their work. I implemented on computer the strategy, but it quickly became evident that it would require too much computations to be able to conclude. This is this very problem that gave me the motivation to write explicitly the (complicated) complex detailed in Theorem \ref{thm: complex 2 intro cod 2} (even if I was not able to get anything out of it!).

Finally, I would like to thank Stéphane Druel, who indicated me the following result: for a finite map \(f: \P^{N} \to \P^{N}\), the push-forward vector bundle \(f_{*}\O_{\P^{N}}\) splits.
}

\section{Cohomology of twisted symmetric powers of cotangent bundles of projective spaces.}
\label{sect: coho proj}
\subsection{Generalized Euler exact sequence.}
\label{subs: gen Euler seq}
Recall that on the projective space \(\P^{N}\), there exists canonical (negatively twisted) vector fields, usually called \textsl{Euler vector fields}, defined as follows. Denote \(p\colon \C^{N+1} \setminus \Set{0} \to \P^{N}\) the canonical projection onto the projective space of dimension \(N\), and denote \((e_{i})_{0 \leq i \leq N}\) the canonical basis of \(\C^{N+1}\). Define then for \(0\leq i \leq N\)
\[
\gamma_{i} \colon
\left(
\begin{array}{ccc}
 \P^{N} & \to  &  T\P^{N}(-1) \\
  \lbrack x \rbrack & \mapsto  & (\lbrack x \rbrack, (\diff p)_{x}(e_{i}))
\end{array}
\right).
\]
The negative twist (necessary to make the maps \(\gamma_{i}\)'s well-defined) comes from the fact that for any \(\lambda \in \C^{*}\), one has the equality
\[
(\diff p)_{\lambda x}=\frac{1}{\lambda}(\diff p)_{x}.
\]
Using the Euler vector fields, one constructs the usual \textsl{Euler exact sequence}
\begin{equation}
\label{Euler exact seq}
\xymatrix{
0 \ar[r] & \Omega_{\P^{N}}(1) \ar[r] & \overset{N+1}{\underset{i=0}{\bigoplus}}\O_{\P^{N}} \ar[r] & \O_{\P^{N}}(1) \ar[r] & 0
},
\end{equation}
where the first arrow is given by
\[
\big(\lbrack x \rbrack, \omega\big)
\mapsto
\big(\lbrack x \rbrack, (\omega(\gamma_{0}), \dotsc, \omega(\gamma_{N}))\big)
\]
and the second arrow by
\[
\big(\lbrack x \rbrack, (v_{0}, \dotsc, v_{N})\big)
\mapsto
\big(\lbrack x \rbrack, \sum\limits_{i=0}^{N} X_{i}(x)v_{i} \big),
\]
where \(X_{0}, \dotsc, X_{N} \in H^{0}(\P^{N}, \O_{\P^{N}}(1))\) are the usual homogeneous coordinates on the projective space \(\P^{N}\).

In order to generalize the previous exact sequence to higher symmetric powers of the cotangent bundle of \(\P^{N}\), note that the exact sequence \eqref{Euler exact seq} can be rewritten as follows:
\begin{equation*}
\xymatrix{
0 \ar[r] & \Omega_{\P^{N}}(1) \ar[r] & \C[Y]_{1}\otimes\O_{\P^{N}}  \ar[r]^-{\delta} & \O_{\P^{N}}(1) \ar[r] & 0
},
\end{equation*}
where \(\delta\) is the (differential) operator
\[
\delta \bydef \sum\limits_{i=0}^{N} X_{i} \frac{\partial}{\partial Y_{i}}.
\]
Here, the sections \(X_{0}, \dotsc, X_{N}\) are seen as global sections of the line bundle \(\O_{\P^{N}}(1)\) on the base space, and \(Y\bydef (Y_{0}, \dotsc, Y_{N})\) are seen as the variables of the graded vector space of polynomials in \(N+1\) variables 
\[
\C[Y]
=
\bigoplus_{i=0}^{\infty}
\C[Y]_{i},
\]
where \(\C[Y]_{i}\) is the set of homogeneous polynomials of degree \(i\in \N\).

We are then naturally lead to consider, for any \(m \in \N_{\geq 1}\), the exact sequence of vector bundles
\begin{equation}
\label{Gen Euler exact seq}
\xymatrix{
0 \ar[r] & \Ker(\delta) \ar[r] & \C[Y]_{m}\otimes\O_{\P^{N}}  \ar[r]^-{\delta} & \C[Y]_{m-1}\otimes\O_{\P^{N}}(1)}.
\end{equation}
The complex \eqref{Gen Euler exact seq} is actually exact on the right, i.e. the map \(\delta\) is surjective. We record it in a lemma:
\begin{lemma}
\label{lemma: Euler 1}
In the above complex \eqref{Gen Euler exact seq}, the map \(\delta\) is surjective.
\end{lemma}
\begin{proof}
Let \( x = \lbrack x_{0}: \dotsb : x_{N} \rbrack \in \P^{N}\), and suppose without loss of generality that \(x_{0} \neq 0\). Consider the local sections around \(x\)
\[
\left\{ 
\begin{array}{ll}
L_{0}\bydef \frac{Y_{0}}{X_{0}} \\
L_{i} \bydef Y_{i} - \frac{X_{i}}{X_{0}}Y_{0} & \text{for \(1 \leq i \leq N\)}
\end{array}
\right..
\]
Note that \(\delta(L_{0})=1\), whereas \(\delta(L_{i})=0\) for \(1 \leq i \leq N\). By the Leibnitz rule, one has therefore the following equality for any \(\mathbf{\alpha}=(\alpha_{0}, \dotsc, \alpha_{N}) \in \N^{N+1}\):
\[
\delta(L_{0}^{\alpha_{0}+1} L_{1}^{\alpha_{1}} \dotsb L_{N}^{\alpha_{m}})
=
(\alpha_{0}+1)L_{0}^{\alpha_{0}}\dotsb L_{N}^{\alpha_{N}}.
\]
The surjectivity of \(\delta\) now follows from the equality
\[
\O_{\P^{N}, x}[Y_{0}, \dotsc, Y_{N}]=\O_{\P^{N}, x}[L_{0}, \dotsc, L_{N}].
\]
\end{proof}
The locally free sheaf \(\Ker(\delta)\) in the the complex \eqref{Gen Euler exact seq} can naturally be identified with the \(m\)th symmetric power of \(\Omega_{\P^{N}}(1)\):
\begin{lemma}
\label{lemma: Euler 2}
There is a natural isomorphism 
\[
\Ker(\delta)
\simeq
S^{m}\Omega_{\P^{N}}(m).
\]
\end{lemma}
\begin{proof}
Consider the natural injective map
\[
i\colon
\left(
\begin{array}{ccc}
  S^{m} \Omega_{\P^{N}}(m) 
  & 
  \longrightarrow 
  & 
  \C[Y]_{m} \otimes \O_{\P^{N}}   
  \\
  \big(x, \omega\big)
  & 
  \longmapsto  
  & 
  (x , \sum\limits_{i_{1}, \dotsc, i_{m}} \omega(\gamma_{i_{1}}\dotsb\gamma_{i_{m}}) Y_{i_{1}} \dotsb Y_{i_{m}}
\end{array}
\right),
\]
where one recalls that  \(\gamma_{i}\) denotes the \(i\)th Euler vector field, with \(0 \leq i \leq N\). It is a straightforward computation to see that the image actually lies in \(\Ker(\delta)\):
\begin{eqnarray*}
&
&
\delta\big(\sum\limits_{i_{1}, \dotsc, i_{m}} \omega(\gamma_{i_{1}} \dotsb \gamma_{i_{m}}) Y_{i_{1}} \dotsb Y_{i_{m}}\big)
\\
&
=
&
\sum\limits_{k=1}^{m} \sum\limits_{\dotsc, i_{k-1}, i_{k+1},\dotsc} 
\Big(
 \sum\limits_{i=0}^{N}X_{i} \omega(\gamma_{i_{1}}\dotsb\gamma_{i_{k-1}} \gamma_{i_{k+1}} \dotsb \gamma_{i_{m-1}}\cdot \gamma_{i}
\Big)
Y_{i_{1}}\dotsb Y_{i_{k-1}} Y_{i_{k+1}} \dotsb Y_{i_{m-1}}
\\
&
=
&
0.
\end{eqnarray*}
The fact that it is identically zero follows from the observation that, for \(i_{1} \dotsb i_{m-1}\) fixed, \(\omega(\gamma_{i_{1}} \dotsb\gamma_{i_{m-1}} \cdot)\) is a (local) section of \(\Omega_{\P^{N}}(1)\).

Reciprocally, let\footnote{We adopt the usual multi-indexes notation: \(Y^{\mathbi{\alpha}} \bydef Y_{0}^{\alpha_{0}} \dotsb Y_{N}^{\alpha_{N}}\).} 
\[
P=\sum\limits_{\abs{\mathbi{\alpha}}=m} P_{\mathbi{\alpha}} Y^{\mathbi{\alpha}} 
\] 
be a local section of \(\Ker(\delta)\), where the \(P_{\mathbi{\alpha}}\)'s are local functions on \(\P^{N}\). At every point \(x=[x_{0}:\dotsb:x_{N}]\), the fiber \(\big(S^{m}T\P^{N}(-m)\big)_{x}\) has the following description
\[
 \big(S^{m}T\P^{N}(-m)\big)_{x}
 =
 \frac{\underset{\abs{\mathbi{\alpha}}=m}{\bigoplus} \C \cdot \gamma^{\mathbi{\alpha}}}
 {\underset{\abs{\mathbi{\alpha}}=m-1}{\bigoplus} \C \cdot \gamma^{\mathbi{\alpha}} \big(\sum\limits_{i=0}^{N} x_{i}\gamma_{i}\big)},
 \]
 where by definition 
 \(\gamma^{\mathbi{\alpha}} 
 \bydef 
 \gamma_{1}^{\alpha_{1}} \dotsb \gamma_{N}^{\alpha_{N}}
 \in H^{0}(\P^{N}, S^{m}T\P^{N}(-m))\).
Now, for any \(x\) where \(P\) is defined and any \(\mathbi{\alpha}\) of length \(\abs{\mathbi{\alpha}}=m\), define
 \[
 \omega_{x}(\gamma^{\mathbi{\alpha}})
 \bydef
\mathbi{\alpha}!P_{\mathbi{\alpha}}(x).
 \]
As \(P\) is in \(\Ker(\delta)\), this descends to the quotient defining \(\big(S^{m}T\P^{N}(-m)\big)_{x}\), so that \(\omega\) defines   a local section of \(S^{m}\Omega_{\P^{N}}(m)\).

One immediately checks that two maps described above are inverse to each other, so that the lemma is proved.
\end{proof}
Putting together Lemma \ref{lemma: Euler 1} and Lemma \ref{lemma: Euler 2}, we have thus proved:
\begin{theorem}[Generalized Euler exact sequence]
\label{thm: gen Euler}
For any \(m\in \N_{\geq 1}\), the \(m\)th symmetric power of \(\Omega_{\P^{N}}(1)\) fits into the short exact sequence:
\begin{equation}
\label{eq: gen Euler}
\xymatrix{
0 \ar[r] 
& 
S^{m}\Omega_{\P^{N}}(m) 
\ar[r] 
& 
\C[Y]_{m}\otimes\O_{\P^{N}}  
\ar[r]^-{\delta} 
&
 \C[Y]_{m-1}\otimes\O_{\P^{N}}(1)  
 \ar[r] 
&
0,
}
\end{equation}
where \(\delta=\sum\limits_{i=0}^{N} X_{i} \frac{\partial}{\partial Y_{i}}\).
\end{theorem}

\subsection{Geometric interpretation.}
\label{subs: geom inter}

The generalized Euler exact sequence has a simple geometric interpretation. The projectivization \(\P(T\P^{N})\) of the tangent space \(T\P^{N}\) is naturally isomorphic to the flag variety \(\Flag_{(1,2)} \C^{N+1}\) via the following map:
\[
j \colon \left(
\begin{array}{ccc}
  \P(T\P^{N}) & \longrightarrow   &  \Flag_{(1,2)}\C^{N+1}  \\
  \big([x], v) & \longmapsto  &  \big(\C\cdot x \subsetneq \C\cdot x \oplus \C \cdot v\big)
\end{array}
\right).
\]
This is indeed well-defined, as the tangent space at \([x]\) may be described as follows
\[
T_{[x]} \P^{N}
\simeq
\frac{\C^{N+1}}{\C \cdot x}
\]
(indeed, \((\diff p)_{x} : \C^{N+1} \to T_{[x]}\P^{N}\) is surjective, with kernel \(\C \cdot x\)),
and it is clear that \(j\) is an isomorphism.

Recall that there exists two tautological vector bundles on \(\Flag_{(1,2)} \C^{N+1}\), defined as sub-vector bundles of the trivial bundle \(\Flag_{(1,2)}\C^{N+1} \times \C^{N+1}\) as follows:
\begin{itemize}
\item{} \(U_{1}\) such that the fiber above \(\xi \bydef (\C\cdot x \subsetneq C \cdot x \oplus \C \cdot v)\) is the line spanned by \(x\);
\item{} \(U_{2}\)  such that the fiber above \(\xi \bydef (\C\cdot x \subsetneq C \cdot x \oplus \C \cdot v)\) is the plane spanned by \(x\) and \(v\).
\end{itemize}
It gives rise to the short exact sequence of vector bundles
\[
\xymatrix{
0 \ar[r] & U_{1} \ar[r] & U_{2} \ar[r] & \frac{U_{2}}{U_{1}} \ar[r] & 0
},
\]
where the two extremities are line bundles on \(\Flag_{(1,2)} \C^{N+1}\). Denote, for any \(n, m \in \Z\), 
\[
\L_{m,n}
\bydef
(\frac{U_{2}^{\vee}}{U_{1}^{\vee}})^{m}
\otimes 
(U_{1}^{\vee})^{n}.
\]
We have the following proposition:
\begin{proposition}
\label{prop: geom inter}
The pull-back line bundle \(j^{*}\L_{m,n}\) on \(\P(T\P^{N})\) is equal to
\[
\O_{\P(T\P^{N})}(m) \otimes \pi^{*}\O_{\P^{N}}(m+n),
\]
where \(\pi: \Flag_{(1,2)}\C^{N+1} \to \P^{N}\) is the natural projection onto the space of lines in \(\C^{N+1}\).
\end{proposition}
\begin{proof}
First, observe that one has tautologically the following equality of line bundles:
\[
j^{*}U_{1}
=
\pi^{*}\O_{\P^{N}}(-1).
\]
Then, observe that there is a natural isomorphism
\[
j^{*}(\frac{U_{2}}{U_{1}}) 
\simeq 
\O_{\P(T\P^{N})}(-1)\otimes\pi^{*}\O_{\P^{N}}(-1)
\]
defined as follows:
\[
\Big(\big([x], [v]\big), \overline{v}\Big)
\mapsto
\Big(\big([x], [v]\big), \sum\limits_{i=0}^{N} v_{i} \gamma_{i}([x])\Big).
\]
(Note that it is indeed well defined since \(\sum\limits_{i=0}^{N}x_{i}\gamma_{i}([x])=0\)). The lemma follows immediately from the previous two observations.
\end{proof}

Recall Bott's formulas (in the absolute setting, see Appendix \ref{appendix: Bott}), that says in particular that for any \(n \in \Z\) and any \(m \in \N\), one has the following isomorphism:
\[
H^{0}(\Flag_{(1,2)}\C^{N+1}, \L_{m,n})
\simeq 
S^{(n,m)}\C^{N+1}.
\]
Here, \(S^{(n,m)}\) denotes the Schur functor associated to the partition \((n,m)\) if \(n \geq m\), and it denotes the zero functor otherwise, by convention. On the other hand, Bott's formulas (in the relative setting, in a particular case) implies also that for any \(n \in \Z\) and any \(m \in \N\) one has the isomorphism
\[
H^{0}(\P(T\P^{N}), \O_{\P(T\P^{N})}(m) \otimes \pi^{*} \O_{\P^{N}}(n+m))
\simeq 
H^{0}(\P^{N}, S^{m}\Omega_{\P^{N}}(m+n)).
\]
Therefore, one has the isomorphism for any \(n \in \Z\) and any \(m \in \N\)
\[
H^{0}(\P^{N}, S^{m}\Omega_{\P^{N}}(m+n))
\simeq 
S^{(n,m)} \C^{N+1}.
\]
 
 In order to relate this isomorphism to the generalized Euler exact sequence, recall that the vector space \(S^{(n,m)}\C^{N+1}\) has the following interpretation as a sub-vector space of the set of bi-homogeneous polynomials of bi-degrees \((m,n)\) relatively to the variables \(Y=(Y_{0}, \dotsc, Y_{N})\) and \(X=(X_{0}, \dotsc, X_{N})\)(see \cite{DemVanish}[Sect.2] for more details):
 \begin{equation*}
 S^{(n,m)} \C^{N+1}
 \simeq 
 \Set{P \in \C[Y,X]_{m,n} \ | \ \forall t \in \C, P(Y+tX, X)=P(Y,X) }.
 \end{equation*}

 One then easily sees that the set of \((m,n)\) bi-homogeneous polynomials sastifying the above functional equation coincides with set of \((m,n)\) bi-homogeneous polynomials \(P\) satisfying the partial differential equation
 \[
 (\sum\limits_{i=0}^{N} X_{i} \frac{\partial}{\partial Y_{i}})(P)
 \equiv 
 0.
 \]
 We recover therefore the description given by taking the long exact sequence in cohomology associated to the generalized Euler exact sequence.

 \subsection{Cohomology of twisted symmetric powers of \(\Omega_{\P^{N}}\).}
 \label{subs: coho proj space}
Via the geometric interpretation given in the previous paragraph, and more precisely using Proposition \ref{prop: geom inter}, the computation of cohomology of twisted symmetric powers of cotangent bundles of projective spaces is a straightforward application of Bott's formulas (see Appendix \ref{appendix: Bott}). 

A different approach would be to rather use the generalized Euler exact sequence \eqref{eq: gen Euler}. We will describe the two approaches.
\begin{theorem}[Cohomology of twisted symmetric powers of \(\Omega_{\P^{N}}\)]
\label{thm: coho sym proj}
Let \(n \in \Z\) be a natural number, and \(m \in \N_{\geq 0}\) be an integer.
\begin{itemize}
\item{} 
If \(n \geq m \geq 0\), then 
\[
\left\{ 
\begin{array}{ll}
H^{0}(\P^{N}, S^{m}\Omega_{\P^{N}}(m+n)) \simeq S^{(n,m)}\ \C^{N+1} \\
H^{i}(\P^{N}, S^{m}\Omega_{\P^{N}}(m+n))=(0) \ \text{for \(i>0\)}
\end{array}
\right..
\]
\item{}
If \(-(N+1) < n < m\), then
\[
\left\{ 
\begin{array}{ll}
H^{1}(\P^{N}, S^{m}\Omega_{\P^{N}}(m+n)) \simeq S^{(m-1,n+1)}\ \C^{N+1} \\
H^{i}(\P^{N}, S^{m}\Omega_{\P^{N}}(m+n))=(0) \ \text{for \(i\neq 1\)}
\end{array}
\right..
\]

\item{}
If \(n \leq -(N+1)\), then 
\[
\left\{ 
\begin{array}{ll}
H^{N}(\P^{N}, S^{m}\Omega_{\P^{N}}(m+n))
\simeq
S^{\lambda(n,m)}\C^{N+1}
\\
H^{i}(\P^{N}, S^{m}\Omega_{\P^{N}}(m+n))=(0) \ \text{for \(i\neq N\)}
\end{array}
\right.,
\]
where \(\lambda(n,m)\) is the partition 
\[
\big(m-(n+N+1), \underbrace{-(n+N+1), \dotsc, -(n+N+1)}_{\times (N-1)}\big).
\]
\end{itemize}
For the last item, the dimension of \(H^{N}(\P^{N}, S^{m}\Omega_{\P^{N}}(m+n))\) is equal to
\[
\dim(S^{m}\C^{N+1}\otimes S^{-(N+1)-n}\C^{N+1}) 
-
\dim(S^{m-1}\C^{N+1}\otimes S^{-(N+1)-n-1}\C^{N+1}).
\]
\end{theorem}

\begin{proof}[Proof using Bott's formulas]
Use the isomorphism \(\P(T \P^{N}) \simeq \Flag_{(1,2)}\C^{N+1}\) under which \(\O_{\P(T\P^{N})}(1)\) corresponds to \(\L_{1,-1}\) by Proposition \ref{prop: geom inter}. A first application of Bott's formulas (in the relative setting) implies that for any \(m \geq 0\) and any \(n \in \Z\), one has the following equality:
\[
H^{*}(\P^{N}, S^{m}\Omega_{\P^{N}}(m+n))
\simeq
H^{*}(\Flag_{(1,2)}\C^{N+1}, \L_{m,n}).
\]

Apply now a second time Bott's formulas, this time in the absolute setting: see Corollary \ref{cor: Bott}, and notations therein. In the case at hand, consider 
\[
\mathbi{\alpha}=(n,m,\underbrace{0, \dotsc,0}_{\times (N-1)}),
\]
and distinguish several cases:
\begin{enumerate}[(i)]
\item{}
If \(n \geq m\), then \(\mathbi{\alpha}\) is already a partition and one is in situation (2) of Corollary \ref{cor: Bott}.
\item{}
If \(-(N+1)<n<m\) and \(m-1 \geq n+1 \geq 0\), then for the permutation
\(\sigma=(12)\) one has:
\[
\tilde{\sigma}(\mathbi{\alpha})
=
(m-1,n+1,0, \dotsc, 0),
\]
so that one is in situation (2) of Corollary \ref{cor: Bott}
\item{}
If \(-(N+1)<n<m\) and \(m-1 < n+1\), then one necessarily has \(n=m-1\). Therefore, for the permutation \(\sigma=(1,2)\), one has
\[
\tilde{\sigma}(\mathbi{\alpha})
=
\mathbi{\alpha},
\]
so that one is in situation (1) of Corollary \ref{cor: Bott}.
\item{}
 If \(-(N+1)<n<m\) and \(n <-1\), then for the permutation \(\sigma=(1,n)\) one has 
 \[
 \tilde{\sigma}(\mathbi{\alpha})
 =
 \mathbi{\alpha}
 \]
 so that one is in situation (1) of Corollary \ref{cor: Bott}.

\item{}
If \(n \leq -(N+1)\), then for the cyclic permutation \(\sigma=(1,N+1,N, \dotsc, 2)\), one obtains:
\[
\tilde{\sigma}(\mathbi{\alpha})
=
(m-1,-1,-1, \dotsc, -1, n+N),
\]
which is non-increasing, so that one is in the situation (2) of Corollary \ref{cor: Bott}.
\end{enumerate}
Putting all these cases together, one obtains the description given in the statement. It remains to prove the last equality regarding the dimension of \(H^{N}\): this is a simple application of general formulas to compute the dimension of Schur bundles: see e.g. \cite{Fulton}. (This will also follow from the proof using the generalized Euler exact sequence).
\end{proof}

\begin{proof}[Proof using the generalized Euler exact sequence]
Let \(m \in \N\), and \(n \in \Z\). If \(n <-1\), then the results follow immediately from the (twisted) generalized Euler exact sequence combined with the well-known cohomology of the Serre line bundles \(\O_{\P^{N}}(i), i \in \Z\). 

Suppose now that \(n \geq -1\). Using the (twisted) generalized Euler exact sequence, one sees that, in order to compute the \(0\)th and \(1\)st cohomology groups of  \(S^{m}\Omega_{\P^{N}}(m+n)\), one has to understand the kernel and cokernel of the following map:
\[
\delta_{m,n}\bydef \sum\limits_{i=0}^{N} X_{i} \frac{\partial}{\partial Y_{i}}
\colon
\C[Y,X]_{m,n}
\longrightarrow 
\C[Y,X]_{m-1,n+1}.
\]
Note that if \(n=-1\), the result is obvious. Suppose accordingly that \(n \geq 0\).
The crucial observation is that the map \(\delta\) is equivariant with respect to the following natural action of the general linear group \(\GL_{N+1}(\C)\) on \(\C[Y,X]\):
\[
A \in \GL_{N+1}(\C), \ A \cdot (Y, X) \bydef (AY, AX).
\]
Therefore, \(\delta_{m,n}\) is a morphism of \(\GL_{N+1}(\C)\)-representations. 

Suppose first that \(n \geq m\). By Pieri's formula (see e.g \cite{Fulton}), one has the following decomposition of \(\C[Y,X]_{m,n}\) into irreducible components:
\[
\C[Y,X]_{m,n}
\simeq 
\bigoplus_{i=0}^{m}
S^{(n+i, m-i)}\C^{N+1}.
\]
On the other hand, for any \(0 \leq i \leq m\), the vector
\[
X_{0}^{n-m+i}Y_{0}^{i}(X_{1}Y_{0}-X_{0}Y_{1})^{m-i}
\]
is a highest weight vector, which spans accordingly the irreducible representation isomorphic to \(S^{(n+i, m-i)}\C^{N+1}\). One now easily sees that \(\delta_{m,n}\) sends \(S^{(n,m)}\C^{N+1}\) to zero, and realizes an isomorphism between \(S^{(n+i,m-i)}\C^{N+1}\) and its image for any \(1 \leq i \leq m\). Therefore, the kernel of \(\delta_{m,n}\) is isomorphic to \(S^{(m,n)}\C^{N+1}\), and using once again Pieri's formula for \(\C[Y,X]_{m-1,n+1}\), one sees that \(\delta_{m,n}\) is surjective.

Suppose now that \(0 \leq n < m\). Similarly, by Pieri's formula, one has the decomposition:
\[
\C[Y,X]_{m,n}
\simeq 
\bigoplus_{i=0}^{n}
S^{(m+i, n-i)}\C^{N+1}.
\]
The highest weight vectors spanning the irreducible components in the above decomposition are this time the following:
\[
\Big(X_{0}^{i}Y_{0}^{m-n+i}(X_{1}Y_{0}-X_{0}Y_{1})^{n-i}\Big)_{0 \leq i \leq n}.
\]
These vectors are never sent to zero by \(\delta_{m,n}\), so that the map \(\delta_{m,n}\) is injective. Using again Pieri's formula for \(\C[Y,X]_{m-1,n+1}\), one easily sees that the cokernel of \(\delta_{m,n}\) is isomorphic to \(S^{(m-1,n+1)}\C^{N+1}\): this finishes the proof.
\end{proof}

\subsection{Dualized generalized Euler sequence.}
\label{subs: dual Euler}
By Serre duality, one immediately deduces from Theorem \ref{thm: coho sym proj} the cohomology of twisted symmetric powers of tangent bundles of projective spaces. However, it is always useful to have an explicit description of these cohomology groups. By dualizing (and twisting) the generalized Euler exact sequence \eqref{eq: gen Euler}, i.e. by applying to it the functor \(\mathcal{H}om(\cdot, \O_{\P^{N}})\), one deduces the following exact sequence for any \(m \in \N\) and \(n \in \Z\)
\[
 \resizebox{\displaywidth}{!}{
\xymatrix{
0 \ar[r] 
& 
\C[Y]_{m-1}\otimes\O_{\P^{N}}(n-1)
\ar[r]^-{\delta^{*}} 
&
 \C[Y]_{m}\otimes\O_{\P^{N}}(n)
 \ar[r] 
&
S^{m}T\P^{N}(-m+n) 
\ar[r] 
& 
0.
}
}
\]
(Recall that \(\mathcal{E}xt^{1}(\O_{\P^{N}}(i), \O_{\P^{N}})=0\) for any \(i \in \Z\)). The dual map \(\delta^{*}\) can be computed explicitly, and it takes the following form:
\[
\delta^{*}
=
(\sum\limits_{i=0}^{N} X_{i}Y_{i})\times \Id 
+
\sum\limits_{i=0}^{N} X_{i}Y_{i}^{2} \frac{\partial}{\partial Y_{i}}.
\]
It has quite an unaesthetic form, and indeed this is not the "right" description of \(S^{m}T\P^{N}(-m+n)\). Consider rather the renormalization map
\[
u \colon\left(
\begin{array}{ccc}
  \C[Y] & \longrightarrow   &  \C[Y]
  \\
  Y^{\mathbi{\alpha}} & \longmapsto  &  \frac{Y^{\alpha}}{\mathbi{\alpha}!}
  \end{array}
\right),
\]
where by definition \(\mathbi{\alpha}!\bydef \alpha_{0}! \dotsb \alpha_{N}!\). For any \(m \in \N\) and any \(n \in \Z\), the renormalization map induces an isomorphism of sheaves
\[
u_{m,n} \colon\left(
\begin{array}{ccc}
  \C[Y]_{m} \otimes \O_{\P^{N}}(n) & \longrightarrow   &  \C[Y]_{m} \otimes \O_{\P^{N}}(n) 
  \\
  Y^{\mathbi{\alpha}} \otimes s & \longmapsto  &  \frac{Y^{\alpha}}{\mathbi{\alpha}!}\otimes s
\end{array}
\right).
\]
Form then the following diagram:
\begin{equation}
\label{eq: dual}
 \resizebox{\displaywidth}{!}{
\xymatrix{
& 0 \ar[d] & 0 \ar[d] & &
\\
0 \ar[r] 
& 
\C[Y]_{m-1}\otimes\O_{\P^{N}}(n-1)
\ar[r]^-{\delta^{*}} 
\ar[d]^{u_{m-1,n-1}}
&
 \C[Y]_{m}\otimes\O_{\P^{N}}(n)
 \ar[r] 
 \ar[d]^{u_{m,n}}
&
S^{m}T\P^{N}(-m+n) 
\ar[r] 
& 
0
\\
0 \ar[r] 
& 
\C[Y]_{m-1}\otimes\O_{\P^{N}}(n-1)
\ar[r]^-{\delta_{*}} 
\ar[d]
&
 \C[Y]_{m}\otimes\O_{\P^{N}}(n)
 \ar[r] 
 \ar[d]
&
Q
\ar[r] 
& 
0
\\
& 0 & 0 & &,
}
}
\end{equation}
where the map \(\delta_{*}\) is  the multiplication by \(q\bydef\sum\limits_{i=0}^{N} X_{i}Y_{i}\). The diagram is commutative: one indeed computes that for any \(\mathbi{\alpha}\), \(\abs{\mathbi{\alpha}}=m-1\), and any \(s \in \O_{\P^{N}}(n-1)\) local section:
\begin{eqnarray*}
(u_{m,n}\circ\delta^{*})(Y^{\mathbi{\alpha}}\otimes s)
&
=
&
\sum\limits_{i=0}^{N}sX_{i}\otimes\big(u(Y^{\mathbi{\alpha}+\mathbi{e_{i}}}) + \alpha_{i}u(Y^{\mathbi{\alpha}+\mathbi{e_{i}}})\big)
\\
&
=
&
\sum\limits_{i=0}^{N}sX_{i}\otimes \frac{Y^{\mathbi{\alpha}}}{\mathbi{\alpha}!}Y_{i}
\\
&
=
&
(\delta_{*} \circ u_{m-1,n-1})(Y^{\mathbi{\alpha}}\otimes s).
\end{eqnarray*}
Therefore, the above diagram \eqref{eq: dual} naturally induces an isomorphism between \(Q\) and \(S^{m}T\P^{N}(-m+n)\), and this is this description that we will favor. We record this in the following theorem:
\begin{theorem}[Dualized generalized Euler exact sequence]
\label{thm: gen Euler dual}
For any \(m\in \N_{\geq 1}\), the \(m\)th symmetric power of \(T\P^{N}(-1)\) fits into the short exact sequence:
\begin{equation}
\label{eq: gen Euler dual}
\xymatrix{
0 \ar[r] 
& 
\C[Y]_{m-1}\otimes\O_{\P^{N}}(-1)
\ar[r]^-{\delta_{*}} 
&
 \C[Y]_{m}\otimes\O_{\P^{N}}
 \ar[r] 
&
S^{m}T\P^{N}(-m) 
\ar[r] 
& 
0,
}
\end{equation}
where \(\delta_{*}=\cdot \times (\sum\limits_{i=0}^{N} X_{i} Y_{i})\).
\end{theorem}

\subsection{Some notations.}
\label{subs: notations}
In view of the previous sections, it is natural to consider the partial differential operator \(\delta\) as acting on the following (infinite dimensional) bi-graded vector bundle:
\[
\mathcal{S}
\bydef 
\bigoplus_{m \in \N, n \in \Z}
\C_{m}[Y]\otimes \O_{\P^{N}}(n).
\]
The bi-graduation is the natural one, i.e. \(\mathcal{S}_{m,n} \bydef \C_{m}[Y]\otimes \O_{\P^{N}}(n)\), and \(\delta\) shifts the graduation by \((-1,1)\), i.e.
\[
\delta(\mathcal{S}_{m,n})
\subset
\mathcal{S}_{m-1,n+1}.
\]
Note that this is actually an equality.
If we wish to restrict the operator \(\delta\) to a given graded part (as we implicitly did in the previous sections), we will denote:
\[
\delta_{m,n}
\bydef 
\delta_{\vert \mathcal{S}_{m,n}}.
\]
With these notations, the generalized Euler exact sequence Theorem \ref{thm: gen Euler} tells us exactly that:
\[
\Ker \delta
=
\bigoplus_{m \in \N, n \in \Z} \Ker \delta_{m,n}
\simeq
\bigoplus_{m \in \N, n \in \Z} S^{m}\Omega_{\P^{N}}(m+n).
\]
Another equivalent but aesthetic way of writing the above equality is the following:
\begin{equation}
\Ker \delta
\simeq
\bigoplus_{m \in \N, n \in \Z}
S^{m}(\Ker(\delta_{1,0}))(n).
\end{equation}

The same considerations hold for the operator \(\delta_{*}\), which this time shifts the graduation by \((1,1)\). The dualized generalized Euler exact sequence provides this time the following equality:
\[
\Coker(\delta_{*})
=
\bigoplus_{m \in \N, n \in \Z} \Coker(\delta_{*})_{m,n}
\simeq
\bigoplus_{m \in \N, n \in \Z} S^{m}T\P^{N}(m-n).
\]

In the following sections, we will deal with other partial differential operators acting on \(\mathcal{S}\), and shifting the bi-graduation. We therefore introduce the following terminology:
\begin{definition}
Let \(i, j \in \Z\) be natural numbers. A \((i,j)\)-bigraded operator on \(\mathcal{S}\) is an endomorphism of the sheaf \(\mathcal{S}\) that shifts the graduation by \((i,j)\).
\end{definition}
\begin{example}
For any \(k \in \N\), \(\delta^{k}\) is a \((-k,k)\) operator on \(\mathcal{S}\).
\end{example}

\subsection{Study of a natural class of endomorphisms of \(\mathcal{S}\).}
\label{subs: a particular class}
To any finite endomorphism \(f: \P^{N} \to \P^{N}\) is naturally associated an endomorphism of \(\mathcal{S}\) as follows. Denote \(d \bydef \deg(f)\), and recall that there exists (unique up to multiplication by a non-zero scalar) homogeneous polynomials \(P_{0}, \dotsc, P_{N} \in \C[X]_{d}\) of degree \(d\) such that \(f\) writes on \(\P^{N}\)
\[
f([x])
=
[P_{0}(x): \dotsb : P_{N}(x)].
\]
Consider then the endomorphism of \(\mathcal{S}\) defined as follows
\[
\delta(f)
\bydef 
\sum\limits_{i=0}^{N} P_{i} \frac{\partial}{\partial Y_{i}},
\]
and note that it is a \((-1,d)\)-bigraded operator on \(\mathcal{S}\). Note also that for \(f=\Id\), one recovers the endomorphism \(\delta\). The goal of this section is to see how cohomology provides a natural way to understand the global morphism
\[
\delta(f)(\P^{N}): S \longrightarrow S, \ \ A(Y,X) \longmapsto \sum\limits_{i=0}^{N} P_{i}(X) \frac{\partial A}{\partial Y_{i}}(Y,X),
\]
where by definition \(S \bydef H^{0}(\mathcal{S}) =  \C[Y,X]\).

Observe that, by definition, the vector bundle \(\Ker \delta(f)_{1,0}\) is the rank \(N\) vector bundle on \(\P^{N}\) whose fiber over \([x] \in \P^{N}\) is the hyperplane defined by the equation:
\[
 P_{0}(x)Y_{0} + \dotsb + P_{N}(x)Y_{N}=0.
 \]
Furthermore, the simple but key observation is that \(\Ker\delta(f)_{1,0}\) is nothing but the pull-back by \(f\) of \(\Ker \delta_{1,0}\), i.e.
\[
\Ker \delta(f)_{1,0}
=
f^{*} \Ker \delta_{1,0}.
\]
In the spirit of what we did in the previous sections, we prove the following:
\begin{proposition}
For any \(m \in \N\), the \(m\)th symmetric power of \(\Ker \delta(f)_{1,0}\) fits into the exact sequence
\begin{equation*}
\label{eq: inj map eul}
\xymatrix{
0 \ar[r] 
& 
S^{m}\Ker \delta(f)_{1,0}
\ar[r] 
& 
\C[Y]_{m}\otimes\O_{\P^{N}}  
\ar[r]^-{\delta(f)} 
&
 \C[Y]_{m-1}\otimes\O_{\P^{N}}(d)  
 \ar[r] 
&
0.
}
\end{equation*}
\end{proposition}
\begin{proof}
The injection of locally free sheaves 
\(
S^{m}\Ker \delta(f)_{1,0}
\to
\C[Y]_{m}\otimes\O_{\P^{N}}  
\)
is given by the map
\[
\big([x], v^{1}\dotsb v^{m}\big)
\overset{i}{\longmapsto}
\big(
[x], \sum\limits_{0 \leq i_{1}, \dotsc, i_{m} \leq N} v^{1}_{i_{1}} \dotsb v^{m}_{i_{m}} Y_{i_{1}} \dotsb Y_{i_{m}}
\big).
\]
Its image does indeed lie in the kernel of \(\delta(f)\):
\begin{eqnarray*}
(\delta(f) \circ i)(v^{1}\dotsb v^{m})
&
=
&
\sum\limits_{ i_{1}, \dotsc, i_{m}} \sum\limits_{k=1}^{m} v_{i_{1}}^{1}\dotsb v_{i_{m}}^{m} P_{i_{k}}\overset{\wedge_{i_{k}}}{Y_{i_{1}} \dotsb Y_{i_{m}}}
\\
&
=
&
\sum\limits_{k=1}^{m}
\sum\limits_{ \dotsc, i_{k-1}, i_{k+1}, \dotsc}
v_{i_{1}}^{1}Y_{i_{1}}\dotsb v_{i_{k-1}}^{k-1}Y_{i_{k-1}}
(\underbrace{
\sum\limits_{i=0}^{N} 
P_{i}(x)v_{i}^{k}}_{=0})
v_{i_{k+1}}^{k+1}Y_{i_{k+1}} \dotsb v_{i_{m}}^{m}Y_{i_{m}}
\\
&
=
&
0,
\end{eqnarray*}
since all the vectors \(v^{1}, \dotsc, v^{m}\) lie in \(\Ker \delta(f)_{1,0}(x)\).

The surjectivity of the map \(\delta(f)\) is proved as in the proof of Lemma \ref{lemma: Euler 1} (almost verbatim). The only difference is that, in order to construct an antecedent of a local section, one rather considers the covering of affine open sets
\(
\big(\Set{P_{i} \neq 0}\big)_{0 \leq i \leq N},
\)
and the following local sections of \(\C[Y]_{1}\otimes\O_{\P^{N}}\)
\[
\left\{ 
\begin{array}{ll}
L_{i}\bydef \frac{Y_{i}}{P_{i}} \\
L_{j} \bydef Y_{j} - \frac{P_{j}}{P_{i}}Y_{i} & \text{for \(j \neq i \)}
\end{array}
\right..
\]

As for the exactness in the middle, it follows from the equality of dimension:
\[
\dim \C[Y]_{m}
=
\dim \C[Y]_{m-1}
+
\rank (S^{m}\Ker \delta(f)_{1,0}).
\]
\end{proof}
\begin{remark}
Note that for \(f=\Id\), one recovers the generalized Euler exact sequence.
\end{remark}
In plain words, the previous statement says that for any \(x \in \P^{N}\), the map 
\[
\delta(f)_{x}
\bydef 
 \sum\limits_{i=0}^{N} P_{i}(x) \frac{\partial}{\partial Y_{i}}
\]
is the defining equation of \((S^{m}\Ker \delta(f)_{1,0})_{x}\) in \(\S^{m} \C^{N+1}\simeq \C[Y]_{m}\). 

By taking the long exact sequence in cohomology associated to the short exact sequence \eqref{eq: inj map eul} twisted by \(\O_{\P^{N}}(n), n \in \N\), one finds that 
\[
\left\{ 
\begin{array}{ll}
 \Ker\big(\delta(f)_{m,n}(\P^{N})\big)
 \simeq
  H^{0}\big(\P^{N}, (S^{m} \Ker \delta(f)_{1,0})\otimes \O_{\P^{N}}(n)\big)
 \\
  \Coker\big(\delta(f)_{m,n}(\P^{N})\big) 
  \simeq
  H^{1}\big(\P^{N}, (S^{m} \Ker \delta(f)_{1,0})\otimes \O_{\P^{N}}(n)\big)
\end{array}
\right..
\]
Since one still has the following equality for any \(m \in \N_{\geq 1}\)
\[
f^{*}S^{m}\Ker \delta_{1,0}
=
S^{m}\Ker \delta(f)_{1,0},
\]
the projection formula implies the following:
\[
H^{*}\big(\P^{N},S^{m} \Ker \delta(f)_{1,0}\otimes \O_{\P^{N}}(n)\big)
\simeq
H^{*}\big(\P^{N},
\underbrace{S^{m}\Ker \delta_{1,0}}_{\simeq S^{m}\Omega_{\P^{N}}(m)}
\otimes f_{*}\O_{P^{N}}(n)
\big).
\]
Accordingly, in order to understand the kernel and cokernel of the global map 
\[
\delta(f)(\P^{N})\colon
\C[Y,X] \longrightarrow \C[Y,X],
\]
it is enough to understand the direct images \(f_{*}(\O_{\P^{N}}(n))\) for any \(n \in \Z\). 
It turns out that such direct images are particularly well-understood, and one has:
\begin{lemma}
\label{lemma: direct image}
Let \(0 \leq r < d\) be an integer. The vector bundle \(f_{*}(\O_{\P^{N}}(r))\) splits as a direct sum of line bundles, all of whom have degree less or equal than zero. 
\end{lemma}
\begin{proof}
Since \(f\) is a finite map, the push-forward sheaf \(f_{*}(\O_{\P^{N}}(r))\) is locally free. Furthermore, the projection formula implies that for any \(j \in \Z\), one has:
\[
H^{*}(\P^{N}, f^{*}(\O_{\P^{N}}(j)) \otimes \O_{\P^{N}}(r))
\simeq
H^{*}(\P^{N}, \O_{\P^{N}}(j) \otimes f_{*}\O_{\P^{N}}(r))
\]
In particular, this implies that
\[
H^{i}(\P^{N}, f_{*}(\O_{\P^{N}}(r)) \otimes \O_{\P^{N}}(j))=(0)
\]
for any \(1 \leq i \leq N-1\) and any \(j\in \Z\). By a result of Horrocks (see \cite{Beauville}[Corollary 1.3]), this implies that \(f_{*}(\O_{\P^{N}}(r))\) splits as a direct sum of line bundles.

On the other hand, as \(f^{*}\O_{\P^{N}}(1)=\O_{\P^{N}}(d)\), and since by hypothesis one has the inequality \(r-d<0\), one deduces the vanishing
\[
H^{0}(\P^{N}, f_{*}(\O_{\P^{N}}(r)) \otimes \O_{\P^{N}}(-1))
\simeq
H^{0}(P^{N}, \O_{\P^{N}}(r-d))
=
(0).
\]
This immediately implies that all the line bundles appearing in the splitting of \(f_{*}(\O_{\P^{N}}(r))\) have degree less or equal than zero.
\end{proof}
Knowing what degrees appear in the splitting of the line bundles \(f_{*}\O_{\P^{N}}(r)\), \(0 \leq r < d\), seems however to be complicated in full generality. 

For later use, we record the following result, which is interesting in itself:
\begin{theorem}
\label{prop: bound inj}
Let \(m \in \N\), and \(n \in \N\). As soon as \(dm > n\), the linear map
\[
\delta(f)_{m,n}(\P^{N})\colon \
S_{m,n} \to S_{m-1,n+d}
\]
is injective, and the bound is sharp.
\end{theorem}
\begin{proof}
By the above (and the projection formula), the kernel of \(\delta(f)_{m,n}(\P^{N})\) is isomorphic to 
\[
H^{0}(\P^{N}, S^{m}\Omega_{\P^{N}}(m+q) \otimes \gamma_{*}(\O_{\P^{N}}(r))),
\]
where \(n=qd+r\) is the euclidean division of \(n\) by \(d\). By Lemma \ref{lemma: direct image}, the above vector space writes as a direct sum of pieces of the following form
\[
H^{0}(\P^{N}, S^{m}\Omega_{\P^{N}}(m+k)),
\]
where \(k \leq q=\lfloor \frac{n}{d} \rfloor\). By Theorem \ref{thm: coho sym proj}, these cohomology groups are zero if and only if \(k < m\): this finishes the proof.
\end{proof}

\section{Koszul complexes resolving structural sheaves of projectivized tangent bundles of smooth complete intersections.}
\label{sect: resolutions}

\subsection{A first Koszul complex: the case of hypersurfaces.}
\label{section: Koszul complex 1: hyp}
Let us fix \(H\bydef \Set{P=0} \subset \P^{N}\) a smooth hypersurface of degree \(d \geq 1\), and let us denote by 
\[
\pi_{H}\colon \P(T\P^{N}_{\vert H}) \to H
\]
the canonical projection from the restricted projectivized tangent bundle of the projective space to the base space \(H\).

\subsubsection{Construction of the Koszul complex}

A simple but crucial observation is that the differential \(\diff P\) allows to define a global section of \(\O_{\P(T\P^{N}_{\vert H})}(1) \otimes \pi_{H}^{*}\O_{H}(d)\):
\begin{lemma}
\label{lemma: basic differential map}
The map defined pointwise on \(\xi = ([x], [\overline{v}]) \in \P(T\P^{N}_{\vert H}(-1))\) by\footnote{The multiplicative factor \(\frac{1}{d}\) is motivated by the identity:
\[
\delta\big(\frac{1}{d} (\diff P)_{X}(Y)\big)
=
P.
\]
}
\[
s(P)(\xi)
\bydef
(\xi, v \mapsto \frac{1}{d}(\diff P)_{x}(v))
\]
glues to a global section of \(\O_{\P(T\P^{N}_{\vert H})}(1) \otimes \pi_{H}^{*}\O_{H}(d)\).
\end{lemma}
\begin{proof}
Note first that, in the statement, one implicitly identifies \((T\P^{N}(-1))_{[x]}\) with its natural description via the Euler vector fields, i.e.
\[
(T\P^{N}(-1))_{[x]}
\simeq 
\frac{\C \cdot \gamma_{0} \oplus \dotsb \oplus \C \cdot \gamma_{N}}
{\C \cdot \big(x_{0}\gamma_{0} + \dotsb + x_{N} \gamma_{N}\big)}.
\]
Therefore, a vector \(\overline{v} \in (T\P^{N}(-1))_{[x]}\) is seen as an equivalence class of a vector \(v=(v_{0}, \dotsc, v_{N})\) in \(\C^{N+1}\).

Let then \(\lambda, \mu \in \C^{*}\). One has the following two simple identities:
\begin{enumerate}
\item{}
\( (\diff P)_{\lambda x}(\mu v) = \lambda^{d-1} \mu \cdot (\diff P)_{x}(v) \);
\item{}
\( (\diff P)_{x}(v + \lambda x) = (\diff P)_{x}(v) + \underbrace{\lambda^{d-1}P(x)}_{=0 \ \text{since \([x] \in H\)}} \).
\end{enumerate}
Therefore, the map \(s(P)\) defines a global section of
\[
\O_{\P(T\P^{N}_{\vert H}(-1))}(1) \otimes \tilde{\pi}_{H}^{*}\O_{H}(d-1),
\]
where \(\tilde{\pi}_{H} \colon \P(T\P^{N}_{\vert H}(-1)) \to H \) is the canonical projection. 

The lemma now follows from the natural identification\footnote{Keep in mind that, throughout the paper, we use projectivization of lines.}
 between
\[
\O_{\P(T\P^{N}_{\vert H}(-1))}(1) \otimes \tilde{\pi}_{H}^{*}\O_{H}(d-1)
\]
and
\[
\O_{\P(T\P^{N}_{\vert H})}(1) 
\otimes
\pi_{H}^{*} \O_{H}(d).
\]

\end{proof}
Now, observe that the zero locus of the global section \(s(P)\) defines set-theoretically the projectivized tangent bundle \(\P(TH)\) inside \(\P((T\P^{N})_{\vert H})\), by very definition of the tangent bundle. Furthermore, by smoothness of \(H\), the ideal sheaf spanned by \(s(P)\) is reduced. Therefore, we deduce the following exact sequence:
\begin{equation}
\label{eq: exact seq hypersurface}
\xymatrix{
\mathcal{K}(s(P))\colon
0 
\ar[r]
&
\O_{\P(T\P^{N}_{\vert H})}(-1) \otimes \pi_{H}^{*}\O_{H}(-d)
\ar[rr]^-{\cdot \times s(P)}
&
&
\O_{\P(T\P^{N}_{\vert H})},
}
\end{equation}
which provides a locally free resolution of \(j_{*}\O_{\P(TH)}\), where \(j\colon \P(TH) \hookrightarrow \P(T\P^{N}_{\vert H})\) is the closed immersion of \(\P(TH)\) inside \(\P(T\P^{N}_{\vert H})\).

\subsubsection{Interpretation of the (twisted) Koszul complex on the base}
\label{subs: interpretation Koszul: hyp}
By twisting the exact sequence \eqref{eq: exact seq hypersurface} by \(\O_{\P(T\P^{N}_{\vert H})}(m) \otimes \pi_{H}^{*}\O_{H}(m+n)\)\footnote{One may wonder why we twist by \(\pi_{H}^{*}\O_{H}(m+n)\) and not by \(\pi_{H}^{*}\O_{H}(n)\): we do it for the sake of coherence, as throughout the paper, it will be more natural to compute the cohomology of \(S^{m}\Omega_{X}(m+n)\), where \(X\) is a complete intersection, rather than the cohomology of \(S^{m}\Omega_{X}(n)\).}, where \(m \in \N_{\geq 1}\), and \(n \in \Z\), one obtains the exact sequence:
\begin{equation}
\label{eq: exact seq hypersurface twisted}
\resizebox{\displaywidth}{!}{
\xymatrix{
\mathcal{K}(s(P))_{m,n}\colon
0
\ar[r]
&
\O_{\P(T\P^{N}_{\vert H})}(m-1) \otimes \pi_{H}^{*}\O_{H}(m+n-d)
\ar[r]
&
\O_{\P(T\P^{N}_{\vert H})}(m) \otimes \pi_{H}^{*}\O_{H}(m+n),
}
}
\end{equation}
which provides a locally free resolution of \(j_{*}\O_{\P(TH)}(m) \otimes \pi_{H}^{*}\O_{H}(m+n)\)\footnote{Throughout the text, we will often omit to write these push-forward inclusions, as they are implicit (and often burdens the notations).}.

Recall that, as soon as \(r \geq 0\), Bott's formulas (see Appendix \ref{appendix: Bott}) imply in particular that the higher direct images
\begin{enumerate}
\item{}
\(
(R^{i} \pi_{H})_{*} \big(\O_{\P(T\P^{N}_{\vert H})}(r) \otimes \pi_{H}^{*}\O_{H}(s)\big)
\)
\item{}
\(
(R^{i} \pi_{H})_{*}\big(j_{*}(\O_{\P(TH)}(r)) \otimes \pi_{H}^{*}\O_{H}(s)\big)
\simeq 
(R^{i} (j\circ\pi_{H}))_{*}\big(\O_{\P(TH)}(r) \otimes (j\circ \pi_{H})^{*}\O_{H}(s)\big)
\)
\end{enumerate}
vanish for \(i>0\) and any \(s \in \Z\). They also imply the following two equalities for any \(r \in \N_{\geq 0}\) and any \(s \in \Z\):
\begin{enumerate}
\item{}
\(
(\pi_{H})_{*}\big(\O_{\P(T\P^{N}_{\vert H})}(r) \otimes \pi_{H}^{*}\O_{H}(s)\big)
\simeq
S^{r}\Omega_{\P^{N}}(s)_{\vert H};
\)
\item{}
\(
(\pi_{H})_{*}\big(j_{*}\O_{\P(TH)}(r) \otimes \pi_{H}^{*}\O_{H}(s)\big)
\simeq
S^{r}\Omega_{H}(s).
\)
\end{enumerate}
By pushing forward by \(\pi_{H}\) the exact sequence \eqref{eq: exact seq hypersurface twisted}, one therefore keeps an exact sequence, which takes the following form
\begin{equation}
\label{eq: exact seq hyper 2}
\xymatrix{
0
\ar[r]
&
S^{m-1}\Omega_{\P^{N}}(m+n-d)_{\vert H}
\ar[rr]^-{(\pi_{H})_{*}(s(P))}
&
&
S^{m}\Omega_{\P^{N}}(m+n)_{\vert H},
}
\end{equation}
and which provides a locally free resolution of \(S^{m}\Omega_{H}(m+n)\).
By very construction, one easily sees that the push-forward map \((\pi_{H})_{*}(s(P))\) is nothing but the map induced by the multiplication by the \((1,d-1)\) bi-homogeneous polynomial \(\frac{1}{d}(\diff P)_{X}(Y)\):
\[
\xymatrix{
0
\ar[d]
&
&
0
\ar[d]
\\
S^{m-1}\Omega_{\P^{N}}(m+n-d)_{\vert H}
\ar[d]
\ar[rr]
&
&
S^{m}\Omega_{\P^{N}}(m+n)_{\vert H}
\ar[d]
\\
C[Y]_{m-1} \otimes \O_{H}(n-d+1)
\ar[d]^-{\delta}
\ar[rr]^-{\alpha(P)}
&
&
C[Y]_{m} \otimes \O_{H}(n)
\ar[d]^{\delta}
\\
\C[Y]_{m-2} \otimes \O_{H}(n-d+2)
\ar[d]
&
&
C[Y]_{m-1} \otimes \O_{H}(n+1)
\ar[d]
\\
0
&
&
0
}
\]
Here, by definition, one has denoted
\[
\alpha(P)\bydef \cdot \times \frac{1}{d}(\diff P)_{X}(Y)
\]
the multiplication map by \(\frac{1}{d}(\diff P)_{X}(Y)\).
Note that the induced map is indeed well-defined, since one has the following identity:
\[
\delta\circ\alpha(P)
=
\cdot \times P
+
\alpha(P) \circ \delta.
\]

Using the notations of Section \ref{sect: coho proj}, we can summarize the results of this section in the following statement:
\begin{theorem}
\label{thm: Koszul cx hyp}
The bi-graded (Koszul)complex
\[
\xymatrix{
\mathcal{K}(\alpha(P))\colon
0
\ar[r]
&
(\Ker \delta)_{\vert H}[-1,-d+1]
\ar[rr]^-{\alpha(P)}
&&
(\Ker \delta)_{\vert H}.
}
\]
provides a locally free bi-graded resolution of
\[
\bigoplus_{m \geq 1, n \in \Z} S^{m}\Omega_{H}(m+n).
\]
\end{theorem}

\subsubsection{A remark on the more general case of a smooth hypersurface of a smooth projective variety}
Consider the following more general setting: take \(X\) a smooth projective variety, fix \(L \) a line bundle on \(X\), and suppose that there exists \(s \in H^{0}(X,L)\) whose zero locus \(H \bydef \Set{s=0}\) defines a smooth hypersurface. One easily sees that Lemma \ref{lemma: basic differential map} generalizes as follows:
\begin{lemma}
There exists a global section of the line bundle
\[
\O_{\P(TX_{\vert H})}(1) \otimes \pi_{H}^{*}L
\]
which is locally induced by the differential of \(s\).
\end{lemma}
\begin{proof}
Let \((U_{i})_{i \in I}\) be an open covering of \(X\) trivializing both \(TX\) and \(L\), and denote by \((s_{i})_{i \in I}\) the local sections on \(U_{i}\) induced by \(s\). It is straightforward to check that the transition maps of the local sections defined pointwise on \(\P(TX_{\vert U_{i}})\) by
\[
 ([x], [v]) \to ([x], v \mapsto (ds_{i})_{x}(v))
\]
does glue to a global section of \(\O_{\P(TX_{\vert H})}(1) \otimes \pi_{H}^{*}L\).
\end{proof}
In this more general setting, Bott's formulas still holds, which allows to show the following (in exactly the same fashion as in the previous Section \ref{subs: interpretation Koszul: hyp}):
\begin{proposition}
\label{prop: general}
Let \(m \in \N\). The twisted symmetric power \(S^{m}\Omega_{H}\) fits into the short exact sequence
\[
\xymatrix{
0
\ar[r]
&
S^{m-1}\Omega_{X}(-L)
\ar[r]
&
S^{m}\Omega_{X}
\ar[r]
&
S^{m}\Omega_{H}
\ar[r]
&
0.
}
\]
\end{proposition}

\subsection{A first Koszul complex: the case of smooth complete intersections.}
\label{sect: Koszul complex 1: ci}
Let us now fix \(c\) hypersurfaces \((H_{i}=\Set{P_{i}=0})_{1 \leq i \leq c}\) of degrees \(d_{i} \geq 1\) in such a way that their intersection
\[
X \bydef H_{1} \cap \dotsb \cap H_{c} \subset \P^{N}
\]
defines a smooth complete intersection. Let us also denote by 
\[
\pi_{X}: \P(T\P^{N}_{\vert X}) \to X
\]
the canonical projection from the restricted projectivized tangent bundle of the projective space to the base space \(X\). 

\subsubsection{Construction of the Koszul complex}
As in the previous section, for every \(1 \leq i \leq c\), the differential \(\diff P_{i}\) induces a global section
\[
s(P_{i}) 
\in 
H^{0}(\P(T\P^{N}_{\vert H_{i}}), \O_{\P(T\P^{N}_{\vert H_{i}})}(1) \otimes \pi_{H_{i}}^{*}\O_{H_{i}}(d_{i})).
\]
After restriction to \(X\), such a global section \(s(P_{i})\) induces as well a global section of
\[
\O_{\P(T\P^{N}_{\vert X})}(1) \otimes \pi_{X}^{*}\O_{X}(d_{i}).
\]
We will denote \(s_{i}\) this restricted section, and define \(\mathbi{s}\bydef (s_{1}, \dotsc, s_{c})\).

Observe that the zero locus of the sections \(s_{1}, \dotsc, s_{c}\) defines set-theoretically the projectivized tangent bundle \(\P(TX)\) inside \(\P(T\P^{N}_{\vert X})\). Furthermore, the smoothness of \(X\) implies that the ideal sheaf spanned by these sections is reduced (as one easily sees via the Jacobian criterion). By constructing the Koszul complex \(\mathcal{K}(\mathbi{s})\) associated to the sections \(s_{1}, \dotsc, s_{c}\) (see e.g. \cite{Laz1}[Appendix B.2]), one therefore obtains a complex abuting to \(\O_{\P(TX)}\). What is more, the smoothness of \(\P(T\P^{N}_{\vert X})\) implies that this complex is exact, therefore providing  a resolution of \(\O_{\P(TX)}\). 
\begin{example}
For \(c=2\), the exact complex \(\mathcal{K}(\mathbi{s})\) takes the following form:
\begin{equation*}
\resizebox{\displaywidth}{!}{
\xymatrix{
\mathcal{K}(\mathbi{s})\colon
0
\ar[r]
&
\O_{\P(T\P^{N}_{\vert X})}(-2) \otimes \pi_{X}^{*}\O_{X}(-d_{1}-d_{2})
\ar@{-}[rr]^-{(s_{1}, s_{2})}
&
&
\
\\
\
\ar[r]
&
\O_{\P(T\P^{N}_{\vert X})}(-1) \otimes \pi_{X}^{*}\O_{X}(-d_{2})
\bigoplus 
\O_{\P(T\P^{N}_{\vert X})}(-1) \otimes \pi_{X}^{*}\O_{X}(-d_{1})
\ar[rr]^-{s_{2}(\cdot) - s_{1}(\cdot)}
&
&
\O_{\P(T\P^{N}_{\vert X})}.
}
}
\end{equation*}
\end{example}

\subsubsection{Interpretation of the (twisted) Koszul complex on the base}
We can follow the same line of reasoning as in the case of hypersurfaces. Twisting by
 \[
 \O_{\P(T\P^{N}_{\vert X})}(m) \otimes \pi_{X}^{*}\O_{X}(m+n)
 \] 
the Koszul complex \(\mathcal{K}(\mathbi{s})\), where \(m \in \N_{\geq c}\) and \(n \in \Z\), we obtain a resolution of \(\O_{\P(T\P^{N}_{\vert X})}(m) \otimes \pi_{X}^{*}\O_{X}(m+n)\). As in the case of hypersurfaces, we denote this resolution by \(\mathcal{K}(\mathbi{s})_{m,n}\). 

Using Bott's formulas, this resolution can be pushed-forward by \(\pi_{X}\) into a resolution of \(S^{m}\Omega_{X}(m+n)\). The pushed-forward maps \((\pi_{X})_{*}(s_{i})\) are again induced by the multiplication by the \((1, d_{i}-1)\)-bihomogeneous polynomial \(\frac{1}{d_{i}}(\diff P_{i})_{X}(Y)\). Namely, this is nothing but the map of graded vector bundles:
\[
\alpha(P_{i}) \colon
(\Ker \delta)_{\vert X}[-1,-d_{i}+1]
\longrightarrow
(\Ker \delta)_{\vert X}.
\]
The push-foward Koszul complex \((\pi_{X})_{*}\big(\mathcal{K}(\mathbi{s})_{m,n}\big)\) is then nothing but the suitable graded pieces of the Koszul complex
\[
\mathcal{K}\big(\alpha(P_{1}), \dotsc, \alpha(P_{c})\big)
\]
on \(\Ker \delta_{\vert X}\).
\begin{example}
For \(c=2\), the (push-forward) Koszul complex \((\pi_{X})_{*}\big(\mathcal{K}(\mathbi{s})_{m,n}\big)\) takes the following form:
\begin{equation*}
\resizebox{\displaywidth}{!}{
\xymatrix{
0
\ar[r]
&
(\Ker\delta_{m-2,n-d_{1}-d_{2}+2})_{\vert X}
\ar[rr]^-{(\alpha(P_{1}), \alpha(P_{2}))}
&
&
(\Ker\delta_{m-1,n-d_{2}+1})_{\vert X} 
\bigoplus 
(\Ker\delta_{m-1,n-d_{1}+1})_{\vert X} 
\ar@{-}[r]
&
\
\\
\ 
\ar[rr]^-{\alpha(P_{2})(\cdot) - \alpha(P_{1})(\cdot)}
&
&
(\Ker\delta_{m,n})_{\vert X},
}
}
\end{equation*}
and it provides a locally free resolution of
\(S^{m}\Omega_{X}(m+n)\).
\end{example}

The main result of this section can thus be summarized in the following statement:
\begin{theorem}
\label{thm: resolution base 1}
The Koszul complex
\[
\mathcal{K}\big(\alpha(P_{1}), \dotsc, \alpha(P_{c})\big)
\] 
induced by the multiplication maps \(\big(\alpha(P_{i})\big)_{1 \leq i \leq c}\) on \((\Ker \delta)_{\vert X}\)
provides a locally free bi-graded resolution of
\[
\bigoplus_{m \geq c, n \in \Z} S^{m}\Omega_{X}(m+n).
\]
\end{theorem}

 \subsection{A second Koszul complex: the case of smooth hypersurfaces.}
 \label{section: Koszul complex 2: hyp}
 Let \(H\bydef (P=0)\) be a smooth hypersurface of degree \(d\) in the projective space \(\P^{N}\), with \(P \in \C[X]\) an homogeneous polynomial of degree \(d\) in the variables \(X\bydef (X_{0}, \dotsc, X_{N})\). 
Following the same idea as in Section \ref{subs: geom inter}, we define a closed immersion of \(\P(TH)\) inside \(\Flag_{(1,2)}\C^{N+1} \simeq \P(T\P^{N})\) as follows:
\[
j \colon
\left(
\begin{array}{ccc}
  \P(TH) & \longrightarrow  & \Flag_{(1,2)}\C^{N+1}  \\
  \big([x], v\big) & \longmapsto   &   \big( \C \cdot x \subsetneq \C\cdot x \oplus \C \cdot v)\big)
\end{array}
\right).
\]
Note that the exact same proof of Proposition \ref{prop: geom inter} gives the following:
\begin{proposition}
The pull-back line bundle \(j^{*}\L_{m,n}\) on \(\P(TH)\) is equal to
\[
\O_{\P(TH)}(m) \otimes \pi_{H}^{*}\O_{H}(m+n),
\]
where \(\pi_{H}: \P(TH) \to H\) is the canonical projection onto the hypersurface \(H\).
\end{proposition}
The goal of this section is to construct a rank \(2\) vector bundle on \(\Flag_{(1,2)} \C^{N+1}\) admitting a global section defining the ideal sheaf of (the image of) \(\P(TH)\) inside \(\Flag_{(1,2)} \C^{N+1}\). From such a data follows classically the construction of a Koszul complex on \(\Flag_{(1,2)}\C^{N+1}\) resolving the structural sheaf of \(j_{*}\O_{\P(TH)}\). We then give the interpretation of the push-forward complex on the base space \(H\).

 \subsubsection{Construction of the Koszul complex}
 \label{subs: Koszul complex 2: hyp}
 We keep the notations of Section \ref{subs: geom inter}, and in particular still denote by \(\pi: \Flag_{(1,2)}\C^{N+1} \to \P^{N}\) the canonical projection.
Consider the covering family of open sets of \(\Flag_{(1,2)}\C^{N+1}\)
\[
\big(\tilde{V_{i}}\bydef \pi^{-1}(V_{i})\big)_{0 \leq i \leq N},
\]
where by definition \(V_{i} \bydef \Set{X_{i} \neq 0}\). On each open set \(\tilde{V_{i}}\), there exists a family of local sections \((g_{ji})_{0 \leq j \leq N} \in \L_{1,0}(V_{i})\) defined as follows. Let \(\xi=\big( \C\cdot x \subsetneq \C\cdot x \oplus \C \cdot y \big)\) be a point in \(\tilde{V_{i}}\), and define
\[
g_{ji}\big((\xi, \lambda x + \mu y)\big)
=
\mu \frac{x_{i}y_{j}-x_{j}y_{i}}{x_{i}}.
\]
If we substitute \(x\) with \(\alpha x\), and \(y\) with \(\beta_{1}y+ \beta_{2}x\) (for \(\alpha, \beta_{1} \in \C^{*}\) and \(\beta_{2} \in \C\)), then the value of \(g_{ji}\) is multiplied by \(\beta_{1}\). In other words, \(g_{ji}(\xi, .)\) defines a linear functional on the fiber \(\big(\frac{U_{2}}{U_{1}}\big)_{\xi}\), so that \(g_{ji}\) is indeed a local section of \(\L_{1,0}=(\frac{U_{2}}{U_{1}})^{\vee}\) on \(V_{i}\).

We then construct a rank \(2\) vector bundle on \(\Flag_{(1,2)}\C^{N+1}\) by glueing the rank \(2\) vector bundles 
\[
\big(\L_{0,0} \oplus \L_{1,0}\big)_{\vert \tilde{V_{i}}}
\]
with the transition maps from the \(i\)th chart to the \(j\)th chart
\[
\varphi_{ij}\colon
\left(
\begin{array}{ccc}
  \big(\L_{0,0}\oplus \L_{1,0}\big)_{\vert V_{i} \cap V_{j}} 
  &  
  \longrightarrow
  &   
  \big(\L_{0,0} \oplus \L_{1,0}\big)_{\vert V_{j}\cap V_{i}}
  \\
  (s_{1}, s_{2})
  &
  \longmapsto
  &  
  (s_{1}, s_{2})
  \begin{pmatrix}
  \frac{X_{i}}{X_{j}} 
  & 
  g_{ij}
  \\
  0
  &
  1
  \end{pmatrix}  
\end{array}
\right).
\]
Note that this indeed makes sense since
\begin{eqnarray*}
\begin{pmatrix}
  \frac{X_{i}}{X_{j}} 
  & 
  g_{ij}
  \\
  0
  &
  1
  \end{pmatrix}  
  \begin{pmatrix}
  \frac{X_{j}}{X_{k}} 
  & 
  g_{jk}
  \\
  0
  &
  1
  \end{pmatrix}  
  &
  =
  &
  \begin{pmatrix}
  \frac{X_{i}}{X_{k}} 
  & 
  \frac{X_{i}}{X_{j}}g_{jk} + g_{ij}
  \\
  0
  &
  1
  \end{pmatrix}  
  \\
  &
  =
  &
    \begin{pmatrix}
  \frac{X_{i}}{X_{k}} 
  & 
  g_{ik}
  \\
  0
  &
  1
  \end{pmatrix}  .
\end{eqnarray*}
We record this construction in the following definition:
\begin{definition}
Denote by \(\E\) the rank \(2\) vector bundle on \(\Flag_{(1,2)}\C^{N+1}\) obtained by the previous construction. For any \(m, n \in \Z\), denote as well
\[
\E_{m,n}
\bydef 
\E \otimes \L_{m,n}.
\]
\end{definition}

The vector bundle \(\E\) has a simple structure, as it is a (non-trivial) extension of \(\L_{1,0}\) by \(\L_{0,1}\):
\begin{proposition}
\label{prop: extension}
The vector bundle \(\E\) fits canonically into a short exact sequence
\[
\xymatrix{
0 \ar[r] 
& 
\L_{1,0} \ar[r]
&
\E \ar[r]
&
\L_{0,1} \ar[r]
&
0
}.
\]
\end{proposition}
\begin{proof}
The inclusion of \(\L_{1,0}\) into \(\E\) is straightforward from the very construction, and is simply given by
\[
s \mapsto (0,s).
\]
On the other hand, on each open set \(\tilde{V_{i}}\), the maps
\[
p_{i}\colon
\left(
\begin{array}{ccc}
  \E_{\vert \tilde{V_{i}}} & \longrightarrow  & (\L_{0,0})_{\vert \tilde{V_{i}}}
  \\
  (s_{1}, s_{2}) & \longmapsto   &   X_{i}s_{1}
\end{array}
\right)
\]
glue into a global map \(\E \to \L_{0,1}\), by very construction of \(\E\).

The complex obtained via these two maps is clearly exact, and provides the sought exact sequence.
\end{proof}

The fundamental property of the vector bundle \(\E\) is that, with a suitable twist, it admits a global section defining \(j(\P(TH))\) inside \(\Flag_{(1,2)}\C^{N+1}\). Indeed, for any \(0 \leq i \leq N\), define a local section \(s^{i} \in \big(\L_{0,d-1} \oplus \L_{1,d-1}\big)(\tilde{V_{i}})\) as follows. Let \(\xi=\big( \C\cdot x \subsetneq \C\cdot x \oplus \C \cdot y \big)\) be a point in \(\tilde{V_{i}}\), and define:
\[
s^{i}(\xi, (x,y))
=
\Big(\frac{P(x)}{x_{i}}, \frac{1}{d} (\diff P)_{x}(y) - \frac{y_{i}}{x_{i}} P(x)\Big).
\]
It is clear that the first coordinate of \(s^{i}\) is in \(\L_{0,d-1}(\tilde{V_{i}})\). 
Using the fact that 
\[
\frac{1}{d} (\diff P)_{x}(x)=P(x),
\]
one immediately checks that if one substitutes  \(x\) with \(\alpha x\), and \(y\) with \(\beta_{1}y+ \beta_{2}x\) (for \(\alpha, \beta_{1} \in \C^{*}\) and \(\beta_{2} \in \C\)), then the value of the second coordinate of \(s^{i}\) is multiplied by \(\beta_{1}\alpha^{d-1}\). Therefore, the second coordinate of \(s^{i}\) does indeed define a local section of \(\L_{1,d-1}(\tilde{V_{i}})\).
It is then straightforward to check  that the local sections \(s^{i}\) glue to give a global section \(s(P) \in \E_{0,d-1}\). We record it in the following lemma:
\begin{lemma}
\label{lemma: local form}
The local sections
\[
s^{i}=\Big(\frac{P(X)}{X_{i}}, \frac{1}{d} (\diff P)_{X}(Y) - \frac{Y_{i}}{X_{i}} P(X)\Big)
\in 
\big(\L_{0,d-1} \oplus \L_{1,d-1}\big)(\tilde{V_{i}})
\]
glue to a global section \(s(P)\) of \(\E_{0,d-1}\).
\end{lemma}
\begin{proof}
It suffices to check that
\[
s^{i}   \begin{pmatrix}
  \frac{X_{i}}{X_{j}} 
  & 
  g_{ij}
  \\
  0
  &
  1
  \end{pmatrix}  
=
s^{j}.
\]
\end{proof}

It is clear that the (image of the multiplication map by the) section 
\[
s\bydef s(P)
\]
 defines the ideal sheaf associated to the closed immersion \(j\colon\P(TH) \hookrightarrow \Flag_{(1,2)}\C^{N+1}\), by its very construction.  We can accordingly construct the Koszul complex associated to it
\[
\xymatrix{
\mathcal{K}(s)\colon
0 
\ar[r] 
&
\L_{-1, -2d+1}
\ar[r] 
&
\E_{-1,-d} \ar[r]
&
\L_{0,0},
}
\]
which provides a locally free resolution of \(j_{*}\O_{\P(TH)}\).
The (twisted) third term in the complex comes from the isomorphism 
\[
\bigwedge^{2} \E_{0,d-1} \simeq \L_{1,2d-1}
\]
obtained by twisting by \(\L_{0,d-1}\) the exact sequence in Proposition \ref{prop: extension}, and taking its determinant.
It is an exact complex, since \(\Flag_{(1,2)}\C^{N+1}\) is smooth, and \(\codim j(\P(TH))=\rank(\E_{0,d-1})=2\). We can always twist this exact complex by \(\L_{m,n}\) to obtain the following exact complex for any \(n, m \in \Z\):
\begin{equation}
\label{eq: resolution hyp}
\xymatrix{
\mathcal{K}(s)_{m,n}\colon
0 \ar[r] 
&
\L_{m-1,n-2d+1}
\ar[r] 
&
\E_{m-1,n-d} 
\ar[r]
&
\L_{m,n},
}
\end{equation}
which provides a locally free resolution of \(j_{*}\O_{\P(TH)}\otimes \L_{m,n}\).

 \subsubsection{Interpretation of the (twisted) Koszul complex on the base}
 \label{subs: interpretation hyp}
As soon as \(m \geq 0\), Bott's formulas (see Appendix \ref{appendix: Bott}) imply that the higher direct image sheaves
\[
R^{i}\pi_{*}\L_{m,n}
\]
vanish for \(i>0\) and any \(n \in \Z\), where we recall that \(\pi: \Flag_{(1,2)}\C^{N+1} \to \P^{N}\) is the canonical projection. Using the isomorphism of Proposition \ref{prop: geom inter}, Bott's formulas also imply that
\[
\pi_{*}\L_{m,n}
\simeq
S^{m}\Omega_{\P^{N}}(m+n)
\simeq 
\Ker(\delta_{m,n}).
\]
Using still Bott's formulas, and the projection formula, one also gets that
\[
\pi_{*}(j_{*}\O_{\P(TH)}\otimes \L_{m, n})
\simeq
i_{*}(S^{m}\Omega_{H})(m+n),
\]
where \(i\colon H \hookrightarrow \P^{N}\) is the closed immersion of the hypersurface \(H\) inside the projective space \(\P^{N}\). As soon as \(m\geq 1\), if one takes the push forward by \(\pi\) of the exact sequence \eqref{eq: resolution hyp}, one keeps an exact sequence, which takes the following form
\begin{equation}
\label{eq: resolution hyp pushed}
\xymatrix{
\pi_{*}(\mathcal{K}(s)_{m,n})\colon
0 
\ar[r]
&
\Ker \delta_{m-1,n-2d+1}
\ar[r]
&
\pi_{*}(\E_{m-1,n-d}) 
\ar[r]
&
\Ker \delta_{m,n} 
},
\end{equation}
and which resolves \(i_{*}S^{m}\Omega_{H}(m+n)\).

In order to apprehend the above complex, we must  understand the push-forward sheaf \(\pi_{*}(\E_{m-1,n-d})\). It turns out that it has a simple interpretation:
\begin{proposition}
\label{prop: push 1}
For any \(s \in \Z\) and any \(r \in \N\), the locally free sheaf \(\pi_{*}(\E_{r,s})\) is canonically isomorphic to the kernel sheaf of the map:
\[
\xymatrix{
\C[Y]_{r+1}\otimes \O_{\P^{N}}(s)
\ar[r]^-{\delta \circ \delta}
&
\C[Y]_{r-1}\otimes \O_{\P^{N}}(s+2),
}
\]
i.e.  \(\pi_{*}(\E_{r,s}) \simeq \Ker \delta^{2}_{r+1,s}\).
\end{proposition}
\begin{proof}
Recall that, by definition, the locally free sheaf \(\E_{s,r}\) is obtained by suitably glueing the free sheaves 
\[
\big(\L_{r,s} \oplus \L_{r+1,s}\big)_{\vert \tilde{V_{i}}}
\]
on the covering open sets \(\tilde{V_{i}}=\pi^{-1}(\Set{X_{i}\neq 0})\) (see Subsection \ref{subs: Koszul complex 2: hyp}). Pushing everything forward by \(\pi\), one sees that  \(\pi_{*}(\E(s,r))\) is therefore obtained by glueing
\[
\big(\Ker \delta_{r,s} \oplus \Ker \delta_{r+1,s}\big)_{\vert V_{i}}
\]
via the transition maps
\[
A_{ij}
\bydef
\begin{pmatrix}
  \frac{X_{i}}{X_{j}} 
  & 
  \frac{X_{j}Y_{i}-X_{i}Y_{j}}{X_{j}}
  \\
  0
  &
  1
  \end{pmatrix}.
\]

There is then a natural map from  \(\pi_{*}(\E(s,r))\) to \(\Ker \delta^{2}_{r+1,s}\) defined as follows. On each trivializing open set \(V_{i}\), consider the map
\[
\Theta_{i}\colon
\left(
\begin{array}{ccc}
  (\Ker \delta_{r,s})_{\vert V_{i}} \oplus (\Ker \delta_{r+1,s})_{\vert V_{i}}
  & 
  \longrightarrow
  &   
  (\Ker \delta^{2}_{r+1,s})_{\vert V_{i}}
  \\
  (s_{1}, s_{2}) 
  &
   \longmapsto
  & 
  s_{2} + Y_{i}s_{1}
\end{array}
\right).
\]
Note that the image does indeed lie in \((\Ker \delta^{2}_{r+1,s})_{\vert V_{i}}\), since (using the Leibniz rule for \(\delta\)) one has the equality
\[
(\delta \circ \delta)(s_{2} + Y_{i}s_{1})=(\delta \circ \delta)(Y_{i})\times s_{1}=0.
\]
These local maps glue to a global map 
\[
\Theta \colon \pi_{*}(\E(r,s)) \to \Ker \delta^{2}_{r+1,s},
\] 
which follows from the equality:
\[
\begin{pmatrix}
  \frac{X_{i}}{X_{j}} 
  & 
  \frac{X_{j}Y_{i}-X_{i}Y_{j}}{X_{j}}
  \\
  0
  &
  1
  \end{pmatrix}
  \begin{pmatrix}
  Y_{j}
  \\
  1
  \end{pmatrix}
  =
  \begin{pmatrix}
  Y_{i}
  \\
  1
  \end{pmatrix}.
\]

The injectivity of the map \(\Theta\) is straightforward to check: if one has locally on \(V_{i}\) the equality
\[
s_{2}+Y_{i}s_{1}=0,
\]
then by applying \(\delta\) one gets the equality \(\delta(Y_{i})s_{1}=0 \Leftrightarrow X_{i}s_{1}=0\). This implies that \(s_{1}\) is zero, and thus  so is \(s_{2}\). As for the surjectivity, observe that if \(s \in (\Ker \delta^{2}_{r+1,s})_{\vert V_{i}}\), then by setting
\[
\left\{ 
\begin{array}{ll}
s_{1}\bydef \frac{\delta(s)}{X_{i}}
\\
s_{2}\bydef s-\frac{\delta(s)}{X_{i}}Y_{i}
\end{array}
\right.
\]
one has the equality
 \[
 s=s_{2}+Y_{i}s_{1}=\Theta((s_{1}, s_{2})),
 \]
 with \(\delta(s_{1})=\delta(s_{2})=0\). This finishes the proof.
\end{proof}

Note that in view of the above Proposition \ref{prop: push 1}, the push-forward of the twisted exact sequence
\[
\xymatrix{
0 \ar[r] 
& 
\L_{s,r+1}
 \ar[r]
&
\E_{r,s} 
\ar[r]
&
\L_{s+1,r} 
\ar[r]
&
0
}
\]
of Proposition \ref{prop: extension} writes as follows: 
\begin{equation*}
\xymatrix{
0 \ar[r]
&
\Ker \delta_{r+1,s}
\ar[r]
&
\Ker \delta^{2}_{r+1,s}
\ar[r]^-{\delta}
&
\Ker \delta_{r,s+1}
\ar[r]
&
0.
}
\end{equation*}
Furthermore, the resolution of \(i_{*}(S^{m}\Omega_{H})(m+n)\), where \(m\geq 1\), can be rewritten as follows:
\begin{equation}
\label{eq: fundamental diagram}
\xymatrix{
0 
\ar[r]
&
\Ker \delta_{m-1,n-2d+1}
\ar[r]^-{\alpha(P)}
&
\Ker \delta^{2}_{m,n-d}
\ar[r]^-{\beta(P)}
&
\Ker \delta_{m,n}.
}
\end{equation}

It now remains to describe the different arrows in this complex. Let us justify that they take the following form:
\[
\left\{ 
\begin{array}{ll}
\alpha(P)\bydef \cdot \times \frac{1}{d} (\diff P)_{X}(Y)
\\
\beta(P)\bydef \cdot \times P - \alpha(P) \circ \delta
\end{array}
\right..
\]
On the trivializing open set \(\tilde{V}_{i}\), the section \(s(P)\) takes, by construction, the following form:
\[
s(P)
\overset{\text{loc}}{=}
\Big(\frac{P(X)}{X_{i}}, \frac{1}{d} (\diff P)_{X}(Y) - \frac{Y_{i}}{X_{i}} P(X)\Big).
\]
Under the isomorphism between \(\pi_{*}(\E_{0,d-1})\) and \(\Ker \delta_{1,d-1}^{2}\), the section becomes
\[
\Big(\frac{1}{d} (\diff P)_{X}(Y) - \frac{Y_{i}}{X_{i}} P(X) \Big)+ Y_{i}\times \Big(\frac{P(X)}{X_{i}} \Big)
=
\frac{1}{d} (\diff P)_{X}(Y).
\]
Hence the shape of the first map. As for the second map, consider  \(t=(t_{1}, t_{2}) \in \E_{0,d_1}(\tilde{V}_{i})\) a local section. One computes that
\begin{eqnarray*}
s(P) \wedge t
&
\overset{\text{loc}}{=}
&
\frac{P(X)}{X_{i}}t_{2} - \Big(\frac{1}{d} (\diff P)_{X}(Y) - \frac{Y_{i}}{X_{i}} P(X) \Big)t_{1}
\\
&
\overset{\text{loc}}{=}
&
\frac{1}{X_{i}}
\big(P(X)\times (t_{2}+Y_{i}t_{1})
-
\alpha(P)(X_{i}t_{1})
\big).
\end{eqnarray*}
Now, observe that under the isomorphism between \(\pi_{*}(\E_{0,d-1})\) and \(\Ker \delta_{1,d-1}^{2}\), \(t\) is identified with \(t_{2}+Y_{i}t_{1}\), where \(t_{1}, t_{2} \in \Ker \delta\), and compute accordingly that
\[
\delta( t_{2}+Y_{i}t_{1})
=
X_{i}t_{1}.
\]
Therefore, one has that, locally,
\[
s(P) \wedge t
\overset{\text{loc}}{=}
\frac{1}{X_{i}}\beta(P)(t_{2}+Y_{i}t_{1}).
\]
To conclude, observe that the factor \(\frac{1}{X_{i}}\) comes from the (implicit in the above writing) trivialization of the determinant line bundle \(\det(\E)\)\footnote{Recall that \(\E\) is obtained by suitably glueing \(\L_{1,0} \P^{N} \oplus \L_{0,0} \P^{N}\) on the open sets \(\tilde{V_{i}}\).}.
\begin{remark}
Using the equality:
\[
\delta \circ \alpha - \alpha \circ \delta 
=
\cdot \times P,
\]
one easily sees that that the chain of maps \eqref{eq: fundamental diagram} is indeed a complex (and well-defined).
\end{remark}

Note that the map \(\alpha(P)\) is a \((1, d-1)\) operator on \(\mathcal{S}\), whereas the map \(\beta(P)\) is a \((0,d)\)-operator on \(\mathcal{S}\). The fundamental result that we have shown in this section can be summed up as follows:
\begin{theorem}
\label{thm: resol hyp}
The complex 
\begin{equation}
\xymatrix{
0 
\ar[r]
&
\Ker \delta[-1,-2d+1]
\ar[r]^-{\alpha(P)}
&
\Ker \delta^{2}[0,-d]
\ar[r]^-{\beta(P)}
&
\Ker \delta
}
\end{equation}
provides a locally free bi-graded resolution of
 \[
 \bigoplus_{m \geq 1, n \in \Z} S^{m}\Omega_{H}(m+n).
 \]
\end{theorem}

\subsection{A second Koszul complex: the case of smooth complete intersections.}
\label{sect: Koszul complex 2: ci}
Let \(X\bydef (P_{1}=0) \cap \dotsc \cap (P_{c}=0)\) be a smooth complete intersection of multi-degree \(\mathbi{d}\bydef (d_{1}, \dotsc, d_{c})\) in the projective space \(\P^{N}\), with \(P_{i} \in \C[X]\) an homogeneous polynomial of degree \(d_{i}\) in the variables \(X\bydef (X_{0}, \dotsc, X_{N})\). 
As in the previous Section \ref{section: Koszul complex 2: hyp}, we define a closed immersion of \(\P(TX)\) inside \(\Flag_{(1,2)}\C^{N+1} \simeq \P(T\P^{N})\) as follows:
\[
j \colon
\left(
\begin{array}{ccc}
  \P(TX) & \longrightarrow  & \Flag_{(1,2)}\C^{N+1}  \\
  \big([x], v\big) & \longmapsto   &   \big([x], (\C\cdot x \oplus \C \cdot v)\big)
\end{array}
\right).
\]
Once again, the exact same proof of Proposition \ref{prop: geom inter} gives the following:
\begin{proposition}
The pull-back line bundle \(j^{*}\L_{m,n}\) on \(\P(TX)\) is equal to
\[
\O_{\P(TX)}(m) \otimes \pi_{X}^{*}\O_{X}(m+n),
\]
where \(\pi_{X}: \P(TX) \to X\) is the canonical projection onto the smooth complete intersection \(X\).
\end{proposition}

This section is organized exactly as in the previous one: we provide the construction of a Koszul complex on \(\Flag_{(1,2)}\C^{N+1}\) resolving the structural sheaf \(j_{*}\O_{\P(TX)}\), and give the interpretation of the push-forward complex on the base space \(\P^{N}\) in the case of codimension \(2\) complete intersections (for sake of simplicity).

\subsubsection{Construction of the Koszul complex}
Following notations of the previous Section \ref{section: Koszul complex 2: hyp}, consider, for \(1 \leq i \leq c\), the global sections \(s_{i} \bydef s(P_{i}) \in H^{0}(\Flag_{(1,2)}\C^{N+1}, \E_{d_{i}-1,0})\). They allow to define a global section 
\(
\mathbi{s}\bydef (s_{1}, \dotsc, s_{c})
\)
of 
\[
\E^{\mathbi{d}}
\bydef
\E_{0,d_{1}-1} \oplus \dotsc \oplus \E_{0,d_{c}-1}.
\]
The zero set of the section \(\mathbi{s}\) coincide with with \(j(\P(TX))\) since, by construction, it is given by the equations:
\[
\left\{ 
\begin{array}{ll}
P_{1}= \dotsb = P_{c}=0
\\
\diff P_{1} = \dotsb = \diff P_{c} = 0
\end{array}
\right..
\]
Furthermore, the smoothness of \(X\) implies that the ideal sheaf \(\mathcal{I}(\mathbi{s})\) spanned by these equations is reduced (as an immediate application of, say, the Jacobian criterion). One has therefore the short exact sequence:
\[
\xymatrix{
0 
\ar[r]
&
\mathcal{I}(\mathbi{s}) 
\ar[r]
&
\O_{\Flag_{(1,2)}\C^{N+1}} 
\ar[r]^-{\pr}
&
j_{*}\O_{\P(TX)} 
\ar[r]
&
0.
}
\]
Constructing in the usual fashion the Koszul complex associated to the section \(\mathbi{s}\) (see \cite{Laz1}[Appendix B]), one obtains accordingly the following complex:
\begin{equation}
\label{Koszul complex 2: ci}
\xymatrix{
\mathcal{K}(\mathbi{s})\colon
0
\ar[r]
&
\L_{-c, -2\abs{\mathbi{d}}+c}
\ar[r]^-{s \times \cdot }
&
(\E^{\mathbi{d}})_{-c, -2\abs{\mathbi{d}}+c}
\ar[r]^-{s \wedge \cdot}
&
\dotsb
\ar[r]^-{s \wedge \cdot}
&
\L_{0,0},
}
\end{equation}
which provides a locally free resolution of \(\O_{\P(TX)}\).

In order to understand explicitly this complex, we will push it forward by
 \[
 \pi: \Flag_{(1,2)}\C^{N+1} \to \P^{N}.
 \]
 However, whereas the push-forward sheaf \(\pi_{*} \E\) has a simple description, as indicated by Proposition \ref{prop: push 1} of the previous Section \ref{subs: interpretation hyp}, it is not as simple when it comes to computing the pushed-forward sheaves \(\pi_{*}\E^{\otimes k}\) for \(k \geq 2\) (which are the building blocks of \(\pi_{*}\bigwedge^{\bigcdot} \E^{\mathbi{d}}\)).

In the next section, we will give a full description in the case of smooth complete intersection of codimension \(2\). The path followed can of course be adapted to the general case, but it becomes much more complicated as far as notations, combinatorics, and determination of the push-forward maps are concerned. As we will not need this description in the applications (the Koszul complex upstair on \(\Flag_{(1,2)}\C^{N+1}\) will be enough), we refrain ourselves of making such a general study. 

\subsubsection{Interpretation of the (twisted) Koszul complex on the base: the case of smooth complete intersections of codimension \(2\)}
In exactly the same fashion as in the beginning of Section \ref{section: Koszul complex 2: hyp}, for any \(i \in \N_{\geq 1}\), and any \(r \in \Z, s \in \N\), we can push-forward the Koszul complex \eqref{Koszul complex 2: ci}, and keep an exact complex.

In order to understand the push-forward on the base of the Koszul complex, we first give a satisfying description of the push-forward sheaf \(\pi_{*}((\E^{\otimes 2})_{r,s})\). We start with the following lemma, which generalizes the construction in Proposition \ref{prop: push 1}:

\begin{lemma}
\label{prop: push map}
For any \(k \in \N_{\geq 1}\), for any \(s \in \Z\) and any \(r \in \N\), there is a natural map of locally free sheaves
\[
\xymatrix{
\pi_{*}((\E^{\otimes k})_{s,r})
\ar[r]^-{\Theta}
&
\Ker \delta^{\circ k+1}_{r+k,s}.
}
\]
\end{lemma}
\begin{proof}
Recall that, by definition, the locally free sheaf \(\E\) is obtained by suitably glueing the free sheaves 
\[
\big(\L_{0,0} \oplus \L_{1,0}\big)_{\vert \tilde{V_{i}}}
\]
on the covering open sets \(\tilde{V_{i}}=\pi^{-1}(\Set{X_{i}\neq 0})\) (see Section \ref{section: Koszul complex 2: hyp}).
Note that there are canonical isomorphisms for any \(0 \leq i \leq N\):
\[
\big(\L_{0,0} \oplus \L_{1,0}\big)^{\otimes k}_{\vert \tilde{V_{i}}}
\simeq 
\bigoplus_{\mathbi{\alpha} \in \Set{0,1}^{k}}
(\L_{\abs{\mathbi{\alpha}},0})_{\vert \tilde{V_{i}}},
\]
where by definition \(\abs{\mathbi{\alpha}}\) is the sum of the coordinates of \(\mathbi{\alpha}\) (i.e. the number of coordinates equal to one). Twisting by \(\L_{s,r}\) and pushing forward by \(\pi\), one gets the following canonical isomorphism on each open set \(V_{i}\), \(0 \leq i \leq N\):
\begin{equation}
\label{eq: triv}
\pi_{*}((\E^{\otimes k})_{r,s})_{\vert V_{i}}
\simeq
\big(
\bigoplus_{\mathbi{\alpha} \in \Set{0,1}^{k}}
\Ker \delta_{r+\abs{\mathbi{\alpha}}, s}
\big)_{\vert V_{i}}.
\end{equation}

For notational reasons, adopt the following formalism\footnote{It will "dualize" the notations used so far, but it is convenient for the proof here.}. Fix \(\mathbi{e}_{0}\) and \(\mathbi{e}_{1}\) two formal symbols, and denote for \(\mathbi{\alpha} \in \Set{0,1}^{k}\)
\[
\mathbi{e}_{\mathbi{\alpha}}
\bydef 
\mathbi{e}_{\alpha_{0}}
\otimes 
\dotsb
\otimes
\mathbi{e}_{\alpha_{k}}.
\]
Rewrite then the above isomorphism \eqref{eq: triv} as follows:
\begin{equation}
\label{eq: triv2}
\pi_{*}((\E^{\otimes k})_{r,s})_{\vert V_{i}}
\simeq
\bigoplus_{\mathbi{\alpha} \in \Set{0,1}^{k}}
(\Ker \delta_{r+\abs{\mathbi{\alpha}},s})_{\vert V_{i}}
\cdot \mathbi{e}_{\mathbi{\alpha}}.
\end{equation}
Denote then
\(A_{ij}^{*} \bydef 
\begin{pmatrix}
  \frac{X_{i}}{X_{j}} 
  & 
  0
  \\
  \frac{X_{j}Y_{i}-X_{i}Y_{j}}{X_{j}}
  &
  1
  \end{pmatrix},
  \)
and interpret this matrix as an isomorphism of the \(\C(X)[Y]\)-vector space of dimension \(2\) with canonical basis
\[
\C(X)[Y]\ \cdot \mathbi{e}_{0} 
\oplus
\C(X)[Y]\ \cdot \mathbi{e}_{1}.
\]
With these notations, the transition maps between the above trivializations \eqref{eq: triv2} are given by the matrix tensor product \((A_{ij}^{*})^{\otimes k}\). In other words, for any \(s \in  (\Ker \delta_{r+\abs{\mathbi{\alpha}},s})_{\vert V_{i}}\), the element 
\[
s \cdot \mathbi{e}_{\mathbi{\alpha}}
=
s \cdot (\mathbi{e}_{\alpha_{0}} \otimes \dotsb \otimes \mathbi{e}_{\alpha_{k}})
\]
becomes, after a change of trivialization from the \(i\)th chart to the \(j\)th chart,
\[
s \cdot (A_{ij}^{*}\mathbi{e}_{\alpha_{0}} \otimes \dotsb \otimes A_{ij}^{*}\mathbi{e}_{\alpha_{k}}).
\]

There is then a natural map from  \(\pi_{*}((\E^{\otimes k})_{r,s})\) to \(\Ker \delta^{\circ k+1}_{r+k,s}\) defined as follows. On each trivializing open set \(V_{i}\), consider the map
\[
\Theta_{i}\colon
\left(
\begin{array}{ccc}
  \bigoplus_{\mathbi{\alpha} \in \Set{0,1}^{k}}
(\Ker\delta_{r+\abs{\mathbi{\alpha}},s}
)_{\vert V_{i}}
\cdot \mathbi{e}_{\mathbi{\alpha}}.
  & 
  \longrightarrow
  &   
  (\Ker\delta^{\circ k+1}_{r+k,s})_{\vert V_{i}}
  \\
  (s_{\mathbi{\alpha}})_{\mathbi{\alpha} \in \Set{0,1}^{k}} 
  &
   \longmapsto
  & 
  \sum\limits_{\mathbi{\alpha} \in \Set{0,1}^{k}} Y_{i}^{k-\abs{\mathbi{\alpha}}}s_{\mathbi{\alpha}}
  \end{array}
\right).
\]
One easily sees that the image of \(\Theta_{i}\) does indeed lie in \(\Ker(\delta^{\circ k+1})_{\vert V_{i}}\). Furthermore, the maps \(\Theta_{i}\) glue to a global map. Indeed, by linearity, it suffices to show the following equality for any \(\mathbi{\alpha} \in \Set{0,1}^{k}\) and any \(s \in (\Ker \delta_{r+\abs{\mathbi{\alpha}},s})_{\vert V_{i}\cap V_{j}}\):
\[
 \Theta_{j}(s \cdot A_{ij}^{*}\mathbi{e}_{\alpha_{0}} \otimes \dotsb \otimes A_{ij}^{*}\mathbi{e}_{\alpha_{k}})
 =
 \Theta_{i}(s \cdot \mathbi{e}_{\mathbi{\alpha}}).
\]
Suppose without loss of generality that \(\mathbi{\alpha}=(\underbrace{0, \dotsc, 0}_{\ell \bydef k-\abs{\mathbi{\alpha}}}, 1, \dotsc, 1)\). One has that
\[
A_{ij}^{*}\mathbi{e}_{\alpha_{0}} \otimes \dotsb \otimes A_{ij}^{*}\mathbi{e}_{\alpha_{k}}
=
(\frac{X_{i}}{X_{j}}\mathbi{e}_{0} +  \frac{X_{j}Y_{i}-X_{i}Y_{j}}{X_{j}} \mathbi{e}_{1})^{\otimes \ell} \otimes \mathbi{e}_{1}^{\otimes k-\ell}.
\]
Then, one computes that:
\begin{eqnarray*}
 \Theta_{j}(s \cdot A_{ij}^{*}\mathbi{e}_{\alpha_{0}} \otimes \dotsb \otimes A_{ij}^{*}\mathbi{e}_{\alpha_{k}})
 &
 =
 &
 s
 \sum\limits_{p=0}^{\ell} \binom{\ell}{p} Y_{j}^{p} \big(\frac{X_{i}}{X_{j}}\big)^{p} \big(\frac{X_{j}Y_{i}-X_{i}Y_{j}}{X_{j}}\big)^{\ell-p}
 \\
 &
 =
 &
 s Y_{i}^{\ell} 
 \\
 &
 =
 &
 \Theta_{i}(s \cdot \mathbi{e}_{\mathbi{\alpha}}).
 \end{eqnarray*}
 This concludes the proof.
\end{proof}

We can now provide a convenient description of the various twists of the push-forward sheaf \(\pi_{*}(\E^{\otimes k})\) for \(k=2\):
\begin{proposition}
\label{prop: push 2}
For any \(s \in \Z\) and any \(r \in \N\), the locally free sheaf \(\pi_{*}((\E^{\otimes 2})_{s,r})\) is isomorphic to
\[
\Ker \delta^{3}_{r+2,s} \oplus \Ker \delta_{r+1,s+1}.
\]
\end{proposition}
\begin{proof}
Keep the notations introduced in the proof of Lemma \ref{prop: push map}, and consider the map 
\[
\Theta\colon
\pi_{*}((\E^{\otimes 2})_{s,r})
\to 
\Ker \delta^{3}_{r+2,s}.
\]
Recall that, locally on the open sets \(V_{i}\), it is defined as follows:
\[
\xymatrix{
(\Ker\delta_{r,s})_{\vert V_{i}}
\oplus
(\Ker\delta_{r+1,s})^{\oplus 2}_{\vert V_{i}}
\oplus
(\Ker\delta_{r+2,s})_{\vert V_{i}}
\ar[r]
&
(\Ker \delta^{3}_{r+2,s})_{\vert V_{i}}
\\
(s_{0}, s_{1}, s_{2}, s_{3})
\ar@{|->}[r]
&
s_{0}Y_{i}^{2} + (s_{1}+s_{2})Y_{i} + s_{3}.
}
\]
Consider also the following map of locally free sheaves
\[
\Delta\colon
\pi_{*}((\E^{\otimes 2})_{s,r})
\to 
\Ker \delta_{r+1,s+1}
\]
defined  on each open set \(V_{i}\) as follows:
\[
\xymatrix{
(\Ker\delta_{r,s})_{\vert V_{i}}
\oplus
(\Ker\delta_{r+1,s})^{\oplus 2}_{\vert V_{i}}
\oplus
(\Ker\delta_{r+2,s})_{\vert V_{i}}
\ar[r]
&
(\Ker \delta_{r+1,s+1})_{\vert V_{i}}
\\
(s_{0}, s_{1}, s_{2}, s_{3})
\ar@{|->}[r]
&
X_{i}(s_{2}-s_{1}).
}
\]
Note that it is well defined: on the \(j\)th chart, the element \((0, s_{1}, s_{2}, 0)\) becomes 
\[
(0, \frac{X_{i}}{X_{j}}s_{1}, \frac{X_{i}}{X_{j}}s_{2}, *),
\]
so that \(X_{i}(s_{2}-s_{1})=X_{j}(\frac{X_{i}}{X_{j}}s_{2}-  \frac{X_{i}}{X_{j}}s_{1})\).

It remains to justify that the morphism of locally free sheaves
\[
(\Theta, \Delta)\colon
\pi_{*}((\E^{\otimes 2})_{s,r})
\mapsto
\Ker \delta^{3}_{r+2,s} \oplus \Ker \delta_{r+1,s+1}
\]
is an isomorphism. For the injectivity, suppose that 
\[
\mathbi{s}=(s_{0}, s_{1}, s_{2}, s_{3}) 
\in 
(\Ker\delta_{r,s})_{\vert V_{i}}
\oplus
(\Ker\delta_{r+1,s})^{\oplus 2}_{\vert V_{i}}
\oplus
(\Ker\delta_{r+2,s})_{\vert V_{i}}
\]
satisfies \(\Theta(\mathbi{s})=\Delta(\mathbi{s})=0\), namely:
\[
\left\{
\begin{array}{ll}
s_{0}Y_{i}^{2} + (s_{1}+s_{2})Y_{i} + s_{3}=0;
\\
X_{i}(s_{2}-s_{1})=0.
\end{array}
\right.
\]
The second equation implies \(s_{2}=s_{1}\). Applying successively \(\delta^{2}\) and \(\delta\) to the first equation implies successively the equalities \(s_{0}=s_{1}=s_{2}=s_{3}=0\), which proves the injectivity. As for the surjectivity, let \((u,v) \in (\Ker \delta^{3}_{r+2,s})_{\vert V_{i}} \oplus (\Ker \delta_{r+1,s+1})_{\vert V_{i}}\). Consider
\[
\left\{
\begin{array}{ll}
s_{0}=\frac{1}{2} \frac{\delta^{2}(u)}{X_{i}^{2}}
\\
s_{1}=\frac{1}{2}(\frac{\delta(u)}{X_{i}} - \frac{\delta^{2}(u)Y_{i}}{X_{i}^{2}}-v)
\\
s_{2}=\frac{1}{2}(\frac{\delta(u)}{X_{i}} - \frac{\delta^{2}(u)Y_{i}}{X_{i}^{2}}+v)
\\
s_{3}=u-(s_{2}+s_{1})Y_{i}-s_{0}Y_{i}^{2}
\end{array}
\right.,
\]
and denote \(\mathbi{s}=(s_{0}, s_{1}, s_{2}, s_{3})\).
By construction, one has that \((\Theta, \Delta)(\mathbi{s})=(u,v)\). One then checks immediately that
\(\delta(s_{0})=\delta(s_{1})=\delta(s_{2})=0\). Furthermore, compute that
\begin{eqnarray*}
\delta(s_{3})
&
=
&
\delta(u)-(s_{1}+s_{2})X_{i}-2s_{0}Y_{i}X_{i}
\\
&
=
&
\delta(u)-\big(\delta(u)-\delta^{2}(u)\frac{Y_{i}}{X_{i}}\big)-\delta^{2}(u)\frac{Y_{i}}{X_{i}}
\\
&
=
&
0.
\end{eqnarray*}
This proves the surjectivity, and finishes the proof.
\end{proof}

Let us now go back to our Koszul complex \eqref{Koszul complex 2: ci} (where \(\mathbi{d}=(d_{1}, d_{2})\) as we work in codimension \(2\) here):
\begin{equation}
\label{Koszul complex 2: ci codim 2}
\resizebox{\displaywidth}{!}{
\xymatrix{
\L_{-2, -2\abs{\mathbi{d}}+2}
\ar@{^{(}->}[r]^-{s \times \cdot }
&
(\E^{\mathbi{d}})_{-2, -2\abs{\mathbi{d}}+2}
\ar[r]^-{s \wedge \cdot}
&
(\bigwedge^{2}\E^{\mathbi{d}})_{-2, -2\abs{\mathbi{d}}+2}
\ar[r]^-{s \wedge \cdot}
&
(\bigwedge^{3}\E^{\mathbi{d}})_{-2, -2\abs{\mathbi{d}}+2}
\ar[r]^-{s \wedge \cdot}
&
\L_{0,0}.
}
}
\end{equation}
From elementary multi-linear algebra, one has the isomorphisms
\[
\left\{
\begin{array}{ll}
\bigwedge^{2}\E^{\mathbi{d}}
\simeq
\L_{1, 2d_{1}-1}
\oplus
\L_{1, 2d_{2}-1}
\oplus
(\E^{\otimes 2})_{0, \abs{\mathbi{d}}-2}
\\
\bigwedge^{3}\E^{\mathbi{d}}
\simeq
\E_{1, d_{1}+2d_{2}-2}
\oplus
\E_{1, d_{2}+2d_{1}-2}
\end{array}
\right..
\]
Let \(m \in \N_{\geq 2}, n \in \Z\), and twist the above complex \eqref{Koszul complex 2: ci codim 2} by \(\L_{m,n}\).
Using the isomorphisms given in Propositions \ref{prop: push 1} and \ref{prop: push 2}, we see that the (twisted) push-forward Koszul complex \(\pi_{*}\big(\mathcal{K}(\mathbi{s})_{m,n}\big)\) takes the following form:
\begin{equation}
\label{pushed Koszul complex 2: ci codim 2}
\resizebox{\displaywidth}{!}{
\xymatrix{
\Ker \delta_{m-2,n-2\abs{\mathbi{d}}+2}
\ar@{^{(}->}[r]^-{(f_{11}, f_{12})}
&
\Ker \delta^{2}_{m-1, n-d_{1}-2d_{2}+1}
\oplus 
\Ker \delta^{2}_{m-1, n-d_{2}-2d_{1}+1}
\ar@{-}[r]^-{(f_{2i})_{1 \leq i \leq 4}}
&
\
\\
\
\ar[r]
&
(\bigoplus_{i=1}^{2}\Ker \delta_{m-1,n-2d_{i}+1})
 \oplus 
 \Ker \delta^{3}_{m, n-\abs{\mathbi{d}}}
 \oplus
 \Ker \delta_{m-1,n-\abs{\mathbi{d}}+1}
\ar@{-}[r]
&
\
\\
\
\ar[r]^-{(f_{31}, f_{32})}
&
\Ker \delta^{2}_{m, n-d_{1}}
\oplus 
\Ker \delta^{2}_{m, n-d_{2}}
\ar[r]^-{f_{4}}
&
\Ker \delta_{m,n},
}
}
\end{equation}
and it provides a locally free resolution of 
\(S^{m}\Omega_{X}(m+n)\).

In order to have a complete picture of the push-forward Koszul complex, it remains to describe the arrows. Recall first that for any polynomial \(P \in \C[X]\), we denoted:
\[
\left\{ 
\begin{array}{ll}
\alpha(P)= \cdot \times (\diff P)_{X}(Y)
\\
\beta(P)=\cdot \times P - \alpha(P) \circ \delta
\end{array}
\right..
\]
It is a (tedious) verification, that we leave to the reader\footnote{It is similar to the verification made in the case of hypersurfaces, but more involved.}, to check that the maps in the complex \eqref{pushed Koszul complex 2: ci codim 2} are the following:
\begin{equation}
\label{eq: arrows cod 2}
\left\{
\begin{array}{ll}
f_{1i}=\alpha(P_{i}), 1 \leq i \leq 2;
\\
f_{21}(A,B)=\beta(P_{2})(B) \ \text{and} \ f_{22}(A,B)=\beta(P_{1})(A);
\\
f_{23}(A,B)=\alpha(P_{2})(A)-\alpha(P_{1})(B);
\\
f_{24}(A,B)=\beta(P_{2})(A)+\beta(P_{1})(B);
\\
f_{31}(A,B,C,D)=\alpha(P_{1})(A) + \frac{1}{2}\big(\beta(P_{2})(C)+P_{2}C - \alpha(P_{2})(D)\big);
\\
f_{32}(A,B,C,D)=\alpha(P_{2})(B) - \frac{1}{2}\big(\beta(P_{1})(C) +P_{1}C + \alpha(P_{1})(D)\big);
\\
f_{4}(A,B)=\beta(P_{1})(A) + \beta(P_{2})(B).
\end{array}
\right.
\end{equation}

We can thus summarize the main result of this section in the following statement:
\begin{theorem}
\label{thm: resol ci}
The complex
\[
\xymatrix{
0 
\ar[r]
&
\Ker \delta[-2,-2\abs{\mathbi{d}}+2]
\ar@{-}[r]
&
\\
\ar[r]
&
\Ker \delta^{2}[-1,-2\abs{\mathbi{d}}+d_{2}+1]
\oplus 
\Ker \delta^{2}[-1,-2\abs{\mathbi{d}}+d_{1}+1]
\ar@{-}[r]
&
\\
\ar[r]
&
(\underset{i}{\oplus}\Ker \delta[-1,-2d_{i}+1])
 \oplus 
 \Ker \delta^{3}[0,-\abs{\mathbi{d}}]
 \oplus
  \Ker \delta[-1,-\abs{\mathbi{d}}+1]
  \ar@{-}[r]
  &
\\
\ar[r]
&
\Ker \delta^{2}[0, -d_{1}]
\oplus 
\Ker \delta^{2}[0, -d_{2}]
\ar[r]
&
\Ker \delta,
}
\]
whose arrows are given in \eqref{eq: arrows cod 2}, provides a locally free bi-graded resolution of 
\[
\bigoplus_{m \geq 2, n \in \Z}
S^{m}\Omega_{X}(m+n).
\]
\end{theorem}

\section{Cohomology of twisted symmetric powers of cotangent bundles of complete intersections.}
\label{sect: coho sym power ci}

\subsection{A first complex computing cohomology: the case of smooth hypersurfaces.}
\label{sect: coho hyp 1}

For this whole part, we fix \(H\bydef \Set{P=0} \subset \P^{N}\) a smooth hypersurface of degree \(d \geq 1\).

\subsubsection{A complex computing the cohomology}
In order to compute cohomology of (negatively) twisted symmetric powers of cotangent bundles of hypersurfaces, note that as soon as the following inequality is satisfied
\[
(*)
\
\
\
n < -1,
\]
the cohomology of the two terms in the Koszul complex \(\mathcal{K}(\alpha(P))\) (see Theorem \ref{thm: Koszul cx hyp}) is supported in maximal degree, i.e. in degree \(N-1\). Indeed, this follows from the short exact sequence
\[
\xymatrix{
0 
\ar[r]
&
\Ker \delta_{r,s-d}
\ar[r]^-{\cdot \times P}
&
\Ker \delta_{r,s}
\ar[r]
&
(\Ker \delta_{r,s})_{\vert H}
\ar[r]
&
0,
}
\]
combined with Theorem \ref{thm: coho sym proj}, once one has taken the long exact sequence in cohomology. Recall indeed that Theorem \ref{thm: coho sym proj} implies that the cohomology of 
\[
\Ker \delta_{r,s}\simeq S^{r}\Omega_{\P^{N}}(r+s)
\] 
is supported in degree \(N\) as soon as \(s < -1\).

For the rest of this section, we suppose that the condition \( (*) \) is satisfied, so that the cohomology of the terms in the Koszul complex \(\mathcal{K}(\alpha(P))\) is supported in maximal degree. An easy exercice in cohomological algebra allows then to deduce that the cohomology of \(S^{m}\Omega_{H}(m+n)\) is read off through the following complex:
\begin{equation}
\label{eq: coho hyp}
\xymatrix{
H^{N-1}(H, (\Ker \delta_{m-1,n-d+1})_{\vert H})
\ar[rrr]^-{\alpha(P)}
&
&
&
H^{N-1}(H, (\Ker \delta_{m,n})_{\vert H}).
}
\end{equation}
Namely, the kernel is isomorphic to \(H^{N-2}(H, S^{m}\Omega_{H}(m+n))\), and the cokernel is isomorphic to \(H^{N-1}(H, S^{m}\Omega_{H}(m+n))\).

\subsubsection{Reformulation of the complex \eqref{eq: coho hyp} using Serre duality}
\label{subs: reformulation}
Let us recall that for any \(i \geq N+1\), the \(N\)th cohomological group
\(
H^{N}(\P^{N}, \O_{\P^{N}}(-i))
\)
has for natural basis the set of Laurent polynomials
\(
\big(\frac{1}{X^{\mathbi{\gamma}}}\big)_{\mathbi{\gamma}}
\)
with \(\mathbi{\gamma} \in \N^{N+1}\) satifying \(\gamma_{j} \geq 1\) and \(\abs{\mathbi{\gamma}}=i\). 
One sees that there is therefore a natural  identification between \(H^{N}(\P^{N}, \O_{\P^{N}}(-i))\) and \(H^{0}(\P^{N}, \O_{\P^{N}}(i-(N+1)))\) (or rather its dual): this is the very first manifestation of Serre duality, and it takes the following form:
\[
\left(
\begin{array}{ccc}
  H^{N}(\P^{N}, \O_{\P^{N}}(-i))& \longrightarrow  & H^{0}(\P^{N}, \O_{\P^{N}}(i-(N+1)))^{\vee}\\
  \frac{1}{X^{\mathbi{\gamma}} }& \longmapsto  &  (X^{\mathbi{\gamma}-(1, \dotsc, 1)})^{*}
  \end{array}
  \right),
\]
where by definition, for any \(\mathbi{\alpha} \in \N^{N+1}\), the symbol \((X^{\mathbi{\alpha}})^{*}\) denotes the linear functional on \(\C[X]\) taking the value \(1\) on \(X^{\mathbi{\alpha}}\) and zero elsewhere.
We have then the following elementary lemma:
\begin{lemma}
\label{lemma: Serre duality}
Let \(i \in \N\), and let \(P\) be an homogenous polynomial of degree \(d\) in \(N+1\) variables.
Under the identification given by Serre duality, the dual of the map in \(N\)th cohomology induced by the multiplication map
\[
\O_{\P^{N}}(-i) \overset{\cdot \times P}{\longrightarrow} \O_{\P^{N}}(-i+d)
\]
is simply the multiplication by \(P\):
\[
\left(
\begin{array}{ccc}
  H^{0}(\P^{N}, \O_{\P^{N}}(i-d-(N+1)))& \longrightarrow  & H^{0}(\P^{N}, \O_{\P^{N}}(i-(N+1)))   
  \\
  Q & \longmapsto  &   P\times Q
\end{array}
\right).
\]
\end{lemma}

\begin{proof}
By linearity, it suffices to check the statement for a monomial \(X^{\mathbi{\alpha}}\), \(\abs{\mathbi{\alpha}}=d\). 
 Let \(\mathbi{\beta} \in \N^{N+1}\), with \(\beta_{j} \geq 1\) and \(\abs{\mathbi{\beta}}=i\). Then one has that, at the \(N\)th cohomological level,
\[
X^{\mathbi{\alpha}} \cdot \frac{1}{X^{\mathbi{\beta}}}
=
\left\{ 
\begin{array}{ll}
0 \ \ \text{if \(\beta_{j} \leq \alpha_{j}\) for some \(j\)} \\
\frac{1}{X^{\mathbi{\beta}-\mathbi{\alpha}}} \ \ \text{otherwise}
\end{array}
\right..
\]
Under the isomorphism given by Serre duality, this says that for \((X^{\mathbi{\beta}})^{*} \in H^{0}(\P^{N}, \O_{\P^{N}}(i-(N+1)))^{\vee}\), the action of the multiplication by the monomial \(X^{\mathbi{\alpha}}\) is the following
\[
 X^{\mathbi{\alpha}} \cdot (X^{\mathbi{\beta}})^{*}
=
\left\{ 
\begin{array}{ll}
0 \ \ \text{if \(\beta_{j} < \alpha_{j}\) for some \(j\)} \\
(X^{\mathbi{\beta}-\mathbi{\alpha}})^{*} \ \ \text{otherwise}
\end{array}
\right..
\]

Now, let \(X^{\mathbi{\gamma}} \in H^{0}(\P^{N}, \O_{\P^{N}}(i-d-(N+1))))\) and \((X^{\mathbi{\beta}})^{*} \in H^{0}(\P^{N}, \O_{\P^{N}}(i-(N+1)))^{\vee}\). The dual map of the multiplication by \(X^{\mathbi{\alpha}}\) acts as follows:
\begin{eqnarray*}
\big((\cdot \times X^{\mathbi{\alpha}})^{*}(X^{\mathbi{\gamma}})\big)((X^{\mathbi{\beta}})^{*})
&
=
&
(X^{\mathbi{\alpha}} \cdot (X^{\mathbi{\beta}})^{*})(X^{\mathbi{\gamma}})
\\
&
=
&
\left\{ 
\begin{array}{ll}
1 \ \ \text{if \(\mathbi{\beta}=\mathbi{\alpha}+\mathbi{\gamma}\)}
\\
0 \ \ \text{otherwise}
\end{array}
\right..
\end{eqnarray*}
Therefore, one indeed has that \((\cdot \times X^{\mathbi{\alpha}})^{*}(X^{\mathbi{\gamma}})=X^{\mathbi{\alpha}+\mathbi{\gamma}}\), which proves the lemma.
\end{proof}

With this small reminder on Serre duality on projective spaces, we can now give a convenient description of the spaces appearing in the complex \eqref{eq: coho hyp}:
\begin{lemma}
\label{lemma: reformulation Serre}
Let \(r \in \N\), and \(s \in \Z\). There is a natural isomorphism:
\[
H^{N-1}(H, (\Ker \delta_{r,s})_{\vert H})^{\vee}
\simeq
\frac{S_{r,-s-(N+1)+d}}{(P,q)},
\]
where one recalls that \(S\bydef \C[Y,X]\), and \(q=\sum\limits_{i=0}^{N} X_{i}Y_{i}\).
\end{lemma}
\begin{proof}
Via the short exact sequence
\[
\xymatrix{
0 
\ar[r]
&
\Ker \delta_{r,s-d}
\ar[r]^-{\cdot \times P}
&
\Ker \delta_{r,s}
\ar[r]
&
(\Ker \delta_{r,s})_{\vert H}
\ar[r]
&
0,
}
\]
and Theorem \ref{thm: coho sym proj}, one deduces that \(H^{N-1}(H, (\Ker \delta_{r,s})_{\vert H})\) is isomorphic to 
\[
\Ker \Big( H^{N}(\P^{N},\Ker \delta_{r,s-d}) \overset{\cdot \times P}{\longrightarrow} H^{N}(\P^{N},\Ker \delta_{r,s}) \Big).
\]
On the other hand, for any \(r \geq 0\) and any \(s \in \Z\),  the vector space \(H^{N}(\P^{N},\Ker \delta_{r,s})\) is isomorphic to
\[
\Ker \Big( 
\C[Y]_{r} \otimes H^{N}(\P^{N}, \O_{\P^{N}}(s))
\overset{\delta}{\longrightarrow}
\C[Y]_{r-1} \otimes H^{N}(\P^{N}, \O_{\P^{N}}(s+1)) 
\Big).
\]
One can thus form the following commutative diagram
\begin{equation}
\label{eq: comm diagram}
\resizebox{\displaywidth}{!}{
\xymatrix{
&
&
0
&
0
\\
&
&
\C[Y]_{r-1} \otimes H^{N}(\P^{N}, \O_{\P^{N}}(s-d+1)) 
\ar[u]
&
\C[Y]_{r-1} \otimes H^{N}(\P^{N}, \O_{\P^{N}}(s+1)) 
\ar[u]
\\
&
&
\C[Y]_{r} \otimes H^{N}(\P^{N}, \O_{\P^{N}}(s-d)) 
\ar[r]^-{\cdot \times P}
\ar[u]^-{\delta}
&
\C[Y]_{r} \otimes H^{N}(\P^{N}, \O_{\P^{N}}(s)) 
\ar[u]^-{\delta}
\\
0 
\ar[r]
&
H^{N-1}(H, (\Ker\delta_{r,s})_{\vert H})
\ar[r]
&
H^{N}(\P^{N}, \Ker \delta_{r,s-d})
\ar[u]
\ar[r]
&
H^{N}(\P^{N},\Ker \delta_{r,s})
\ar[u]
\\
&
&
0
\ar[u]
&
0
\ar[u]
}
}
\end{equation}
where the columns and lines are exact. 

Now, by dualizing everything, using Serre duality and Lemma \ref{lemma: Serre duality}, the above diagram becomes:
\[
\resizebox{\displaywidth}{!}{
\xymatrix{
&
&
0
\ar[d]
&
0
\ar[d]
\\
&
&
\C[Y]_{r-1} \otimes H^{0}(\P^{N}, \O_{\P^{N}}(-s-(N+1)+d-1)) 
\ar[d]^-{\delta^{*}}
&
\C[Y]_{r-1} \otimes H^{0}(\P^{N}, \O_{\P^{N}}(-s-1-(N+1))) 
\ar[d]^-{\delta^{*}}
\\
&
&
\C[Y]_{r} \otimes H^{0}(\P^{N}, \O_{\P^{N}}(-s-(N+1)+d)) 
\ar[d]
&
\C[Y]_{r} \otimes H^{0}(\P^{N}, \O_{\P^{N}}(-s-(N+1))) 
\ar[d]
\ar[l]^-{\cdot \times P}
\\
0 
&
H^{N-1}(H,(\Ker \delta_{r,s})_{\vert H})^{\vee}
\ar[l]
&
H^{N}(\P^{N},\Ker \delta_{r,s-d})^{\vee}
\ar[d]
\ar[l]
&
H^{N}(\P^{N},\Ker \delta_{r,s})^{\vee}
\ar[l]
\ar[d]
\\
&
&
0
&
0
}
}
\]
As in Section \ref{subs: dual Euler}, use the renormalization map
\[
u \colon
\left(
\begin{array}{ccc}
  \C[Y] & \longrightarrow  & \C[Y]   \\
   Y^{\mathbi{\beta}} & \longmapsto  &  \frac{Y^{\mathbi{\beta}}}{\mathbi{\beta!}}
\end{array}
\right)
\]
to change the two columns in the above diagram, so that it becomes:
\[
\xymatrix{
&
&
0
\ar[d]
&
0
\ar[d]
\\
&
&
S_{r-1,-s-1-(N+1)+d}
\ar[d]^-{\cdot \times q}
&
S_{r-1,-s-1-(N+1)}
\ar[d]^-{\cdot \times q}
\\
&
&
S_{r, -s-(N+1)+d}
\ar[d]
&
S_{r,-s-(N+1)}
\ar[d]
\ar[l]^-{\cdot \times P}
\\
0 
&
H^{N-1}(H, (\Ker \delta_{r,s})_{\vert H})^{\vee}
\ar[l]
&
H^{N}(\P^{N},\Ker \delta_{r,s-d})^{\vee}
\ar[d]
\ar[l]
&
H^{N}(\P^{N},\Ker \delta_{r,s})^{\vee}
\ar[l]
\ar[d]
\\
&
&
0
&
0
}
\]
The lemma now follows immediately from the above diagram.
\end{proof}
Let us now use dualize the complex \eqref{eq: coho hyp}. Using the above Lemma \ref{lemma: reformulation Serre}, it becomes:
\begin{equation}
\label{eq: coho hyp 2}
\xymatrix{
\frac{S_{m, -n-(N+1)+d}}{(P,q)}
\ar[rr]^-{\alpha^{*}(P)}
&
&
\frac{S_{m-1, -n-(N+1)+2d-1}}{(P,q)}.
}
\end{equation}
It remains to understand the induced map, that we denoted by \(\alpha^{*}(P)\):
\begin{lemma}
\label{lemma: Serre map}
The map \(\alpha^{*}(P)\) is induced by the map (still denoted \(\alpha^{*}(P)\) by a slight abuse of notations)
\[
\alpha^{*}(P) \colon\left(
\begin{array}{ccc}
 S & \longrightarrow & S
  \\
 A & \longmapsto  &  \frac{1}{d}(\sum\limits_{i=0}^{N} \frac{\partial P}{\partial X_{i}} \frac{\partial}{\partial Y_{i}})(A)
\end{array}
\right),
\]
where one recalls that \(S\bydef \C[Y,X]\).
\end{lemma}
\begin{remark}
In other words, up to a non-zero multiplicative factor, this is the map induced by \(\delta(\gamma(P))(\P^{N})\), where \(\gamma(P)\) is the Gauss map associated to the hypersurface \(H= \Set{P=0}\) (see Section \ref{subs: a particular class} for notations).
\end{remark}
\begin{proof}
Consider the commutative diagram \eqref{eq: comm diagram} introduced in the previous Lemma \ref{lemma: reformulation Serre}, which provides a convenient description of \(H^{N-1}(H, (\Ker \delta_{r,s})_{\vert H})\). 
The map induced (at the cohomological level) by the multiplication by \(\frac{1}{d}(\diff P)_{X}(Y)\)
\[
\xymatrix{
H^{N-1}(H, (\Ker\delta_{r,s})_{\vert H})
\ar[rrr]^-{\alpha(P)}
&
&
&
H^{N-1}(H,  (\Ker\delta_{r+1,s+d-1})_{\vert H})
}
\]
is induced by the map
\[
\xymatrix{
\C[Y]_{r}\otimes H^{N}(\P^{N}, \O_{\P^{N}}(s-d))
\ar[rr]^-{\alpha(P)}
&
&
\C[Y]_{r+1}\otimes H^{N}(\P^{N}, \O_{\P^{N}}(s-1)).
}
\]
This map fits into the following commutative diagram:
\[
\resizebox{\displaywidth}{!}{
\xymatrix{
\C[Y]_{r}\otimes H^{N}(\P^{N}, \O_{\P^{N}}(s-d))
\ar[rrr]^-{\alpha(P)}
\ar[d]^-{\text{Serre duality}}
&
&
&
\C[Y]_{r+1}\otimes H^{N}(\P^{N}, \O_{\P^{N}}(s-1))
\ar[d]^-{\text{Serre duality}}
\\
\C[Y]_{r}\otimes H^{0}(\P^{N}, \O_{\P^{N}}(-s-(N+1)+d))^{\vee}
\ar[d]^-{u \ (\text{renormalization map})}
\ar[rrr]
&
&
&
\C[Y]_{r+1}\otimes H^{0}(\P^{N}, \O_{\P^{N}}(-s-N))^{\vee}
\ar[d]^-{u \ (\text{renormalization map})}
\\
\C[Y]_{r}\otimes H^{0}(\P^{N}, \O_{\P^{N}}(-s-(N+1)+d))^{\vee}
\ar[rrr]^-{m(P)}
&
&
&
\C[Y]_{r+1}\otimes H^{0}(\P^{N}, \O_{\P^{N}}(-s-N))^{\vee}.
}
}
\]
It is thus enough to understand the dual of the map \(m(P)\), and the goal is to show the following equality:
\[
m(P)^{*}
=
\alpha^{*}(P).
\]

Note that the correspondances \(P \mapsto m(P)\) and \(P \mapsto \alpha^{*}(P)\) are linear in \(P\), so that it suffices to prove this equality for monomials in \(\C[X]_{d}\). The proof now follows from Lemma \ref{lemma: Serre duality}, and the same computation as the one carried over in Section \ref{subs: dual Euler} on the dualized generalized Euler exact sequence.

\end{proof}

\begin{remark}
\label{remark: recipee}
More generally, under the Serre duality isomorphism and the renormalization map, the action of the multiplication by a polynomial \(Q(Y,X)\) on the top cohomology group corresponds to the partial differential equation \(Q(\frac{\partial}{\partial Y}, X)\) on the \(0\)th cohomology group. This is the content of the proof of the previous Lemma \ref{lemma: Serre map}
\end{remark}

We have therefore shown the following theorem:
\begin{theorem}
\label{thm: coho hyp 1}
The cohomology of 
\(
\bigoplus_{m \geq 1, n > 1 }
S^{m}\Omega_{H}(m-n)
\)
can be computed via the graded (Koszul) complex \(K(\alpha^{*}(P))[-1, -(N+1)+2d-1]\):
\[
\xymatrix{
\frac{S}{(P,q)}[0, -(N+1)+d]
\ar[rr]^-{\alpha^{*}(P)}
&
&
\frac{S}{(P,q)}[-1, -(N+1)+2d-1].
}
\]
Namely, the \(i\)th cohomology group of one graded component of 
\(
\bigoplus_{m \geq 1, n > 1 }
S^{m}\Omega_{H}(m-n)
\) 
is isomorphic to the \((N-1-i)\)th cohomology group of the corresponding graded part of the graded complex.
\end{theorem}

\subsection{A first complex computing cohomology: the case of smooth complete intersections.}
\label{sect: coho ci 1}
The same line of reasoning works for smooth complete intersections. Keeping the notations of Section \ref{sect: Koszul complex 1: ci}, consider the push-forward twisted Koszul complex
\[
(\pi_{X})_{*}\big(\mathcal{K}(\mathbi{s})_{m,n}\big),
\]
which provides a locally free resolution of \(S^{m}\Omega_{X}(m+n)\) for \(m \geq c \) and any \(n \in \Z\). 
Observe that the pieces in the locally free resolution \((\pi_{X})_{*}\big(\mathcal{K}(\mathbi{s})_{m,n}\big)\) all write as a direct sum of terms of the following form:
\[
(\Ker \delta_{r,s})_{\vert X},
\]
where \(r \leq m\) and \(s \leq n\). 

Suppose that the condition \((*)\) holds, i.e. that the following inequality holds:
\[
n  < -1.
\]
By twisting by \(\Ker \delta_{r,s}\) the Koszul complex 
\[
\mathcal{K}(P_{1}, \dotsc, P_{c})
\] 
resolving \(\O_{X}\) (recall that \(P_{1}, \dotsc, P_{c}\) are the homogeneous polynomials defining the complete intersection \(X\)), one obtains a locally free resolution of \((\Ker \delta_{r,s})_{\vert X}\). Under the condition \((*)\), for any \(r \leq m\) and any \(s \leq n\), the cohomology of the pieces in the locally free complex resolving \((\Ker \delta_{r,s})_{\vert X}\)
\[
\mathcal{K}(P_{1},\dotsc, P_{c}) \otimes \Ker \delta_{r,s}
\]
is supported in degree \(N\). This allows to show that the cohomology of \((\Ker \delta_{r,s})_{\vert X}\) is supported in degree \( \geq N-c \). On the other hand, since \(\dim(X)=N-c\), the cohomology of \((\Ker \delta_{r,s})_{\vert X}\) is supported in degree less or equal than \(N-c\). 

Therefore, the cohomology of the pieces appearing in the resolution \((\pi_{X})_{*}\big(\mathcal{K}(\mathbi{s})_{m,n}\big)\) is supported in maximal degree, i.e. in degree \(N-c=\dim(X)\).
Accordingly, we deduce that the cohomology of \(S^{m}\Omega_{X}(m+n)\) is read off through the complex obtained by applying the functor \(H^{N-c}(X, \cdot)\) to the locally free resolution \((\pi_{X})_{*}\big(\mathcal{K}(\mathbi{s})_{m,n}\big)\).

Using Serre duality as in the case of hypersurfaces, we can reformulate this complex as follows. The analogous of Lemma \ref{lemma: reformulation Serre} takes the following form:
\begin{lemma}
\label{lemma: reformulation Serre 2}
Let \(r \in \N\), and \(s \in \Z\). There is a natural isomorphism:
\[
H^{N-c}(X, (\Ker \delta_{r,s})_{\vert X})^{\vee}
\simeq
\frac{S_{r,-s-(N+1)+\abs{\mathbi{d}}}}{(P_{1}, \dotsc, P_{c},q)}.
\]
\end{lemma}
\begin{proof}
The proof is completely similar the one of Lemma \ref{lemma: reformulation Serre}. Namely, consider the resolution \(\mathcal{K}(P_{1}, \dotsc, P_{c})\otimes \Ker \delta_{r,s}\) of \((\Ker \delta_{r,s})_{\vert X}\), which allows to give a convenient description of \(H^{N-c}(X,(\Ker \delta_{r,s})_{\vert X})\). By dualizing this description, using Serre duality and Lemma \ref{lemma: Serre duality}, the result follows.
\end{proof}

As in the case of hypersurfaces, the maps in the dualized complex 
\[
H^{N-c}
\Big(
X, (\pi_{X})_{*}\big(\mathcal{K}(s_{1}, \dotsc, s_{c})(m,n)\big)
\Big)^{\vee}
\]
 are induced (under the isomorphism of Lemma \ref{lemma: reformulation Serre 2}) by the maps
 \(
 \alpha^{*}(P_{1}), \dotsc, \alpha^{*}(P_{c}).
 \)
\begin{example}
For \(c=2\), the complex 
\(H^{N-c}
\Big(
X, (\pi_{X})_{*}\big(\mathcal{K}(\mathbi{s})_{m,n}\big)
\Big)^{\vee}
\)
 takes the following form:
\begin{equation*}
\resizebox{\displaywidth}{!}{
\xymatrix{
\frac{S_{m, -n-(N+1)+\abs{\mathbi{d}}}}{(P_{1}, P_{2}, q)}
\ar[rr]^-{(\alpha^{*}(P_{1}), \alpha^{*}(P_{2}))}
&
&
\frac{S_{m-1,-n-(N+1)+2d_{1}+d_{2}-1}}{(P_{1}, P_{2}, q)}
\oplus
\frac{S_{m-1,-n-(N+1)+d_{1}+2d_{2}-1}}{(P_{1}, P_{2}, q)}
\ar@{-}[r]
&
\\
\ar[rr]^-{\alpha^{*}(P_{2})(\cdot)-\alpha^{*}(P_{1})(\cdot)}
&
&
\frac{S_{m-2,-n-(N+1)+2\abs{\mathbi{d}}-2}}{(P_{1}, P_{2},q)}.
}
}
\end{equation*}
Recall that the cohomology of this complex computes the cohomology of \(S^{m}\Omega_{X}(m+n)\) as soon as \(n < -1\) and \(m \geq c=2\). Namely:
\begin{itemize}
\item{}
the first cohomology group in the above complex is isomorphic to \(H^{N-2}(X, S^{m}\Omega_{X}(m+n))\);
\item{}
the middle cohomology group in the above complex is isomorphic to \(H^{N-3}(X, S^{m}\Omega_{X}(m+n))\);
\item{}
the last cohomology group in the above complex is isomorphic to \(H^{N-4}(X, S^{m}\Omega_{X}(m+n))\),
\end{itemize}
and all the other cohomology groups of the vector bundle \(S^{m}\Omega_{X}(m+n)\) are zero.
\end{example}

More generally, note that for any \(1 \leq i \leq c\), the map \(\alpha^{*}(P_{i})\) induces a map of bigraded algebra:
\[
\alpha^{*}(P_{i}) \colon
\frac{S}{(P_{1}, \dotsc, P_{c},q)}[-1,-d_{i}+1]
\longrightarrow \frac{S}{(P_{1}, \dotsc, P_{c}, q)}.
\]
We can therefore consider the Koszul complex
\[
K(\alpha^{*}(P_{1}), \dotsc, \alpha^{*}(P_{c})),
\]
which, by what we just saw, computes the cohomology of \(S^{m}\Omega_{X}(m+n)\) for \(m \geq c\), and \(n < -1\):
\begin{theorem}
\label{thm: coho ci 1}
The cohomology of 
\(
\bigoplus_{m \geq c,  n > 1} S^{m}\Omega_{X}(m-n)
\)
can be computed via the Koszul complex on \(\frac{\C[Y,X]}{(P_{1}, \dotsc, P_{c}, q)}\):
\[
K(\alpha^{*}(P_{1}), \dotsc, \alpha^{*}(P_{c}))[-c, -(N+1)+2\abs{\mathbi{d}}-c].
\]
Namely, the \(i\)th cohomology group of one graded component of 
\(
\bigoplus_{m \geq 1, n > 1 }
S^{m}\Omega_{X}(m-n)
\) 
is isomorphic to the \((N-c-i)\)th cohomology group of the corresponding graded part of the Koszul complex.
\end{theorem}

\subsection{A second complex computing cohomology: the case of smooth hypersurfaces.}
\label{sect: coho hyp 2}
We consider the same setting as in Section \ref{sect: coho hyp 1} : we let \(H=\Set{P=0} \subset \P^{N}\) be a smooth hypersurface, and fix \(m \in \N_{\geq 1}\) and \(n \in \Z\)  two natural numbers satisfying the following condition:
\[
(*)
\
\
\
n < -1.
\]
This time, instead of considering the first resolution given in Section \ref{section: Koszul complex 1: hyp}, we use the second resolution given in Section \ref{section: Koszul complex 2: hyp}. 
The reasoning is then completely similar to the one carried over in Section \ref{sect: coho hyp 1}. Namely, we first apply the functor \(H^{N}(\P^{N}, \cdot)\) to the resolution \eqref{eq: resolution hyp pushed} of \(S^{m}\Omega_{X}(m+n)\), which we recall here:
\[
\xymatrix{
\pi_{*}(\mathcal{K}(s)_{m,n})\colon
0 \ar[r] 
&
\Ker \delta_{m-1,n-2d+1}
\ar[r] 
&
\pi_{*}(\E_{m-1,n-d} )
\ar[r]
&
\Ker\delta_{m,n}.
}
\] 
Let us justify that, under the condition \((*)\), the complex \(H^{N}(\P^{N},\pi_{*}(\mathcal{K}(s)_{m,n}))\) does indeed compute the cohomology. This will follow from the following elementary lemma:
\begin{lemma}
\label{lemma: support}
Let \(m \in \Z, n \in \Z\) and \(k \in \N\). The cohomology of \((\E^{\otimes k})_{m,n}\) is supported in degree \(0, 1\) and \(N\). Furthermore, if one supposes that \(n < -1 -k\), then the cohomology is supported in degree \(N\).
\end{lemma}

\begin{proof}
Proceed by induction on \(k\). For \(k=0\), all statements follow immediately from Bott's formula (see Appendix \ref{appendix: Bott}). Suppose that the results holds for \(k-1 \geq 0\). By Proposition \ref{prop: extension}, one deduces that the vector bundle \((\E^{\otimes k})_{m,n}\) fits into the short exact sequence:
\begin{equation}
\label{eq: fit short exact}
\xymatrix{
0
\ar[r]
&
(\E^{\otimes(k-1)})_{m+1,n}
\ar[r]
&
(\E^{\otimes k})_{m,n}
\ar[r]
&
(\E^{\otimes(k-1)})_{m,n+1}
\ar[r]
&
0.
}
\end{equation}
This immediately implies that the cohomology of \((\E^{\otimes k})_{m,n}\) remains supported in degree \(0\), \(1\) and \(N\).

Suppose now that \(n < -1-k\). By induction hypothesis, the cohomology of the two extremities of the short exact sequence \eqref{eq: fit short exact} is supported in degree \(N\). Therefore, so does the cohomology of the middle term of the short exact sequence: this finishes the proof.
\end{proof}
For instance, if we consider the term \(\pi_{*}(\E_{m-1,n-d} )\) in the complex \(\pi_{*}(\mathcal{K}(s)_{m,n})\), the above Lemma \ref{lemma: support} (combined with Bott's formula) implies that its cohomology is indeed supported in maximal degree, since by hypothesis one has \(n-d \leq n-1 < -2\). The two other terms in the complex are treated similarly.

Now, the analogue of Lemma \ref{lemma: reformulation Serre} is the following:
\begin{lemma}
\label{lemma: reformulation ci}
Let \(r \in \N\), \(s \in \Z\). Then one has the following isomorphism\footnote{Recall that throughout the paper (see Conventions), \(N\) is supposed to be greater or equal than \(2\).}
\[
H^{N}(\P^{N}, \Ker \delta^{k}_{r,s})^{\vee}
\simeq 
\Big(
\frac{S}{(q^{k})}
\Big)_{r,-s-(N+1)}
\]
\end{lemma}
\begin{proof}
Consider the short exact sequence
\[
\xymatrix{
0 
\ar[r]
&
\Ker \delta^{k}_{r,s}
\ar[r]
&
\C[Y]_{r}\otimes \O_{\P^{N}}(s)
\ar[r]^-{\delta^{k}}
&
C[Y]_{r-k}\otimes \O_{\P^{N}}(s+k)
\ar[r]
&
0.
}
\]
Taking the long exact sequence in cohomology, one sees that
\[
H^{N}(\P^{N}, \Ker \delta^{k}_{r,s})
\simeq
\Ker
\Big(
\C[Y]_{r}\otimes H^{N}(\P^{N}, \O_{\P^{N}}(s))
\overset{\delta^{k}}{\longrightarrow}
\C[Y]_{r-k}\otimes H^{N}(\P^{N}, \O_{\P^{N}}(s+k))
\Big).
\]
Using Serre duality, and the renormalization map \(u\) (see Section \ref{subs: dual Euler}), one deduces the lemma.\end{proof}
With the previous lemma, the complex \(H^{N}(\P^{N}, \mathcal{K}(s)(m,n))^{\vee}\) writes as follows\footnote{The injectivity of the map \(\beta^{*}(P)\) can be checked directly, but it follows from the general fact that the cohomology of \(S^{m}\Omega_{H}(m+n)\) is supported in degree less or equal than \(N-1=\dim(H)\).}
\[
\xymatrix{
0
\ar[r]
&
\frac{S_{m,-n-(N+1)}}{(q)}
\ar[rr]^-{\beta^{*}(P)}
&
&
\frac{S_{m,-n-(N+1)+d}}{(q^{2})}
\ar[rr]^-{\alpha^{*}(P)}
&
&
\frac{S_{m-1,-n-(N+1)+2d-1}}{(q)}.
}
\]
Following the same reasoning as in Lemma \ref{lemma: Serre map}, one easily shows that the maps \(\beta^{*}(P)\) and \(\alpha^{*}(P)\) are the following:
\[
\left\{
\begin{array}{ll}
\alpha^{*}(P)(\cdot)= \frac{1}{d} \sum\limits \frac{\partial P}{\partial X_{i}} \frac{\partial}{\partial Y_{i}}(\cdot);
\\
\beta^{*}(P)(\cdot)= \cdot \times P - q \times \alpha^{*}(P)(\cdot).
\end{array}
\right.
\]
The main result of this section can then be stated as follows:
\begin{theorem}
\label{thm: coho hyp 2}
The cohomology of 
\(
\bigoplus_{m \geq 1, n > 1 }
S^{m}\Omega_{H}(m-n)
\)
can be computed via the graded complex
\[
\xymatrix{
\frac{S}{(q)}[0, -(N+1)]
\ar[rr]^-{\beta^{*}(P)}
&
&
\frac{S}{(q^{2})}[0,-(N+1)+d]
\ar[rr]^-{\alpha^{*}(P)}
&
&
\frac{S}{q}[-1, -(N+1)+2d-1]
}.
\]
Namely, the \(i\)th cohomology group of one graded component of 
\(
\bigoplus_{m \geq 1, n > 1 }
S^{m}\Omega_{H}(m-n)
\) 
is isomorphic to the \((N-i)\)th cohomology group of the corresponding graded part of the graded complex.

\end{theorem}

\subsection{A second complex computing cohomology: the case of smooth complete intersections of codimension \(2\).}
\label{sect: coho ci 2}
We keep the notations of Section \ref{sect: Koszul complex 2: ci}.
At this point, it must be clear to the reader what complex we are going to obtain from the resolution \eqref{pushed Koszul complex 2: ci codim 2}:
\begin{itemize}
\item{} apply the functor \(H^{N}(\P^{N}, \cdot)\) to \eqref{pushed Koszul complex 2: ci codim 2}; Lemma \ref{lemma: support} allows to show that the complex formed in this way does compute the sought cohomology;

\item{} use Serre duality and the renormalization map, combined with Lemma \ref{lemma: reformulation ci}, in order to provide a reformulation of the complex;

\item{} consider the "recipee" given in Remark \ref{remark: recipee}, namely, after dualization and renormalization
\begin{itemize}

\item{} any multiplication map \(\cdot \times A(Y,X)\) becomes a partial differential equation \(A(\frac{\partial}{\partial Y}, X)\) (with reversed arrow);

\item{} any partial differential equation \(A(\frac{\partial}{\partial Y}, X)\) becomes a multiplication map \(\cdot \times A(Y,X)\) (with reversed arrow, as we dualize);
\end{itemize}
\item{}
apply this recipee to the various maps appearing in the complex \(\eqref{pushed Koszul complex 2: ci codim 2}\).
\end{itemize}
By applying all these steps, we obtain a graded complex computing the cohomology of
\[
\bigoplus_{m \geq 2, n > 1 }
S^{m}\Omega_{X}(m-n),
\]
yielding the following statement:
\begin{theorem}
\label{thm: coho ci 2}
The cohomology of 
\(\bigoplus_{m \geq 2, n > 1 }
S^{m}\Omega_{X}(m-n)
\)
can be computed via the following complex
\begin{equation}
\label{eq: Koszul complex cod 2}
\xymatrix{
\Big(\frac{S}{(q)}
\ar[r]^-{(g_{1i})_{1 \leq i \leq 2}}
&
\frac{S}{(q^{2})}[0, d_{1}]\oplus \frac{S}{(q^{2})}[0, d_{2}]
\ar@{-}[r]^-{(g_{2i})_{1 \leq i \leq 4}}
&
\\
\ar[r]
&
(\oplus_{i}
\frac{S}{(q)}[-1, 2d_{i}-1])
\oplus 
 \frac{S}{(q^{3})}[0, -\abs{\mathbi{d}}]
\oplus \frac{S}{(q)}[-1, \abs{\mathbi{d}}-1]
\ar@{-}[r]
&
\\
\ar[r]^-{(g_{3i})_{1 \leq i \leq 2}}
&
\frac{S}{(q^{2})}[-1, +2\abs{\mathbi{d}}-d_{2}-1]
\oplus \frac{S}{(q^{2})}[-1, 2\abs{\mathbi{d}}+d_{1}-1]
\ar@{-}[r]
&
\\
\ar[r]^-{g_{4}}
&
\frac{S}{(q)}[-2, +2\abs{\mathbi{d}}-2]\Big)[0,-(N+1)],
}
\end{equation}
where one has:
\begin{equation*}
\left\{
\begin{array}{ll}
g_{11}=\beta^{*}(P_{1}) \ \text{and} \ g_{12}=\beta^{*}(P_{2});
\\
g_{21}(A,B)=\alpha^{*}(P_{1})(A) \ \text{and} \ g_{22}(A,B)=\alpha^{*}(P_{2})(B);
\\
g_{23}(A,B)=\frac{1}{2}\big(\beta^{*}(P_{2})(A)+P_{2}A-\beta^{*}(P_{1})(B)-P_{1}B\big);
\\
g_{24}=-\frac{1}{2}\big(\alpha^{*}(P_{2})(A)+\alpha^{*}(P_{1})(B)\big);
\\
g_{31}(A,B,C,D)=\beta^{*}(P_{1})(B)+\alpha^{*}(P_{2})(C)+\beta^{*}(P_{2})(D);
\\
g_{32}(A,B,C,D)=\beta^{*}(P_{2})(A)-\alpha^{*}(P_{1})(C)-\beta^{*}(P_{1})(D);
\\
g_{4}(A,B)=\alpha^{*}(P_{1})(A)+\alpha^{*}(P_{2})(B).
\end{array}
\right.
\end{equation*}
Namely, the \(i\)th cohomology group of one graded component of 
\(
\bigoplus_{m \geq 2, n > 1 }
S^{m}\Omega_{X}(m-n)
\) 
is isomorphic to the \((N-i)\)th cohomology group of the corresponding graded part of the complex \eqref{eq: Koszul complex cod 2}.
\end{theorem}
Let us mention that, as we know that the cohomology of 
\(\bigoplus_{m \geq 1, n > 1 }
S^{m}\Omega_{X}(m-n)\)
is supported in degrees ranging from \(N-4\) to \(N-2\), we are only interested in a truncated part of the complex given in the above theorem (in either way, it is easy to check the exactness of the complex where it is supposed to be).

\section{Applications.}
\label{sect: applications}

\subsection{First application: vanishing  and non-vanishing theorems.}
\label{sect: vanish and non-vanish}
In this section, we use the resolution given in Section \ref{sect: Koszul complex 2: ci} in order to study cohomology of symmetric powers of cotangent bundles of complete intersections. We first recover the known vanishing results of Bruckmann-Rackvitz \cite{BR}. Then, we prove a few non-vanishing results, which in particular show that the previous vanishing results are optimal (which, to our knowledge, is new). It also allows us to generalize results of Bogomolov--DeOliveira \cite{Bog}: see Theorem \ref{cor: generalization} below.

 Keep accordingly the notations of Section \ref{sect: Koszul complex 2: ci}, and recall that the line bundle \(\O_{\P(TX)}(m)\otimes\pi_{X}^{*}\O_{X}(m+n)\) admits the following resolution:
\begin{equation}
\label{eq: resolution}
\resizebox{\displaywidth}{!}{
\xymatrix{
0
\ar[r]
&
\L_{m-c,n+c-2\abs{\mathbi{d}}}
\ar[r]
&
\E^{\mathbi{d}} \otimes \L_{m-c,n+c-2\abs{\mathbi{d}}}
\ar[r]
&
\dotsb
\ar[r]
&
\big(\bigwedge^{2c-1}\E^{\mathbi{d}}\big)\otimes \L_{m-c,n+c-2\abs{\mathbi{d}}}
\ar[r]
&
\L_{m,n}.
}
}
\end{equation}
Cutting this resolution into \(2c-2\) short exact sequences, one gets:
\begin{equation}
\label{eq: cut exact}
\resizebox{\displaywidth}{!}{
\xymatrix{
0
\ar[r]
&
I_{0}
\ar[r]
&
\L_{m,n}
\ar[r]
&
\O_{\P(TX)}(m)\otimes\pi_{X}^{*}\O_{X}(m+n)
\ar[r]
&
0
\\
0
\ar[r]
&
I_{1}
\ar[r]
&
\big(\bigwedge^{2c-1}\E^{\mathbi{d}}\big)_{m-c,n+c-2\abs{\mathbi{d}}}
\ar[r]
&
I_{0}
\ar[r]
&
0
\\
0
\ar[r]
&
I_{2} 
\ar[r]
&
\big(\bigwedge^{2c-2}\E^{\mathbi{d}}\big)_{m-c,n+c-2\abs{\mathbi{d}}}
\ar[r]
&
I_{1}
\ar[r]
&
0
\\
.
\\
.
\\
0
\ar[r]
&
\L_{m-c,n+c-2\abs{\mathbi{d}}}
\ar[r]
&
\big(\E^{\mathbi{d}}\big)_{m-c,n+c-2\abs{\mathbi{d}}}
\ar[r]
&
I_{2c-2}
\ar[r]
&
0.
\\
}
}
\end{equation}
All the results in this Section \ref{sect: vanish and non-vanish} and the next one Section \ref{sect: algebra} will be obtained via diagram chasing, and the use of Lemma \ref{lemma: support}. Note that we have not indicated any index relatively to the dependance on \(m\) and \(n\) for the intermediate coherent sheaves \(I_{0}, \dotsc, I_{2c-2}\) in the cut-out resolution: within the context they are used, they will always be implicit, leading, we believe, to no confusion.

\subsubsection{Vanishing and non-vanishing theorems for \(H^{0}\), in codimension \(c < \frac{N}{2}\).}
\label{sect: vanish and non vanish 1}
In this subsection, let us first give an alternative proof of the known vanishing theorem of Bruckmann-Rackvitz \cite{BR}[Theorem 4 (iii)], in the case of symmetric powers:

\begin{theorem}
\label{thm: application 1.1 vanish}
Let \(X \subset \P^{N}\) be a smooth complete intersection of codimension \(c < \frac{N}{2}\), and multi-degree \(\mathbi{d}=(d_{1}, \dotsc, d_{c})\).  Denote \(d \bydef \min\Set{d_{i} \ | \ 1 \leq i \leq c}\).

For any \(m \in \N, n \in \Z\) satisfying the inequalities \(m>n\) and \(n<d-2\), one has the following vanishing result:
 \[
 H^{0}(X, S^{m}\Omega_{X}(m+n))
 =
 (0).
 \]
\end{theorem}
\begin{proof}
Use the cut-out resolution \eqref{eq: cut exact} of \(\O_{\P(TX)}(m)\otimes\pi_{X}^{*}\O_{X}(m+n)\) given above.
Observe that, since \(n<m\), the line bundle line bundle \(\L_{m,n}\) has no global sections. Therefore, one deduces from the first short exact sequence an injection
\begin{equation}
\label{eq: proof inter}
\xymatrix{
0
\ar[r]
&
H^{0}(X, S^{m}\Omega_{X}(m+n))
\ar[r]
&
H^{1}(\Flag_{(1,2)}\C^{N+1}, I_{0}).
}
\end{equation}
To prove the result, it is thus enough to show that \(H^{1}(\Flag_{(1,2)}\C^{N+1}, I_{0})=(0)\).

Note that the locally free sheaf \(\big(\bigwedge^{2c-1}\E^{\mathbi{d}}\big)_{m-c,n+c-2\abs{\mathbi{d}}}\) writes as a direct sum of the locally free sheaves 
\begin{equation}
\label{eq: identity 1}
\big(\E_{0,d_{i}-1} \otimes(\bigotimes_{j \neq i} \L_{1, 2d_{j}-1})\big) \otimes \L_{m-c,n+c-2\abs{\mathbi{d}}}
\simeq
\E_{m-1,n-d_{i}}
\end{equation}
for \(1 \leq i \leq c\).
By Lemma \ref{lemma: support}, the locally free sheaf \eqref{eq: identity 1} has its cohomology supported in maximal degree as soon as
\[
n-d_{i} < -2.
\]
This holds by hypothesis, and one deduces in particular that
\[
H^{1}\big(\Flag_{(1,2)}\C^{N+1},\big(\bigwedge^{2c-1}\E^{\mathbi{d}}\big)_{m-c,n+c-2\abs{\mathbi{d}}}\big)
=
(0).
\]

Now, one deduces successively from the short exact sequences of \eqref{eq: cut exact} and Lemma \ref{lemma: support} the following chain of isomorphism:
\[
H^{1}(\Flag_{(1,2)}\C^{N+1}, I_{0})
\simeq
H^{2}(\Flag_{(1,2)}\C^{N+1}, I_{1})
\simeq
\dotsb 
\simeq 
H^{2c-1}(\Flag_{(1,2)}\C^{N+1},I_{2c-2}),
\]
as well as an injection
\[
H^{2c-1}(\Flag_{(1,2)}\C^{N+1},I_{2c-2})
\hookrightarrow
H^{2c}(\Flag_{(1,2)}\C^{N+1},\L_{m-c,n+c-2\abs{\mathbi{d}}}).
\]
Since \(2c<N\), one has the vanishing
\[
H^{2c}(\Flag_{(1,2)}\C^{N+1}, \L_{m-c,n+c-2\abs{\mathbi{d}}})
=
(0),
\]
and with \eqref{eq: proof inter}, this indeed shows that \(S^{m}\Omega_{X}(m+n)\) has no global sections.
\end{proof}

Let us now show that Theorem \ref{thm: application 1.1 vanish} is optimal in the extremal case, namely if we pick \(n=d-2\), and \(m=d-1\) (for \(d \geq 3\))\footnote{Observe that the hypothesis on the codimension has been somewhat relaxed in the statement, as the upper bound has become \(\frac{N+1}{2}\) instead of \(\frac{N}{2}\).}
\begin{theorem}
\label{thm: application 1.1 non-vanish 1}
Let \(n \in \N\), \(\mathbi{d}=(d_{1}, \dotsc, d_{c})\), \(d \bydef \min\Set{d_{i} \ | \ 1 \leq i \leq c}\), and suppose furthermore that \(d \geq 2\). 
A smooth complete intersection \(X \subset \P^{N}\) of codimension \(c < \frac{N+1}{2}\) and multi-degree \(\mathbi{d}\) satisfies the following non-vanishing statements:
\begin{itemize}
\item{} for \(d \geq 3\), \(H^{0}(X,S^{d-1}\Omega_{X}(2d-3)) \neq (0)\);
\item{} for \(d=2\), \(H^{0}(X,S^{2}\Omega_{X}(2)) \neq (0)\).
\end{itemize}
\end{theorem}
\begin{proof}
Suppose first that \(d \geq 3\), and fix \(n=d-2\), as well as  \(m=d-1>n\).
One deduces from the reasoning carried over during the proof of Theorem \ref{thm: application 1.1 vanish} the following two exact sequences:
\begin{equation}
\label{eq: exact seq 1}
\resizebox{\displaywidth}{!}{
\xymatrix{
0
\ar[r]
&
H^{0}(X, S^{m}\Omega_{X}(m+n))
\ar[r]
&
H^{1}(\Flag_{(1,2)}\C^{N+1}, I_{0})
\ar[r]
&
H^{1}(\Flag_{(1,2)}\C^{N+1}, \L_{m,n})
}
}
\end{equation}
and
\begin{equation}
\label{eq: exact seq 2}
\resizebox{\displaywidth}{!}{
\xymatrix{
H^{1}(\Flag_{(1,2)}\C^{N+1}, I_{1})
\ar[r]
&
H^{1}(\Flag_{(1,2)}\C^{N+1}, \big(\bigwedge^{2c-1}\E^{\mathbi{d}}\big)_{m-c,n+c-2\abs{\mathbi{d}}})
\ar[r]
&
H^{1}(\Flag_{(1,2)}\C^{N+1}, I_{0}).
}
}
\end{equation}
Observe that the first cohomology group of \(I_{1}\) vanishes if and only if
\[
H^{1}(\Flag_{(1,2)}\C^{N+1},\big(\bigwedge^{2c-2}\E^{\mathbi{d}}\big)_{m-c,n+c-2\abs{\mathbi{d}}})
=
(0).
\]
(This follows from Lemma \ref{lemma: support}, the successive short exact sequences of \eqref{eq: cut exact} and the hypothesis on the codimension). Note that the locally free sheaf \(\big(\bigwedge^{2c-2}\E^{\mathbi{d}}\big)_{m-c,n+c-2\abs{\mathbi{d}}}\) writes as a direct sum of terms of the following form (where \(i \neq j\)):
\begin{equation}
\label{eq: identity 2}
\big((\E_{0,d_{i}-1}\otimes \E_{0,d_{j}-1}) \otimes (\bigotimes_{k \neq i,j} \L_{1, 2d_{k}-1})\big) \otimes \L_{m-c,n+c-2\abs{\mathbi{d}}}
\simeq
(\E^{\otimes 2})_{m-2,n-(d_{i}+d_{j})}.
\end{equation}
Therefore, by Lemma \ref{lemma: support}, the first cohomology group of  \(\big(\bigwedge^{2c-2}\E^{\mathbi{d}}\big)_{m-c,n+c-2\abs{\mathbi{d}}}\) vanishes as soon as
\[
n<2d-3,
\]
which holds since \(n=d-2\) and \(d \geq 2\). Accordingly, the exact sequence \eqref{eq: exact seq 2} becomes:
\begin{equation}
\label{eq: exact seq 3}
\xymatrix{
0
\ar[r]
&
H^{1}(\Flag_{(1,2)}\C^{N+1}, \big(\bigwedge^{2c-1}\E^{\mathbi{d}}\big)_{m-c,n+c-2\abs{\mathbi{d}}})
\ar[r]
&
H^{1}(\Flag_{(1,2)}\C^{N+1}, I_{0})
}.
\end{equation}

Putting together \eqref{eq: exact seq 1} and \eqref{eq: exact seq 3}, one therefore deduces that the space of global section \(H^{0}(X, S^{m}\Omega_{X}(m+n))\) is non-zero as soon as the kernel 
\[
\Ker \Big( 
\bigoplus_{i=1}^{c}
H^{1}(\Flag_{(1,2)}\C^{N+1},\E_{m-1,n-d_{i}})
\overset{\psi}{\longrightarrow}
H^{1}(\Flag_{(1,2)}\C^{N+1},\L_{m,n})
\Big)
\]
is non-trivial (here, one has used the isomorphism \eqref{eq: identity 1}). Note that in the direct sum
\[
\bigoplus_{i=1}^{c}
H^{1}(\Flag_{(1,2)}\C^{N+1},\E_{m-1,n-d_{i}}),
\]
the only non-zero spaces are the one corresponding to the indexes \(i\) such that \(d_{i}=d\). Without loss of generality, suppose that \(d_{1}=\dotsb=d_{k}=d\), and that \(d_{i} > 2 \) for \(i>k\). 

Recall that \(\E_{m-1,n-d}=\E_{m-1,-2}\) fits into the short exact sequence
\[
\xymatrix{
0
\ar[r]
&
\L_{m,-2}
\ar[r]
&
\E_{m-1,-2}
\ar[r]
&
\L_{m-1,-1}
\ar[r]
&
0.
}
\]
Therefore, using Theorem \ref{thm: coho sym proj}, one obtains that 
\begin{equation}
\label{eq: iso nec}
H^{1}(\Flag_{(1,2)}\C^{N+1},\E_{m-1,-2})
\simeq 
H^{1}(\Flag_{(1,2)}\C^{N+1},\L_{m-1,-1}).
\end{equation}
Using the above isomorphism \eqref{eq: iso nec}, one sees that the map \(\psi\) becomes
\[
\tilde{\psi}\colon
\bigoplus_{i=1}^{k}
H^{1}(\Flag_{(1,2)}\C^{N+1},\L_{m-1,-1})
\longrightarrow 
H^{1}(\Flag_{(1,2)}\C^{N+1},\L_{m,d-2}).
\]
Now, recall that the explicit description of \(H^{1}(\Flag_{(1,2)}\C^{N+1},\L_{m-1,-1})\) is given by
\begin{eqnarray*}
H^{1}(\Flag_{(1,2)}\C^{N+1},\L_{m-1,-1})
&
=
&
\Coker \big( \C[Y,X]_{m-1,-1} \overset{\delta}{\longrightarrow} \C[Y,X]_{m-2,0}\big)
\\
&
=
&
\C[Y]_{m-2}.
\end{eqnarray*}
Similarly, one has
\[
H^{1}(\Flag_{(1,2)}\C^{N+1},\L_{m,d-2})
\simeq
\Coker \big( \C[Y,X]_{m,d-2} \overset{\delta}{\longrightarrow} \C[Y,X]_{m-1,d-1}\big).
\]

Coming back to the explicit description of the resolution \eqref{eq: resolution}, and following all the isomorphisms described above, one sees that the map the map \(\tilde{\psi}\) is given explicitly as follows:
\[
\tilde{\psi} \colon\left(
\begin{array}{ccc}
 \C[Y]_{m-2}^{\oplus k} 
 & 
 \longrightarrow 
 & 
 \Coker \Big( \C[Y,X]_{m,d-2} \overset{\delta}{\longrightarrow} \C[Y,X]_{m-1,d-1}\Big)
  \\
 (A_{1}, \dotsc, A_{k})
  & 
  \longmapsto  
  &  
  (\diff P_{1})_{X}(Y)A_{1} + \dotsb + (\diff P_{k})_{X}(Y)A_{k}
\end{array}
\right).
\]
One has to show that kernel of \(\tilde{\psi}\) is not trivial, and to this end, it is enough to find a non-zero element \(A \in \C[Y]_{m-2}=\C[Y]_{d-3}\) such that
\[
A \times (\diff P_{1})_{X}(Y)
\in 
\Image(\delta).
\]
As a matter of fact, one shows that for \textsl{any} \(A \in \C[Y]_{d-3}\) one has 
\[
A \times (\diff P_{1})_{X}(Y)
\in 
\Image(\delta).
\]
Consider the following polynomial:
\begin{eqnarray*}
B 
\bydef
\sum\limits_{i=1}^{d-2} (-1)^{i}\delta^{d-2-i}(A)\delta^{i}(P_{1}(Y)).
\end{eqnarray*}
Using the Leibnitz rule, compute that
\begin{eqnarray*}
\delta(B)
&
=
&
-\Big(\delta^{d-3}(A) \delta^{2}(P_{1}(Y))\Big) + \Big(\delta^{d-3}(A)\delta^{2}(P_{1}(Y))+\delta^{d-4}(A) \delta^{3}(P_{1}(Y))\Big))
\\
&
&
+ \dotsb + (-1)^{d-2}\Big(\delta(A)\delta^{d-2}(A)+A\delta^{d-1}(P_{1}(Y))\Big)
\\
&
=
&
(-1)^{d-2}A\times \delta^{d-1}(P_{1}(Y)).
\end{eqnarray*}
In order to  conclude, one shows by induction on \(d \geq 1\) the following identity for any \(R \in \C[X]_{d}\):
\[
\delta^{d-1}(R(Y))
=
(d-1)! \times (\diff R)_{X}(Y).
\]
For \(d=1\), the result is straightforward. Let therefore \(d \in \N_{\geq 2}\). By linearity, it suffices to prove the identity for any \(R \in \C[X]_{d}\) polynomial of the form 
\[
R
=
X_{i}R_{1},
\]
where \(0 \leq i \leq N\). By the Leibnitz rule, one has the equality:
\[
\delta^{d-1}(R(Y))
=
\binom{d-1}{1}X_{i} \delta^{d-2}(R_{1}(Y)) + \binom{d-1}{0}Y_{i} \delta\big(\delta^{d-2}(R_{1}(Y))\big).
\]
Then, by induction hypothesis, compute that:
\begin{eqnarray*}
\delta^{d-1}(R(Y))
&
=
&
(d-1)!\times X_{i}(\diff R_{1})_{X}(Y) + (d-2)!\times Y_{i} \delta\big((\diff R_{1})_{X}(Y)\big)
\\
&
=
&
(d-1)! \big(X_{i} (\diff R_{1})_{X}(Y) + Y_{i} R_{1}(X)\big)
\\
&
=
&
(d-1)! (\diff R)_{X}(Y),
\end{eqnarray*}
hence the identity. This proves the theorem for \(d \geq 3\). 

As for the case where \(d=2\), fix \(m=d=2\) (instead of \(m=d-1\)), and keep \(n=d-2=0\). The same line of reasoning works, and one has to show that kernel of the map (keeping the same notations as above):
\[
\tilde{\psi} \colon\left(
\begin{array}{ccc}
 \C[Y]_{0}^{\oplus k} 
 & 
 \longrightarrow 
 & 
 \Coker \Big( \C[Y,X]_{2,0} \overset{\delta}{\longrightarrow} \C[Y,X]_{1,1}\Big)
  \\
 (A_{1}, \dotsc, A_{k})
  & 
  \longmapsto  
  &  
  \Big[(\diff P_{1})_{X}(Y)A_{1} + \dotsb + (\diff P_{k})_{X}(Y)A_{k}\Big]
\end{array}
\right)
\]
is not trivial. This fact follows immediately from the identity:
\[
\delta(P_{i}(Y))
=
(\diff P_{i})_{X}(Y).
\]
This concludes the proof of the theorem.
\end{proof}

Using the previous vanishing and non-vanishing statements, we can now prove the following theorem, which can be seen as a generalization of results in \cite{Bog}[Theorems B and D]:

\begin{theorem}
\label{cor: generalization}
Let \(d \in \N_{\geq 2}\) be a natural number, and let \(X \subset \P^{N}\) be a non-degenerate\footnote{Recall that it means that the complete intersection is not included in an hyperplane.} smooth complete intersection of codimension \(c < \frac{N}{2}\). The complete intersection \(X\) satisfies
\[
\bigoplus_{m \geq \max(d-2,1)} H^{0}\big(X, S^{m}\Omega_{X}(m+\max(d-3,0))\big)
=
(0)
\]
if and only if \(X\) is not included in an hypersurface of degree \( 2 \leq i \leq d\).
\end{theorem}
\begin{proof}
One direction of the equivalence is clear by the vanishing result of Theorem \ref{thm: application 1.1 vanish}.
For the other direction, suppose that \(X\) is included in an hypersurface of minimal possible degree \(2\leq i \leq d\).
If \(i=2\), it follows from the non-vanishing result of Theorem \ref{thm: application 1.1 vanish} that 
\[
H^{0}(X, S^{2}\Omega_{X}(2)) \neq (0).
\]
This proves the wanted result in the case \(i=2\). Suppose now that \(3 \leq i \leq d\). This time, it follows from the non-vanishing result of Theorem \ref{thm: application 1.1 vanish} that 
\[
H^{0}(X, S^{i-1}\Omega_{X}(2i-3)) \neq (0).
\]
Since \(\Omega_{X}(2)\) is globally generated, this implies that
\[
H^{0}\big(X, S^{d-2}\Omega_{X}(\underbrace{2i-3+2(d-i-1)}_{d-2+(d-3)})\big) \neq (0).
\]
This finishes the proof.
\end{proof}
In \cite{Bog}[Theorems B and D], the authors proved that a smooth subvariety \(X \subset \P^{N}\) (not necessarily a complete intersection) of codimension \(c \in \Set{1,2}\), with \(c < \frac{N}{2}\), satisfies the following:
\[
\bigoplus_{m \geq 1} H^{0}(X, S^{m}\Omega_{X}(m))=(0)
\Longleftrightarrow
\ \text{\(X\) is not included in a quadric}.
\]
In view of Hartshorne's conjecture (see \cite{Laz1}[Conjecture 3.2.8]), such a smooth subvariety of codimension \(2\) should be a complete intersection as soon as \(N \geq 7\). 

As an evidence towards Hartshorne's conjecture, a natural problem would be to generalize Theorem \ref{cor: generalization} to the case of an arbitrary smooth subvariety \(X \subset \P^{N}\) of codimension \(c \leq \frac{N}{3}\), the case \(c=2, d=2\) being 
 settled by \cite{Bog}[Theorem D].

\subsubsection{Vanishing and non-vanishing theorems for \(H^{1}\), in codimension \(c < \frac{N}{2}-1\).}
In this subsection, let us first give an alternative proof of the known vanishing theorem of Bruckmann-Rackvitz \cite{BR}[Theorem 4 (ii)], in the case of symmetric powers:
\begin{theorem}
\label{thm: application 1.2 vanish}
Let \(X \subset \P^{N}\) be a smooth complete intersection of codimension \(c < \frac{N}{2}-1\), and multi-degree \(\mathbi{d}=(d_{1}, \dotsc, d_{c})\).  Denote \(d \bydef \min\Set{d_{i} \ | \ 1 \leq i \leq c}\).

For any \(m \in \N\) and any \(n \in \Z\) satisfying the inequality \(n<-1\), one has the following vanishing result:
 \[
 H^{1}(X, S^{m}\Omega_{X}(m+n))
 =
 (0).
 \]
\end{theorem}
\begin{proof}
Considering the first short exact sequence of \eqref{eq: cut exact}, and the vanishing \(H^{1}(\Flag_{(1,2)}\C^{N+1}, \L_{m,n})=(0)\) (which holds under the assumption \(n < -1\), see Theorem \ref{thm: coho sym proj}), one sees that it is enough to show that
\[
H^{2}(\Flag_{(1,2)}\C^{N+1}, I_{0})
=
(0).
\]
This follows from Lemma \ref{lemma: support}, the successive short exact sequences of \eqref{eq: cut exact}, and the fact that \(c < \frac{N}{2}-1\) (in a classical manner, that was already encountered in Theorem \ref{thm: application 1.1 vanish}).
\end{proof}

Let us now show that this vanishing theorem is optimal:

\begin{theorem}
\label{thm: application 1.2 non vanish}
Let \(X \subset \P^{N}\) be a smooth complete intersection of codimension \(c < \frac{N}{2}-1\), and multi-degree \(\mathbi{d}=(d_{1}, \dotsc, d_{c})\).  Denote \(d \bydef \min\Set{d_{i} \ | \ 1 \leq i \leq c}\).

For any \(m \in \N\), and any \(n \in \Z\) satisfying the inequality \(-1<n<\min(d-1,m-2)\), one has the following non-vanishing result:
\[
H^{1}(X, S^{m}\Omega_{X}(m+n))
 \neq
 (0).
\]
\end{theorem}
\begin{proof}
Consider the exact sequence coming from the first short exact sequence of \eqref{eq: cut exact}:
\[
\xymatrix{
H^{1}(\Flag_{(1,2)}\C^{N+1}, I_{0})
\ar[r]
&
H^{1}(\Flag_{(1,2)}\C^{N+1}, \L_{m,n})
\ar[r]
&
H^{1}(X, S^{m}\Omega_{X}(m+n)).
}
\]
Observe that the hypothesis on \(n\) implies the non-vanishing \(h^{1}(\L_{m,n}) \neq 0\), by Theorem \ref{thm: coho sym proj}.
Therefore, in order to prove the non-vanishing of \(H^{1}(X, S^{m}\Omega_{X}(m+n))\), it is enough to show that
\[
h^{1}(I_{0})=0.
\]
Using the second short exact sequence of \eqref{eq: cut exact}, one sees that it is actually enough to show that
\[
h^{1}
\Big(
\big(\bigwedge^{2c-1}\E^{\mathbi{d}}\big)_{m-c,n+c-2\abs{\mathbi{d}}}
\Big)
=
0.
\]
Using the identity \eqref{eq: identity 1} and Lemma \ref{lemma: support}, one deduces the sought result.
\end{proof}

\subsubsection{Vanishing theorems for \(H^{i}\), in codimension \(c<\frac{N}{2}-i\).}
Using the same reasoning as in the proof of Theorem \ref{thm: application 1.1 vanish} or \ref{thm: application 1.2 vanish}, we can again give an alternative proof of the known vanishing theorem of Bruckmann-Rackvitz \cite{BR}[Theorem 4 (i)], in the case of symmetric powers:
\begin{theorem}
\label{thm: application 1.3 vanish}
Let \(i \in \N\) be an integer. Let \(X \subset \P^{N}\) be a smooth complete intersection of codimension \(c < \frac{N}{2}-i\), and multi-degree \(\mathbi{d}=(d_{1}, \dotsc, d_{c})\).  
For any \(m \in \N\) and any \(n \in \Z\), one has the following vanishing result:
 \[
 H^{i}(X, S^{m}\Omega_{X}(m+n))
 =
 (0).
 \]
\end{theorem}
We leave the (easy) details to the reader.

\subsection{Second application: the algebra \(\bigoplus_{m \in \N} H^{0}(X, S^{m}\Omega_{X}(m))\) for smooth complete intersections of codimension \(c < \frac{N}{2}\).}
\label{sect: algebra}
In this section, we fix a smooth complete intersection of codimension \(c\)
 \[
 X
 =
 \Set{q_{1}=0} \cap \dotsb \Set{q_{k}=0} \cap \Set{P_{k+1}=0} \cap \dotsb \cap \Set{P_{c}=0}
  \subset \P^{N},
\]
where \(q_{1}, \dotsc, q_{c}\) are quadratic polynomials, and where \(\deg(P_{i}) > 2\) for \(i>k\). Denote \(\mathbi{d}=(\underbrace{2,\dotsc, 2}_{\times k}, d_{k+1}, \dotsc, d_{c})\) the multi-degree of \(X\).

Following the proof of the non-vanishing Theorem \ref{thm: application 1.1 non-vanish 1}, we can prove the following result:
\begin{theorem}
\label{thm: sym alg}
Suppose that \(X\) is of codimension \(c < \frac{N}{2}\). Then there is a natural graded isomorphism of \(\C\)-algebras
\[
\bigoplus_{m \in \N} H^{0}(X, S^{m}\Omega_{X}(m))
\simeq
\C[q_{1}, \dotsc, q_{k}] 
\subset \C[X].
\]
\end{theorem}
\begin{proof}
Follow the same reasoning as in the proof of \ref{thm: application 1.1 non-vanish 1}, but set instead \(m \in \N_{\geq 1}\) arbitrary, and fix \(n=0\) (so that \(m>n\), which was required during the reasoning). For the purposes of the proof of Theorem \ref{thm: application 1.1 non-vanish 1}, it was enough to know that there is an injection
\begin{equation}
\label{eq: isom}
\xymatrix{
0
\ar[r]
&
H^{1}(\Flag_{(1,2)}\C^{N+1}, \big(\bigwedge^{2c-1}\E^{\mathbi{d}}\big)_{m-c,c-2\abs{\mathbi{d}}})
\ar[r]
&
H^{1}(\Flag_{(1,2)}\C^{N+1}, I_{0}).
}
\end{equation}
(For which the hypothesis \(c < \frac{N+1}{2}\) was enough).
Here, one needs the above map to be an isomorphism.This follows from Lemma \ref{lemma: support}, the successive short exact sequences of \eqref{eq: cut exact}, and the hypothesis on the codimension \(c < \frac{N}{2}\).

That being said, the reasoning of the proof of Theorem \ref{thm: application 1.1 non-vanish 1} shows that for any \(m \in \N_{\geq 1}\), the space of global sections \(H^{0}(X, S^{m}\Omega_{X}(m))\) is isomorphic to the kernel of the map
\[
\phi_{m}\colon\left(
\begin{array}{ccc}
 \C[Y]_{m-2}^{\oplus k} 
 & 
 \longrightarrow 
 & 
 \Coker \Big( \C[Y,X]_{m,0} \overset{\delta}{\longrightarrow} \C[Y,X]_{m-1,1}\Big)
  \\
 (A_{1}, \dotsc, A_{k})
  & 
  \longmapsto  
  &  
  \Big[(\diff q_{1})_{X}(Y)A_{1} + \dotsb + (\diff q_{k})_{X}(Y)A_{k}\Big]
\end{array}
\right).
\]

Start by roughly analyzing the constraints on the elements in the kernel. Let \((A_{1}, \dotsc, A_{k})\) be in the kernel of \(\phi_{m}\). By definition, there exists a polynomial \(B \in \C[Y]_{m}\) such that
\begin{equation}
\label{eq: ident kernel}
\delta(B)=\delta(q_{1}(Y))A_{1} + \dotsb + \delta(q_{k}(Y))A_{k}.
\end{equation}
Consider the operator \(\delta^{rev}\bydef \sum\limits_{i=0}^{N} Y_{i} \frac{\partial}{\partial X_{i}}\), and compute that
\[
\delta^{rev} \circ \delta
=
\sum\limits_{i=0}^{N} Y_{i} \frac{\partial}{\partial Y_{i}}
+
\sum\limits_{0 \leq i,j \leq N} Y_{i}X_{j} \frac{\partial^{2}}{\partial X_{j} \partial Y_{j}}.
\]
Applying \(\delta^{rev}\) to both sides of the equality \eqref{eq: ident kernel} (and using the Leibnitz rule), one finds the equality:
\[
mB
=
2\sum\limits_{i=1}^{k} q_{i}A_{i}.
\]
(A word of caution: from now on, when writing the symbol \(q_{i}\), one means the quadratic polynomial \(q_{i}\) \textsl{in the variables \(Y\)}).
Reinjecting the above equality into the equality \eqref{eq: ident kernel}, one finds that the following equalities are satisfied for any \(0 \leq j \leq N\):
\[
\frac{\partial}{\partial Y_{j}}
\big(\sum\limits_{i=1}^{k} q_{i}A_{i}\big)
=
\frac{m}{2}
\sum\limits_{i=1}^{k} \frac{\partial q_{i}}{\partial Y_{j}}A_{i}.
\]
Now, observe that this set of equalities is equivalent to the following equality
\begin{equation}
\label{eq: coord free}
\diff \big(\sum\limits_{i=1}^{k} q_{i} A_{i}\big)
=
\frac{m}{2} \sum\limits_{i=1}^{k} A_{i}\diff q_{i} ,
\end{equation}
where \(\diff\) is the exterior differential. Reciprocally, one immediately checks that any \(k\)-uple of polynomials \((A_{1}, \dotsc, A_{k}) \in \C[Y]^{\oplus k}\) satisfying the equality \eqref{eq: coord free} defines an element in the kernel of \(\phi_{m}\).

Now that one has given a characterization of the kernel of \(\phi_{m}\) in complex analytic terms, one can use tools from complex analysis to study it. Since \(X\) is a complete intersection, the Jacobian matrix formed with the partial derivatives of \(q_{1}, \dotsc, q_{k}\) is of maximal rank on a non-empty open subset \(U\) of \(\Set{q_{1}=0}\cap \dotsc \cap \Set{q_{k}=0}\). Fix \(u \in U\), and consider accordingly system of holomorphic coordinates around \(u\) of the following form:
\[
Z\bydef (z_{1}\bydef q_{1}, \dotsc, z_{k}\bydef q_{k}, z_{k+1}, \dotsc, z_{N+1}).
\]
One is interested in the convergent power series \((A_{1}, \dotsc, A_{k})\) in the variables \(Z\) satisfying the equality:
\begin{equation}
\label{eq: holo}
\diff \big(\sum\limits_{i=1}^{k} z_{i} A_{i}\big)
=
\frac{m}{2} \sum\limits_{i=1}^{k} A_{i}\diff z_{i}.
\end{equation}
Note first that the equality \eqref{eq: holo} implies that the differential form \(\sum\limits_{i=1}^{k} A_{i}\diff z_{i}\) is closed. Hence, by the Poincaré lemma, there exists a convergent power serie \(A\), depending only the variables \((z_{1}, \dotsc, z_{k})\) and which can be supposed to vanish at \(u\), such that for any \(1 \leq i \leq k\), the following holds
\[
\frac{\partial A}{\partial z_{i}}
=
A_{i}.
\]
Decompose the power serie \(A\) into its homogeneous parts
\[
A=A^{1} + A^{2} \dotsb,
\]
where \(A^{i}\) is a homogeneous polynomial of degree \(i\) in the variables \((z_{1}, \dotsc, z_{k})\). The equality \eqref{eq: holo} can then be rewritten as follows:
\[
\diff
\big(
A^{1} + 2A^{2} + \dotsb
\big)
=
\frac{m}{2}\diff
\big(
A^{1} + A^{2} + \dotsb
\big).
\]
This implies that \(A\) must be zero if \(m\) is odd, and that \(A\) must be a homogeneous polynomial of degree \(\frac{m}{2}\) is \(m\) is even.

In conclusion, the reasoning in the previous paragraphs shows the following:
\begin{itemize}
\item{} if \(m\) is odd, then the kernel of \(\phi_{m}\) is trivial;
\item{} if \(m\) is even, then the kernel of \(\phi_{m}\) is equal to
\[
\Big\{
\big(\frac{\partial P}{\partial q_{1}}, \dotsc, \frac{\partial P}{\partial q_{k}}\big)
\ | \
P(q_{1}, \dotsc, q_{k}) \ \text{homogeneous of degree \(\frac{m}{2}\) (in the variables \((q_{1}, \dotsc, q_{k})\))}
\Big\}.
\]
\end{itemize}
In the second case, the elements in the kernel are thus in natural bijection with \(\C[q_{1}, \dotsc, q_{k}]_{\frac{m}{2}}\). All of this allows to define a natural bijective map
\[
\bigoplus_{m \in \N} H^{0}(X, S^{m}\Omega_{X}(m))
\longrightarrow
\C[q_{1}, \dotsc, q_{k}],
\]
and it is a simple verification to show that it respects the structure of algebra. This finishes the proof of the theorem.
\end{proof}
\begin{remark}
\label{rem: sym alg intro}
Let us note that there are still interesting features that can be said in the case where the codimension is equal to \(c=\frac{N}{2}\). In order to apply the previous proof, one has to show that the map \eqref{eq: isom} is an isomorphism for any \(m \in \N\). By the successive short exact sequences \eqref{eq: cut exact}, one has the isomorphisms:
\[
H^{2}(\Flag_{(1,2)}\C^{N+1}, I_{1}) 
\simeq 
\dotsb 
\simeq 
H^{2c-1}(\Flag_{(1,2)}\C^{N+1}, I_{2c-2}).
\]
Furthermore, with the last short exact sequence of \eqref{eq: cut exact}, one has that the last space \(H^{2c-1}(\Flag_{(1,2)}\C^{N+1}, I_{2c-2})\) is equal to:
\[
\Ker \Big(
H^{N}(\Flag_{(1,2)}\C^{N+1}, \L_{m-c,c-2\abs{\mathbi{d}}})
\overset{\psi}{\longrightarrow}
\bigoplus_{i=1}^{c} H^{N}(\Flag_{(1,2)}\C^{N+1}, \E_{m-c, c-2\abs{\mathbi{d}}+d_{i}-1})
\Big).
\]
Using the ideas and reformulation techniques described throughout Section \ref{sect: coho sym power ci}, the map \(\psi\) can be described as follows:
\[
\xymatrix{
\bigoplus_{i=1}^{c} \frac{\C[Y,X]_{m-c+1, 2\abs{\mathbi{d}}+1-d_{i}-c-(N+1)}}{(q^{2})}
\ar[rrrr]^-{\alpha^{*}(P_{1})(\cdot) + \dotsb + \alpha^{*}(P_{c})(\cdot)}
&&&&
\frac{\C[Y,X]_{m-c, 2\abs{\mathbi{d}}-c-(N+1)}}{(q)}
},
\]
where we recall that \(\alpha^{*}(P_{j})(\cdot)\bydef \sum\limits_{i=0}^{N} \frac{\partial P_{j}}{\partial X_{i}}\frac{\partial}{\partial Y_{i}}(\cdot)\). Proving that \eqref{eq: isom} is an isomorphism is then equivalent to proving that the above map is surjective.

For example, in the case where \(X=\Set{q_{1}=0}\cap\Set{q_{2}=0} \subset \P^{4}\) is a smooth complete intersection of two quadrics in the \(4\)-dimensional projective space, the surjectivity is easily checked. Indeed, for any \(m \in \N\), the map
\[
\alpha^{*}(q_{1})\colon
\C[Y,X]_{m,0} 
\longrightarrow
\C[Y,X]_{m-1,1}
\]
is already surjective: it can be checked directly, or one can notice that the cokernel of this map is actually isomorphic to \(H^{1}(\Flag_{(1,2)}\C^{N+1}, \L_{m,0})\), which is zero. This allows to recover the result that the tangent bundle \(TX\) is not big, which was recently proved in \cite{Mall} and independently in \cite{Hor}. Indeed, one has the isomorphism \(TX \simeq \Omega_{X}(1)\), and the description of
\[
\bigoplus_{m \in \N} H^{0}(X, S^{m}\Omega_{X}(m)) \simeq \C[q_{1}, q_{2}]
\]
readily implies that \(TX\) is not big (using the description of bigness via the asymptotic growth of the dimension of the space of global section of larger and larger symmetric powers of the tangent bundle).
\end{remark}

\subsection{Third application: intermediate ampleness of cotangent bundles of (general) hypersurfaces.}
\label{sect: intermediate ampleness}
Recall that, on a projective variety \((X, \O_{X}(1))\) polarized by a very ample line bundle, a vector bundle \(\E\) is said (uniformly) \(q\)-ample if there exists a constant \(\lambda>0\) such that for any \(m, n \in \N_{\geq 1}\) satifsying the inequality \(m \geq \lambda n\), and any \(i > q \), the following holds:
\[
H^{i}(X, S^{m}\E(-n))
=
(0).
\]
For \(q=0\), one recovers the usual  ampleness for vector bundles, and note that these notions form a string of implication, namely 
\begin{center}
\(0\)-ampleness \(\Rightarrow\) \(1\)-ampleness \(\Rightarrow \dotsb \Rightarrow\) \(\dim(X)\)-ampleness.
\end{center}
Also, note that any vector bundle on \(X\) is automatically \(\dim(X)\)-ample (since any coherent sheaf on a projective variety of dimension \(n\) has its cohomology supported in degree less or equal than \(n\) by Grothendieck's vanishing theorem (see \cite{Harts}[III. Theorem 2.7])). We refer to \cite{Tot} or the survey \cite{Greb} for more details on partial ampleness.

If we consider the projective variety \((\P^{N}, \O_{\P^{N}}(1))\), one easily sees by Theorem \ref{thm: coho sym proj} that the cotangent bundle \(\Omega_{\P^{N}}\) has the least possible positivity property, namely it is merely \(N\)-ample. Recall Xie-Brotbek--Darondeau theorem (see \cite{Xie} and \cite{BD}), which asserts that a sufficiently general complete intersection \(X\), of codimension \(c \geq \frac{N}{2}\) and of multi-degree large enough, has ample cotangent bundle.
 In view of this result, it is natural to address the following question. Consider a smooth projective variety \(X\), and suppose that \(\Omega_{X}\) is \(q\)-ample. Does it hold that, if one cuts \(X\) by a sufficiently ample and sufficiently general (smooth) hypersurface \(H\), the cotangent bundle \(\Omega_{H}\) is \((q-2)\)-ample?\footnote{By convention,  one sets the notion of \(q\)-ampleness to be equivalent to the notion of \(0\)-ampleness for any \(q \leq 0\).}

This question has a positive answer in the case of hypersurfaces in projective spaces, and it turns out that, in this case:
\begin{itemize}
\item{} it is enough to suppose \(\deg(H) \geq 4\);
\item{} the generality assumption on the hypersurface can be dropped.
\end{itemize}
One then may wonder if the generality assumption could be dropped in the above question: we do not know, but it would definitely be striking if we could!

The result in the case of hypersurfaces is quite recent, as it was proved in \cite{Hor}. In \textsl{loc.cit}, the authors rather look at the space of twisted global symmetric vector fields on an hypersurface \(X \subset \P^{N}\), and prove a vanishing result implying the \((N-2)\)-ampleness of the cotangent bundle \(\Omega_{X}\) (via Serre duality). Their proof relies on the isomorphism
\[
TX \simeq \bigwedge^{N-2}\Omega_{X} \otimes K_{X}^{\vee}
\]
and the vanishing theorem of Bruckmann-Rackvitz \cite{BR}[Theorem 4 (iii)] for the composition of Schur functors \(S^{m}  \circ \bigwedge^{N-2}\) (which can be shown, via some Plethysm, to be the Schur functor associated to the partition \((\underbrace{m, \dotsc,m}_{\times (N-2)})\), so that the vanishing theorem applies, see \cite{Hor}). 

In this section, we use the description of cohomology given in Section \ref{sect: coho hyp 1} to prove the following vanishing theorem, slightly weaker than the one in \cite{Hor}:
\begin{theorem}
\label{thm: vanishing hyp}
Let \(H\) be a smooth hypersurface of degree \(d\), whose defining equation \(P\) is of the following form
\[
P \bydef X_{N}^{d}-F(X_{0}, \dotsc, X_{N-1}),
\]
where \(F\) is a smooth polynomial, homogeneous of degree \(d\). 
Let \(m \in \N\), \(n \in \Z\), and suppose that \(d=\deg(H) \geq 3\).
As soon as \(m(d-3) > n\), one has the vanishing:
\[
H^{N-1}(H, S^{m}\Omega_{H}(d-n-(N+1)))
=
(0).
\]
In particular, as soon as \(d \geq 4\), this shows that \(\Omega_{H}\) is uniformly \(q\)-ample (with uniform bound \(\lambda=\frac{1}{d-3}\) if \(\deg(H) \geq N+1\)), and hence so is a general hypersurface of degree \(d \geq 4\).
\end{theorem}
\begin{remark}
The bound in the statement is the best possible. Indeed, consider the simplest case of a curve \(H \subset \P^{2}\). By adjunction, one has the isomorphism
\[
T{H}^{\otimes m}\otimes \O_{H}(n)
\simeq
\O_{H}(n-m(d-3)),
\]
and the space of global sections of such a line bundle is zero if and only if \(m(d-3) > n\).
\end{remark}
The strategy to prove the result is simply the following: by Theorem \ref{thm: coho hyp 1}, the vanishing result is  equivalent to the injectivity of the map
 \[
 \alpha^{*}(P)_{m,m+n}\colon
 \frac{S_{m,m+n}}{(P,q)}
 \longrightarrow
 \frac{S_{m-1, m+n+d-1}}{(P,q)}.
 \]

Denote by \(I \bydef (P,q)\) the ideal spanned by \(P\) and \(q\). For later use, record the following elementary lemma:
\begin{lemma}
\label{lemma: primeness}
The ideal \(I\) is prime.
\end{lemma}
\begin{proof}
The Jacobian criterion allows to show that the variety defined by the vanishing of \(P\) and \(q\) is smooth. This immediately implies the primeness of \(I\).
\end{proof}
Let us now proceed to the proof of Theorem \ref{thm: vanishing hyp}:
\begin{proof}[Proof of Theorem \ref{thm: vanishing hyp}]
Let  \(A \in S_{m,m+n}\) be a polynomial such that \(\alpha^{*}(P)(A) \equiv 0 \mod I\).
One must show that
\[
A \equiv 0 \mod I.
\]
Note first that, up to adding an element of \(I\) to \(A\) (therefore leaving the class \(A \mod I\) unchanged), one can always suppose that
\[
\alpha^{*}(P)(A) 
\equiv 0
\mod (q).
\]
Indeed, by hypothesis, one can write that
\[
\alpha^{*}(P)(A)=PU + qV
\]
for some polynomials \(U, V\). Then, observe that the following equality holds:
\[
\alpha^{*}(P)(A-qU)
=
q\big(V+ \alpha^{*}(P)(U)\big).
\]

Now, deshomogenize the variable \(X\) by dividing by \(X_{N}\), and set \(x_{i}\bydef \frac{X_{i}}{X_{N}}\):
\[
\left\{ 
\begin{array}{ll}
q \rightsquigarrow q^{in}=\sum\limits_{i=0}^{N-1} x_{i}Y_{i} + Y_{N}
\\
P \rightsquigarrow P^{in}(x)=1+F(x)
\\
I \rightsquigarrow I^{in} \bydef (P^{in}, q^{in})
\\
\alpha^{*}(P)
\rightsquigarrow 
\alpha^{*}(P)^{in}
=
\frac{1}{d}
\big(
\sum\limits_{i=0}^{N-1} \frac{\partial F}{\partial x_{i}} \frac{\partial}{\partial Y_{i}}
+
\frac{\partial}{Y_{N}}
\big)
\\
A \rightsquigarrow A^{in}(x,Y)
\end{array}
\right..
\]
Note that one still has the equalities:
\begin{equation}
\label{eq: eq proof 1}
\left\{ 
\begin{array}{ll}
\alpha^{*}(P)^{in}(A^{in}) \equiv 0 \mod q^{in}
\\
\alpha^{*}(P)^{in}(q^{in})=P^{in}
\end{array}
\right..
\end{equation}
As \(q^{in}\) is now unitary in \(Y_{N}\), one can write uniquely (by Euclidean division)
\[
A
=
\sum\limits_{k=0}^{m} a_{k}(x,\overset{\wedge_{N}}{Y})(q^{in})^{k},
\]
where \(a_{k}(x,\overset{\wedge_{N}}{Y})\) does not depend of the variable \(Y_{N}\). Beware that the degree in \(x\) of the \(a_{k}'s\) has increased, and one has the following inequality for \(0\leq k \leq m\):
\[
\deg_{x}(a_{k})
\leq 
n+2m.
\]

The equations \eqref{eq: eq proof 1} imply that
\[
\alpha^{*}(P)^{in}(a_{0}) + P^{in}a_{1} \equiv 0 \mod q^{in}.
\]
Since \(\alpha^{*}(P)^{in}(a_{0}) + P^{in}a_{1}\) is independent of the variable \(Y_{N}\), the above equality implies in turn the equality:
\begin{equation}
\label{eq: eq proof 2}
\alpha^{*}(P)^{in}(a_{0}) + P^{in}a_{1}
=
0.
\end{equation}
Now, the key observation, which uses crucially the form of the equation \(P\),  is that one has the following equality:
\[
\alpha^{*}(P)^{in}(a_{0})
=
\frac{1}{d}\Big(
\sum\limits_{i=0}^{N-1} \frac{\partial F}{\partial x_{i}} \frac{\partial}{\partial Y_{i}}
\Big)(a_{0})
=
\alpha^{*}(F)(a_{0}),
\]
where \(F\) is homogeneous of degree \(d\) in the variables \(x=(x_{0}, \dotsc, x_{N-1})\). Therefore, the equality \eqref{eq: eq proof 2} rewrites:
\begin{equation}
\label{eq: eq proof 3}
\alpha^{*}(F)(a_{0}) 
+
(1+F)a_{1}
=
0.
\end{equation}
One now deduces from \eqref{eq: eq proof 3} the following lemma:
\begin{lemma}
\label{lemma: lemma proof}
The element \(a_{1}\) is in the image of \(\alpha^{*}(F)\).
\end{lemma}
\begin{proof}
For \(a \in S' \bydef \C[x,\overset{\wedge_{N}}{Y}]\), denote by \(a^{min}\) its minimal homogenous component. Since \(\alpha(F)\) is a graded map (i.e. it shifts the grading, since \(F\) is homogeneous), one deduces from \(\eqref{eq: eq proof 3}\) the following equality:
\[
\alpha^{*}(F)(a_{0}^{min})+a_{1}^{min}
=
0.
\]
Therefore, \(a_{1}^{min} \in \Image \alpha^{*}(F)\). One can therefore rewrite \(\eqref{eq: eq proof 3}\) as follows:
\[
\alpha^{*}(F)(\underbrace{a_{0}+(1+F)a_{0}^{min}}_{\bydef \tilde{a_{0}}})
+
(1+F)(\underbrace{a_{1}-a_{1}^{min}}_{\bydef \tilde{a_{1}}})
=
0.
\]
One can now repeat the previous reasoning, so that one eventually shows that every homogeneous component of \(a_{1}\) lies in the image of \(\alpha^{*}(F)\), which shows the result.
\end{proof}
With the previous Lemma \ref{lemma: lemma proof}, one can thus write
\[
a_{1}
=
\alpha^{*}(F)(b),
\]
with \(b \in \C[x,\overset{\wedge_{N}}{Y}]\). Accordingly, the equation \eqref{eq: eq proof 3} rewrites:
\[
\alpha^{*}(F)(a_{0}+P^{in}b)
=
0.
\]
Note that \(\deg_{Y}(a_{0}+P^{in}b_{0})=m\) (it has remained homogeneous of degree \(m\) in the variables \(Y\)), whereas \(\deg_{x}(a_{0}+P^{in}b) \leq n+2m\). By hypothesis, one has that
\[
(d-3)m > n \implies (d-1)m>n+2m.
\]
Therefore, Proposition \ref{prop: bound inj} applies, and the map
\[
\alpha^{*}(F): S'_{m, i} \to S'_{m, i+d-1}
\]
is injective for any \(i \leq n+2m\).
Therefore, one deduces that
\[
a_{0}=-P^{in}b.
\]

Now, rehomogenize everything by multiplying by a suitable power of \(X_{N}\). One deduces that there exists \(M \in \N_{\geq 1}\) such that
\[
X_{N}^{M} A \in I=(P,q).
\]
By Lemma \ref{lemma: primeness}, the ideal \(I\) is prime, so that \(A \in I\): the proof is complete.

\end{proof}

\subsection{Fourth application: effective computations in low-dimensional cases.}
\label{sect: effective}
The goal of this section is to illustrate how the explicit description of cohomology given in Section \ref{sect: coho sym power ci} can be used to recover some known results, and provide new examples of families (of surfaces) that do not satisfy a given property.

 \subsubsection{Surface of bi-tangent lines to quartic surfaces}
 \label{subs: bi-tangent}
Consider \(S \subset \P^{3}\) a quartic surface of Picard rank one. By \cite{Ottem}[Corollary 4.2], it is known that the closure of the effective cone of 
\(
\xymatrix{
\P(TS)
\ar[r]^-{\pi}
&
S
}
\)
is generated by \(H \bydef \pi^{*}\O_{S}(1)\) and \(3\xi+4H\), where \(\xi\bydef \O_{\P(TS)}(1)\). It is then natural to ask whether or not a multiple of \(3\xi+4H\) is effective. The answer is positive, as it is known since the works of \cite{Tiho} and \cite{Welters} that the surface of bi-tangent lines to \(S\) defines an hypersurface in the linear system \(\abs{6\xi+8H}\).

Suppose for a moment that one does not know the existence of such a particular surface (which exists for any smooth quartic in \(\P^{3}\)). By Serre duality, since \(K_{S}=\O_{S}\), one has the isomorphism
\[
H^{2}(S, S^{6}\Omega_{S}(-8))
\simeq
H^{0}(S, S^{6}TS(8)).
\]
On the other hand, since \(S\) is a K3 surface, one also has the isomorphism \(TS \simeq \Omega_{S}\), so that one has
\[
H^{0}(S, S^{6}\Omega_{S}(8))
\simeq
H^{2}(S, S^{6}\Omega_{S}(-8)).
\]

For sake of simplicity, pick \(S\bydef \Set{\underbrace{X_{0}^{4}+\dotsb+X_{3}^{4}}_{\bydef P}=0}\) a Fermat hypersurface. Using the description given in Section \ref{sect: coho hyp 1} (or alternatively in Section \ref{sect: coho hyp 2}), one computes (see Appendix \ref{appendix: sage}) that there is indeed one (and only one) global section of \(S^{6}\Omega_{S}(8)\). 

Explicitly, the rank computation carried over in Appendix \ref{appendix: sage} shows that there exists a (unique modulo \((P,q)\)) polynomial \(W \in S_{6,14}=\C[Y,X]_{6,14}\) which
\begin{itemize}
\item{} is not zero modulo \((P,q)\) (recall that \(q=X_{0}Y_{0}+\dotsb+X_{3}Y_{3}\));
\item{} satisfies the differential equation \(\sum\limits_{i=0}^{3} X_{i} \frac{\partial W}{\partial Y_{i}} \equiv 0 \mod (P,q)\).
\end{itemize}
With some powerful computing machine, one could even find explicitly a solution, i.e. an explicit equation for the surface of bi-tangent lines. If one is optimistic, it might even be possible to extrapolate a general formula for any smooth quartic.

\subsubsection{Examples of non-invariance under deformation of (canonically twisted) symmetric pluri-genera}
In the spirit of Brotbek's examples of family of surfaces in \(\P^{4}\) (see \cite{Brot15}) along which the \(2\)-symmetric pluri-genus (i.e. the dimension of the second symmetric power of the cotangent bundle) does not remain constant, one can consider the \(1\)-parameter family of surfaces in \(\P^{4}\)
\[
X_{t} \bydef \Set{X_{0}^{5}+X_{1}^{5}+\dotsb+X_{4}^{5}=0} \cap \Set{-2X_{0}^{5}-X_{1}^{5} + tX_{0}X_{1}X_{2}X_{3}X_{4}+X_{3}^{5}+2X_{4}^{5}=0}.
\]
In Appendix \ref{appendix: sage}, explicit computations show that \(h^{0}(S^{2}\Omega_{X_{0}})=1\), whereas \(h^{0}(S^{2}\Omega_{X_{1}})=0\) (which, by upper semi-continuity of cohomology, implies that a general member of the family has no global \(2\)-symmetric differential).

It was asked by Paun whether or not the invariance under deformation holds if one considers instead symmetric powers twisted by the canonical bundle. It was answered negatively in \cite{Bog}. We propose to give an explicit and very simple example of a one parameter family \((X_{t})_{t \in \C}\) of surfaces in \(\P^{4}\) along which the invariance of \(S^{6}\Omega_{X_{t}}(K_{X_{t}})\) does not hold:
\begin{theorem}
\label{thm: non-invariance}
The \(1\)-parameter family
\[
X_{t}
\bydef
\Set{X_{0}^{4}+X_{1}^{4}+\dotsb+X_{4}^{4}-tX_{0}^{2}X_{4}^{2}=0} \cap \Set{-2X_{0}^{4}-X_{1}^{4}+X_{3}^{4}+2X_{4}^{4}=0} \subset \P^{4}
\]
satisfies \(h^{0}(S^{6}\Omega_{X_{0}}(K_{X_{0}}))=1\), whereas \(h^{0}(S^{6}\Omega_{X_{0}}(K_{X_{t}}))=0\) for a general \(t \in \C\).
\end{theorem}
\begin{proof}
By the adjunction formula, for any \(t \in \C\), one has the equality
\[
S^{6}\Omega_{X_{t}}(K_{X_{t}})
=
S^{6}\Omega_{X_{t}}(3).
\]
In Appendix \ref{appendix: sage}, explicit computations show that \(h^{0}(S^{6}\Omega_{X_{0}}(3))=1\), whereas \(h^{0}(S^{6}\Omega_{X_{1}}(3))=0\). By semi-continuity of cohomology, this proves the statement.
\end{proof}

\subsection{Fifth application: a criteria ensuring ampleness of cotangent bundles of complete intersection surfaces in \(\P^{4}\).}
\label{subs: criteria ampleness}
This section is not devoted to proving anything new (as the scheme of proof for the criteria given in Lemma \ref{lemma: criteria ample} is classical), but rather to discuss strategies to provide concrete examples of surfaces with ample cotangent bundles (which is still an open question). 



\subsubsection{Discussion on Brotbek's work: \textsl{explicit} constructions of symmetric differentials}
In his seminal paper \cite{Brot15}, Brotbek provided an explicit way to describe the space of global sections of complete intersections of codimension at least as large as their dimension. For people very familar with Brobek's work, the description we give in Section \ref{sect: coho ci 1} (as far as the space of global sections is concerned) is reminiscent of Brotbek's description.

The drawback of Brotbek's description, or the one presented in this paper, is that it only allows to tell whether or not there exists global (negatively twisted) symmetric differentials\footnote{Recall that knowing that a given complete intersection, say a smooth surface in \(\P^{4}\), has global symmetric differentials vanishing along an ample is particularly interesting if one wishes to study the positivity of the cotangent bundle.}, but it gives no information on their form (contrarily to the case of global symmetric vector fields: see Section \ref{subs: bi-tangent}\footnote{The specificity in this case is that one has a natural isomorphism between the tangent and cotangent bundle.}).

The crucial input in Brotbek's work, first implemented in \cite{Brot15} to prove particular cases of Debarre's conjecture, then improved in \cite{BD} to prove Debarre's conjecture in full generality, is the following. In the special case of Fermat-type complete intersections, Brotbek's description of cohomology allows to tell that there exists global negatively twisted symmetric differentials. The difficult part consists in understanding the explicit form of these sections: this is exactly what Brotbek did in \cite{Brot15}[Lemma 4.5]. 
Of course, there is still a lot of work to be done from here to obtain Brotbek's results, but we believe that they are mainly of technical nature. 

It is important to mention that, in order to be able to control the base locus of the constructed global sections (which is necessary to obtain ampleness), it is fundamental to consider the whole situation in \textsl{family} (and this is where technicality arises). It is of course always possible to consider a fixed example, but in this case, there is almost no flexibility to work with: one is reduced to proving that the base locus of the set of sections is empty\footnote{zero dimensional is actually enough.}. Theoretical reasons for which this could be true are yet to be found, but one could always try doing this on a computer. It amounts to computing a Groebner basis, which, for the situation under study, seems out of of reach.

\subsubsection{A cohomological criteria ensuring ampleness}
The approach we propose here is transversal to Brotbek's approach: instead of exhibiting enough global sections to derive ampleness, we would like to use a cohomological approach. We are going to give a simple criteria, but as in the previous case, it would most probably require too much computations to conclude.

Let us consider \(\Flag_{(1,2)} \C^{N+1} \simeq \P(T\P^{N})\), and \(\xi \in \Flag_{(1,2)} \C^{N+1}\). The ideal sheaf \(m_{\xi}\) of the closed point \(\xi\) admits a simple resolution:
\begin{lemma}
\label{lemma: resolution of closed point}
The ideal sheaf \(m_{\xi}\) admits a locally free resolution of the following form
\[
0 \to A_{2N-2} \to A_{2N-3} \to \dotsb \to A_{0} \to m_{\xi} \to 0,
\]
where for each \(0 \leq i \leq 2N-2\), the vector bundle \(A_{i}\) writes as a direct sum of line bundles of the following type
\[
\mathcal{L}_{-k,-i} 
\]
for \(0 \leq k \leq i\).
\end{lemma}
\begin{proof}
Up to applying a suitable automorphism of \(\P^{N}\), one can always suppose that 
\[
\xi=\big( \C \mathbi{e}_{0} \subsetneq \C \mathbi{e}_{0} \oplus \C \mathbi{e}_{1} \big),
\]
where \(\mathbi{e}_{0}, \dotsc, \mathbi{e}_{N}\) is the canonical basis of \(\C^{N+1}\). 
Consider the \(N\) global sections of  \(\mathcal{L}_{0,1}\)
\[
\big(s_{i}=X_{i}\big)_{1 \leq i \leq N},
\] 
as well as the \(N-2\) global sections of \(\mathcal{L}_{1,1}\)\
\[
\big(t_{i}=X_{0}Y_{i}-X_{i}Y_{0}\big)_{2 \leq i \leq N}.
\]
Here, one identifies \(H^{0}(\Flag_{(1,2)}\C^{N+1}, \mathcal{L}_{1,1})\) with the set of \((1,1)\) bi-homogeneous polynomials \(P \in \C[Y,X]\) satisfying the functional equation (see Section \ref{subs: geom inter})
\[
\forall \ t \in \C,
P(Y+tX,X)=P(Y,X).
\]

One easily sees that the zero locus of the sections \(s_{1}, \dotsc, s_{N}, t_{2}, \dotsc, t_{N}\) defines the closed point \(\xi\) without multiplicity. Therefore, the Koszul complex associated to these sections provide a locally free resolution of the ideal sheaf \(m_{\xi}\), which takes the form given in the statement.
\end{proof}

With this Lemma \ref{lemma: resolution of closed point} at hand, we can now prove the following criterion:
\begin{lemma}
\label{lemma: criteria ample}
Let \(X \subset \P^{4}\) be smooth surface. Suppose that there exists \(m \geq 2\) such that
\begin{itemize}
\item{}
\(
H^{1}(X, S^{m}\Omega_{X}(-1))=(0)
\);
\item{}
\(
H^{2}(X,  S^{m}\Omega_{X}(-2))=H^{2}(X,S^{m-1}\Omega_{X}(-3))=(0).
\)
\end{itemize}
Then the cotangent bundle \(\Omega_{X}\) is ample.
\end{lemma}
\begin{proof}
Note first that it is enough to prove that there exists \(m \geq 2\) such that \(\O_{\P(TX)}(m) \otimes \pi_{X}^{*}\O_{X}(-1)\) is globally generated. Indeed, in this case, \(\O_{\P(TX)}(m)\) writes as the sum of nef line bundle and a \(\pi_{X}\)-relatively ample line bundle, so that \(\O_{\P(TX)}(m)\), and thus \(\O_{\P(TX)}(1)\), is ample.

Consider \(\xi \in \P(TX)\) a closed point. By Lemma \ref{lemma: resolution of closed point} and Proposition \ref{prop: geom inter}, the ideal sheaf \(m_{\xi}\) admits a resolution of the following form:
\[
0 \to A_{3} \to A_{2} \to A_{1} \to A_{0} \to m_{\xi} \to 0,
\]
where \(A_{i}\) is a direct summand of line bundles of the form
\(
\O_{\P(TX)}(-k)\otimes \pi_{X}^{*}\O_{X}(-k-i),
\)
where \(k\) is allowed to range between \(0\) and \(i\).

Now, tensor the above complex by \(\O_{\P(TX)}(m)\otimes \pi_{X}^{*}\O_{X}(-1)\). By cutting it into \(3\) short exact sequences, taking long exact sequences in cohomology and using:
\begin{itemize}
\item{} the two vanishing hypothesis of the statement;
\item{} the fact that, for any \(0 \leq i \leq 3\), the cohomology of \(A_{i}\otimes \O_{\P(TX)}(m)\otimes \pi_{X}^{*}\O_{X}(-1)\) is supported in degrees less or equal than \(2\) (by Bott's formulas, and the fact that the cohomology of a coherent sheaf on a surface is supported in degree less or equal than \(2\)),
\end{itemize}
one sees that the following vanishing result holds:
\[
H^{1}(\P(TX), \O_{\P(TX)}(m)\otimes \pi_{X}^{*}\O_{X}(-1)\otimes m_{\xi})
=
(0).
\]

In order to conclude, consider the short exact sequence induced by the evaluation map at \(\xi\)
\[
0 \to m_{\xi} \to \O_{\P(TX)} \to \C_{\xi} \to 0,
\]
where \(\C_{\xi}\) is the skyscraper sheaf supported on \(\Set{\xi}\). Twisting this short exact sequence by \(\O_{\P(TX)}(m)\otimes \pi_{X}^{*}\O_{X}(-1)\), taking the long exact sequence in cohomology and using the above vanishing result, one obtains a surjection:
\[
\resizebox{\displaywidth}{!}{
\xymatrix{
H^{0}(\P(TX), \O_{\P(TX)}(m)\otimes \pi_{X}^{*}\O_{X}(-1)) 
\ar@{->>}[r]
&
H^{0}(\P(TX), \O_{\P(TX)}(m)\otimes \pi_{X}^{*}\O_{X}(-1) \otimes \C_{\xi}).
}
}
\]
This shows that there exists a global section of \(\O_{\P(TX)}(m)\otimes \pi_{X}^{*}\O_{X}(-1)\) that does not vanish when evaluated at \(\xi\). As this holds for any closed point \(\xi\), one deduces that the line bundle  \(\O_{\P(TX)}(m)\otimes \pi_{X}^{*}\O_{X}(-1)\) is globally generated. With the remark at beginning, this concludes the proof.
\end{proof}

With this criterion, one only has to check a few vanishing of cohomology groups (three to be precise) in order to conclude that the cotangent bundle is ample. Since we have a completely explicit description of these cohomology groups for any complete intersection surfaces, one could hope to provide explicit examples by pure computations. The issue is that one would need a huge computing capacity in order to do so. This might be within reach provided the bound \(m\) for which the statement of Lemma \ref{lemma: criteria ample} holds is small, say less than \(30\), and provided that the multi-degree \(\mathbi{d}=(d_{1}, d_{2})\) for which there exists complete intersection surfaces with ample cotangent bundle is also small, ideally \(d_{i} \leq 6\). And even in these specific cases, the amount of computation remains very important: being able to find a concrete example via this method might be a mirage.

There is still the purely theoretic approach, and this is mainly the reason why we described completely the complex computing cohomology of twisted symmetric powers of cotangent bundles of complete intersections of codimension \(2\) (see Section \ref{sect: coho ci 2}). This is a subject we are working on. As a first step, one would already need to settle the case of hypersurfaces in full generality: see Section \ref{sect: intermediate ampleness} where we treat a particular case.

\appendix

\section{Reminder on Bott's formulas.}
\label{appendix: Bott}
The reference for this part is \cite{Wey}.

Let \(V\) be a \(\C\)-vector space of dimension \(n\).
For a decreasing sequence of integers \[\mathbi{s}=(s_{1}>s_{2}>\dotsb>s_{t}),\] with \(t \geq 1\), and \(s_{1} \leq n-1\), the \textsl{partial flag variety \(\Flag_{\mathbi{s}}(V)\) of sequence \(\mathbi{s}\)} is defined as
\[
  \Flag_{\mathbi{s}}(V)
  \bydef
  \Set*{
    V \supsetneq F_{1}  \supsetneq \dotsb \supsetneq F_{t} \supsetneq \Set{0}
    \suchthat
    \text{\(F_{i}\) is a \(\kk\)-vector subspace of dimension \(s_{i}\)}
  }.
\]
On \(\Flag_{s}(V)\), there is a filtration of the trivial bundle \(\Flag_{s}(V)\times V\) by universal subbundles \(U_{s_{1}}\supsetneq\dotsb\supsetneq U_{s_{t}}\), given over the point \(\eta=(F_{1}\supsetneq \dotsb \supsetneq F_{t})\in\Flag_{s}V\) by \(U_{s_{i}}(\eta)=F_{i}\).

\begin{example}
  \label{ex: Grass}
  If we take the sequence \(\mathbi{s}=(k)\), where \(k \geq 1\), then \(\Flag_{\mathbi{s}}(V)\) is the Grassmannian of \(k\)-dimensional subspaces of \(V\), denoted \(\Grass(k, V)\), together with the tautological subbundle \(U_{k}\).
\end{example}

A \(n\)-uple of integer \(\mathbi{\alpha} \in \Z^{n}\) is called \textsl{admissible with jump sequence \(\mathbi{s}\)} if and only if it writes (uniquely) as follows:
\[
\mathbi{\alpha}
=
(\underbrace{\alpha_{1}, \dotsc, \alpha_{1}}_{\times s_{1}}, \underbrace{\alpha_{2}, \dotsc, \alpha_{2}}_{\times s_{2}}, \dotsc, \underbrace{\alpha_{t}, \dotsc, \alpha_{t}}_{\times s_{t}}),
\]
where the \(\alpha_{i}\)'s are all distinct.
To any such  \(\mathbi{\alpha}\), one associates the line bundle
\(\L_{\mathbi{\alpha}}(V)\)\footnote{Note that, in the body of the text, we used the abbreviated notation \(\mathcal{L}_{\mathbi{\alpha}}\): this does not lead to any confusion as the underlying vector space remains the same everywhere (namely, \(\C^{N+1}\)).} on \(\Flag_{\mathbi{s}}(V)\) by setting:
\[
  \L_{\mathbi{\alpha}}(V)
  \bydef
  \bigotimes_{i=1}^{t}
  \det((U_{s_{i}}/U_{s_{i+1}})^{\vee})^{\alpha_{s_{i}}},
\]
where by convention \(U_{s_{t}+1}=\Set{0}\).
\begin{example}
  For \(\mathbi{\alpha}=(m, \dotsc, m) \in \N^{n}\), \(\L_{\mathbi{\alpha}}(V)\) is the \(m\)th tensor power of the Plücker line bundle \(\O(1)=\det(U_{k}^{\vee})\) on \(\Grass(k,V)\). This line bundle allows one to embed \(\Grass(k,V)\) in the projective space \(\P(\Ext^{k}V)\).
\end{example}

Let us now take \(E\) a vector bundle of rank \(n\) over a variety \(M\). Let us fix a sequence \(\mathbi{s}=(s_{1}>s_{2}>\dotsb>s_{t})\) and an admissible \(n\)-uple of integers \(\mathbi{\alpha} \in \Z^{n}\) with jump sequence \(s\). The previous considerations make sense fiberwise, and the constructions globalize.
This allows one to define
\begin{itemize}
  \item{} the projective bundle \(\pi\colon\Flag_{\mathbi{s}}(E)\to M\), such that the fiber over \(x \in M\) is \(\Flag_{\mathbi{s}}(E_{x})\);
  \item{} the line bundle \(\L_{\mathbi{\alpha}}(E)\) over \(\Flag_{\mathbi{s}}(E)\), such that \((\L_{\mathbi{\alpha}}(E))_{| \Flag_{\mathbi{s}}(E_{x})}= \L_{\mathbi{\alpha}}(E_{x})\).
\end{itemize}

The object of Bott's formulas is to compute the higher direct image functors 
\[
R^{i}\pi_{*}\mathcal{L}_{\mathbi{\alpha}}(E)
\]
for any admissible \(n\)-uple \(\mathbi{\alpha} \in \Z^{n}\) with jump sequence \(\mathbi{s}\), and for any \(i \geq 0\). In order to state the result properly, let us first make a couple of definitions. First, for sake of simplicity, let us define Schur functors in the following fashion (for more details on Schur functors, see e.g. \cite{Wey} or \cite{Fulton}):
\begin{definition}
Let \(\mathbi{\lambda} \in \Z^{n}\) be an admissible \(n\)-uple with jump sequence \(\mathbi{s}\). Suppose furthermore that \(\mathbi{\lambda}\) is a partition, namely that \(\lambda_{1} \geq \dotsb \geq \lambda_{n} \geq 0\).
 The \textsl{Schur bundle} associated to the partition \(\mathbi{\lambda}\) with jump sequence \(\mathbi{s}\) is the direct image
  \[
    \S^{\mathbi{\lambda}}(E) 
    \bydef 
    \pi_{*}(\L_{\mathbi{\lambda}}(E^{\vee})).
  \]
\end{definition}

Next, let us denote 
\[
\rho=(n-1, n-2, \dotsc, 0) \in \N^{n}.
\]
For any permutation \(\sigma \in \mathcal{S}_{n}\), define the following induced action on \(\Z^{n}\):
\[
\tilde{\sigma}(\cdot)\bydef \sigma(\cdot+\rho)-\rho.
\]
We can now state Bott's formulas\footnote{Almost in its full generality: see \cite{Wey}, where  a larger class of line bundles than the one introduced here is considered.}, where we keep the notations introduced above:
\begin{theorem}[Bott]
Let \(\mathbi{\alpha} \in \Z^{n}\) be an admissible \(n\)-uple with jump sequence \(\mathbi{s}\).
One of the mutually exclusive possibilities occurs.
\begin{enumerate}
\item{} There exists \(\sigma\in \mathcal{S}_{n} \setminus \Set{Id}\) such that
\[
\tilde{\sigma}(\mathbi{\alpha})=\mathbi{\alpha}.
\]
In this case, all higher direct images \(R^{i}\pi_{*}\mathcal{L}_{\mathbi{\alpha}}(E)\) vanish for \(i \geq 0\).
\item{}
There exists a unique \(\sigma \in \mathcal{S}_{n}\) and a unique \(r \in \Z\) such that
\[
\tilde{\sigma}(\mathbi{\alpha})+(r, \dotsc, r)
=
\underbrace{(\beta_{1}, \dotsc, \beta_{n-1}, 0)}_{\bydef \mathbi{\beta}}
\]
is a partition. In this case, all higher direct images \(R^{i}\pi_{*}\mathcal{L}_{\mathbi{\alpha}}(E)\) vanish for \(i \neq \ell(\sigma)\)\footnote{\(\ell(\sigma)\) is the cardinality of the support of the partition \(\sigma\).}, and one has
\[
 R^{\ell(\sigma)}\pi_{*}\mathcal{L}_{\mathbi{\alpha}}(E)
 \simeq
 S^{\mathbi{\beta}}E^{\vee} \otimes \det(E^{\vee})^{-r}.
\]
\end{enumerate}
\end{theorem}

Let us now mention two particular cases where these formulas apply. 

First, let us consider the simplest where the flag variety is the variety of lines. In this situation, one deals with the projectivized bundle \(\P(E) \overset{\pi}{\longrightarrow} M\), and Bott's formula allows to recover the description given in \cite{Harts}[III. Exercice 8.4]. In particular, this says that
\[
R^{i} \pi_{*} \O_{\P(E)}(m)
=
(0)
\]
for any \(i >0\) and any \(m \in \N\).
Second, let us consider the case where \(M\) is a point. In this case, a straightforward corollary of Bott's formula is the following:
\begin{corollary}
\label{cor: Bott}
Let \(\mathbi{\alpha} \in \Z^{n}\) be an admissible \(n\)-uple  with jump sequence \(\mathbi{s}\).
One of the mutually exclusive possibilities occurs.
\begin{enumerate}
\item{} There exists \(\sigma\in \mathcal{S}_{n} \setminus \Set{Id}\) such that
\[
\tilde{\sigma}(\mathbi{\alpha})=\mathbi{\alpha}.
\]
In this case, all cohomology groups of the line bundle \(\L_{\mathbi{\alpha}}(\C^{n}) \to \Flag_{\mathbi{s}}\C^{n}\) are zero.
\item{}
There exists a unique \(\sigma \in \mathcal{S}_{n}\) and a unique \(r \in \Z\) such that
\[
\tilde{\sigma}(\mathbi{\alpha})+(r, \dotsc, r)
=
\underbrace{(\beta_{1}, \dotsc, \beta_{n-1}, 0)}_{\bydef \mathbi{\beta}}
\]
is a partition. In this case, all cohomology groups of \(\L_{\mathbi{\alpha}}(\C^{n})\) vanish for \(i \neq \ell(\sigma)\), and one has
\[
 H^{\ell(\sigma)}\big(\Flag_{\mathbi{s}}\C^{n}, \L_{\mathbi{\alpha}}(\C^{n})\big)
 \simeq
 S^{\mathbi{\beta}}\C^{n}.
\]
\end{enumerate}
\end{corollary}
Let us note that, if one takes \(\mathbi{\alpha}\) to be a partition, then the conclusion of the previous corollary regarding the space of global sections (i.e. the \(H^{0}\)) is due to Boreil-Weil \cite{BW}.

\section{Computations on Sage.}
\label{appendix: sage}
In this section, we provide the outcome of a few programs written on the language Sage. We used the description of the \(H^{0}\) provided by Theorem \ref{thm: coho ci 1} (for codimension \(1\) and codimension \(2\) complete intersections).

\includepdf[pages=-]{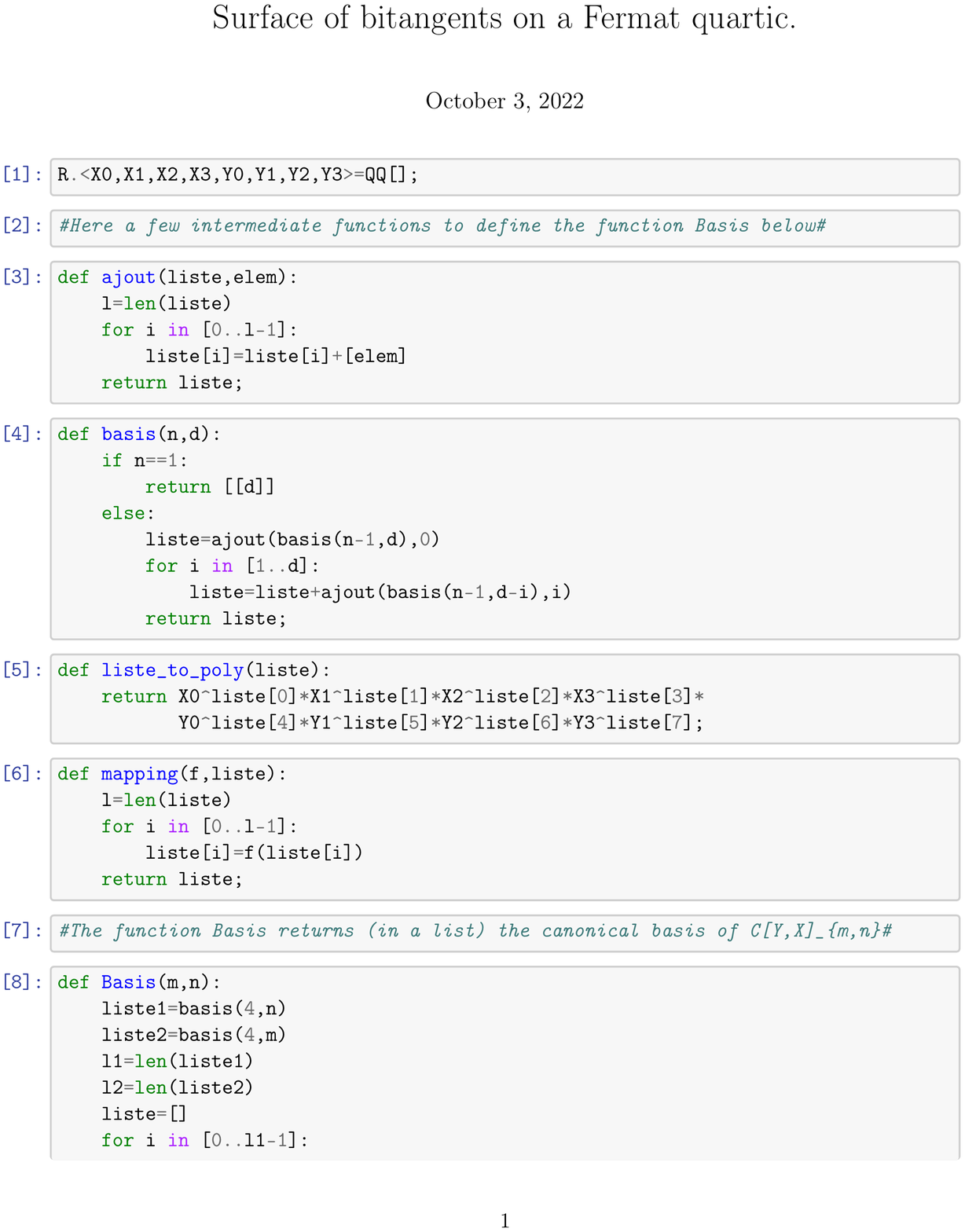}

\includepdf[pages=-]{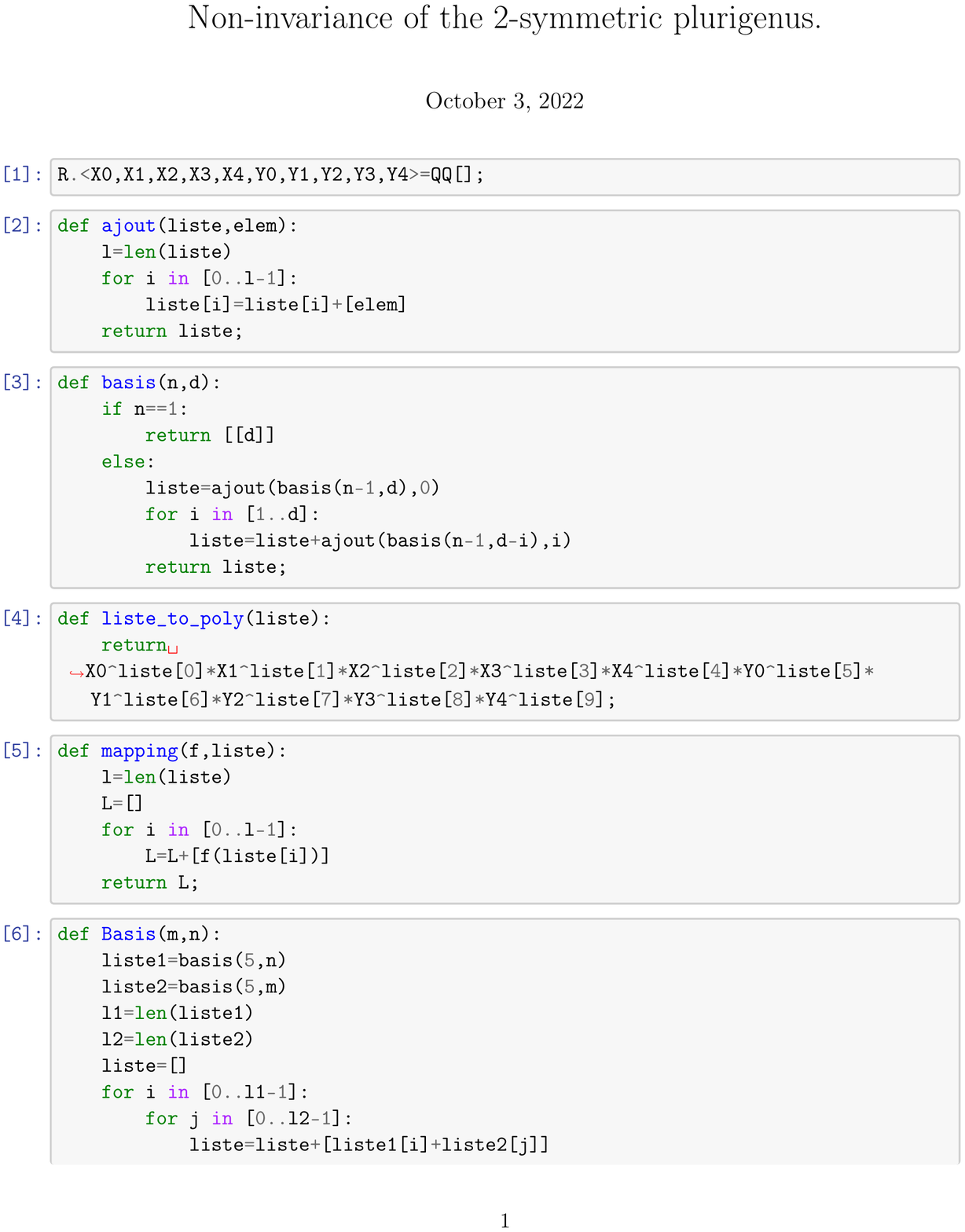}

\includepdf[pages=-]{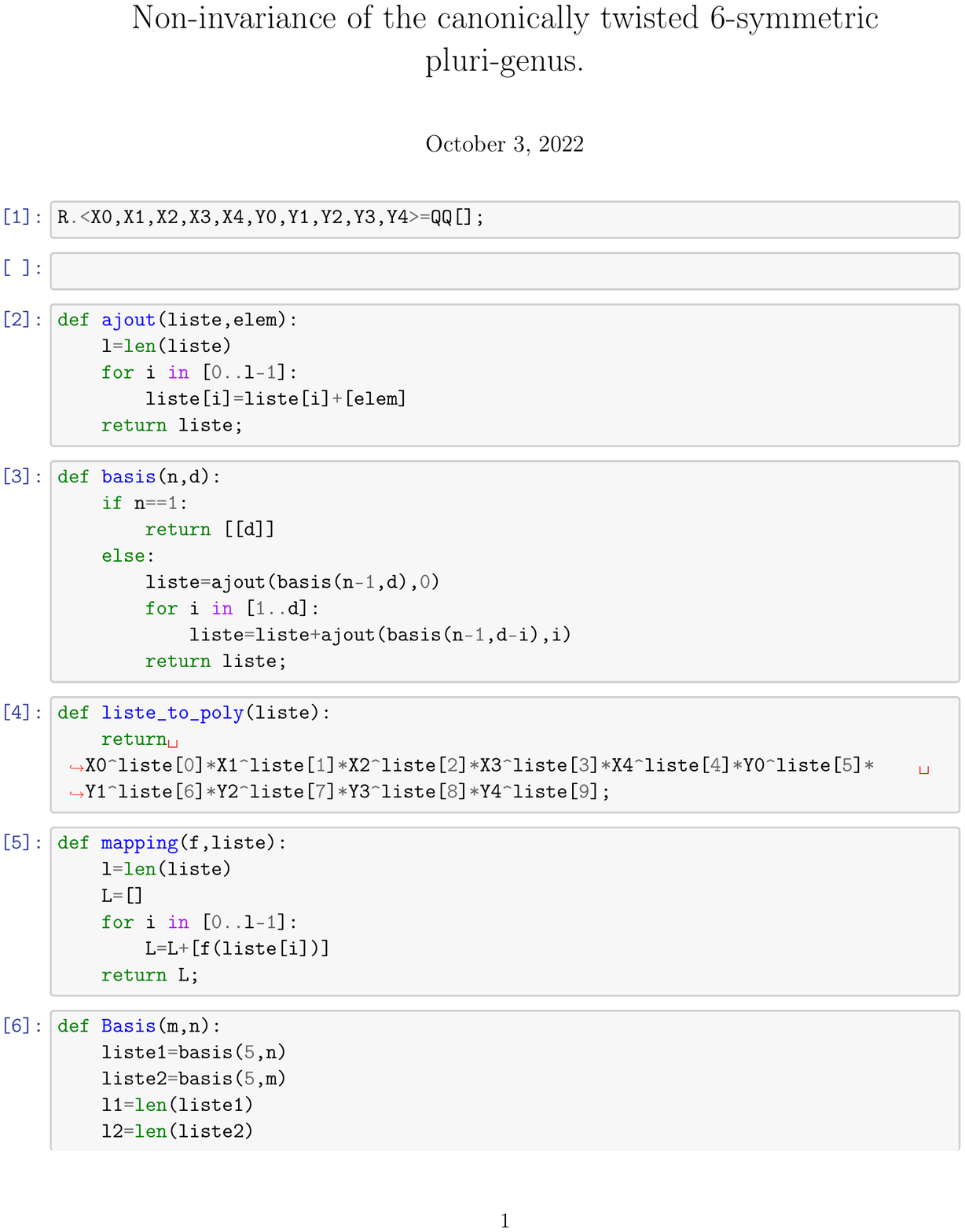}

\bibliographystyle{alpha}
\bibliography{Coho_sym}

 \end{document}